\documentclass[12pt]{amsart}
\usepackage{amssymb}
\usepackage{mathrsfs}
\usepackage{bm}
\usepackage{graphicx}
\usepackage{color}
\usepackage[all]{xy}

\usepackage{hyperref}

\allowdisplaybreaks[4]

\newtheorem{thm}{Theorem}[section]
\newtheorem{lem}[thm]{Lemma}
\newtheorem{cor}[thm]{Corollary}
\newtheorem{prop}[thm]{Proposition}
\newtheorem{definition}[thm]{Definition}

\def \para{\refstepcounter{thm} \par\medskip\noindent
                \textbf{\thethm .} }

\def \remark{\refstepcounter{thm} \par\medskip\noindent
                \textbf{Remark \thethm .} }

\def \remarks{\refstepcounter{thm} \par\medskip\noindent
                \textbf{Remarks \thethm .} }

\def \example{\refstepcounter{thm} \par\medskip\noindent
                \textbf{Example \thethm .} }

\numberwithin{equation}{thm}

\newcommand{\ps}{\vskip .5 \baselineskip} 
\newcommand\BB{\mathbb B}
\newcommand\CC{\mathbb C}
\newcommand\DD{\mathbb D}

\renewcommand\SS{\mathbb S} %######
\newcommand\TT{\mathbb T}

\newcommand\ZZ{\mathbb Z}

\newcommand\bP{\mathbf P}
\newcommand\bQ{\mathbf Q}

\newcommand\bS{\mathbf S}
\newcommand\bT{\mathbf T}

\newcommand\bb{\mathbf b}
\newcommand\bc{\mathbf c}
\newcommand\bd{\mathbf d}

\renewcommand\bm{\mathbf m} %######

\newcommand\bfs{\mathbf s} %!!!!!!!!!!!!
\newcommand\bt{\mathbf t}

\newcommand\cI{\mathcal{I}}

\newcommand\cN{\mathcal{N}}
\newcommand\cO{\mathcal{O}}

\newcommand\cU{\mathcal{U}}

\newcommand\fF{\mathfrak F}

\newcommand\fI{\mathfrak I}

\newcommand\fP{\mathfrak P}
\newcommand\fQ{\mathfrak Q}

\newcommand\fS{\mathfrak S}

\newcommand\fX{\mathfrak X}

\newcommand\fn{\mathfrak n}

\newcommand\sH{\mathscr H}

\newcommand\sS{\mathscr S}

\renewcommand\a{\alpha}  %#####
\renewcommand\b{\beta}   %#####
\newcommand\g{\gamma}  
\renewcommand\d{\delta}  %#####
 
\newcommand\z{\zeta}

\newcommand\la{\lambda}

\newcommand\s{\sigma}
\newcommand\ta{\tau}

\newcommand\w{\omega}

\newcommand\ve{\varepsilon}

\newcommand\vf{\varphi}

\newcommand\D{\Delta}

\renewcommand\Xi{\Xi}

\newcommand\Om{\Omega}

\newcommand\vG{\varGamma}

\newcommand\vL{\varLambda}

\newcommand\Bb{\boldsymbol\beta}  
\newcommand\Bg{\boldsymbol\gamma}

\newcommand\Bz{\boldsymbol\zeta}
\newcommand\Be{\boldsymbol\eta}

\newcommand\Bla{\boldsymbol\lambda}
\newcommand\Bm{\boldsymbol\mu}

\newcommand{\dis}{\displaystyle}

\newcommand\ti{\tilde}
\newcommand\wh{\widehat}
\newcommand\wt{\widetilde}
\newcommand\ol{\overline}

\newcommand\ra{\rightarrow}

\newcommand{\da}{\downarrow}

\newcommand\trr{\triangleright }

\newcommand\trreq{\trianglerighteq}

\newcommand\lan{\langle}
\newcommand\ran{\rangle}

\newcommand\End{\operatorname{End}}

\newcommand\Res{\operatorname{Res}}

\renewcommand\Im{\operatorname{Im}}

\newcommand\Std{\operatorname{Std}}

\newcommand\opp{\operatorname{opp}}

\newcommand\cmod{\operatorname{-mod}}

\newcommand{\owt}{\operatorname{wt}}
\newcommand{\lwt}{\operatorname{\ell-wt}}
\newcommand{\lbwt}{\operatorname{\ell_{[\bb]}-wt}}

\newcommand\Fgl{\mathfrak{gl}}
\newcommand\Fsl{\mathfrak{sl}}

\newcommand\ev{\mathbf{{ev}}}
\newcommand{\SStd}{\operatorname{SStd}} 
\newcommand{\Tab}{\operatorname{Tab}} 
\newcommand{\Sing}{\operatorname{Sing}} 
\newcommand{\fSing}{\operatorname{f-Sing}} 

\newcommand{\shape}{\operatorname{shape}}

\newcommand{\isom}{\,\raise2pt\hbox{$\underrightarrow{\sim}$}\,}

\newcounter{ichi}
\newcommand{\roi}{\roman{ichi}}
\newcounter{ni}
\newcommand{\roii}{\roman{ni}}
\newcounter{san}
\newcommand{\roiii}{\roman{san}}
\newcounter{yon}
\newcommand{\roiv}{\roman{yon}}
\newcounter{go}
\newcommand{\rov}{\roman{go}}
\newcounter{roku}
\newcommand{\rovi}{\roman{roku}}
\newcounter{nana}
\newcommand{\rovii}{\roman{nana}}
\newcounter{hachi}
\newcounter{kyu}
\begin{document}

\setlength{\baselineskip}{4.9mm}
\setlength{\abovedisplayskip}{4.5mm}
\setlength{\belowdisplayskip}{4.5mm}

%%%%%%%%%%%%%%%%%%%%%%%%%%%%%%%%%%%%%%%%%%%%

\renewcommand{\theenumi}{\roman{enumi}}
\renewcommand{\labelenumi}{(\theenumi)}
\renewcommand{\thefootnote}{\fnsymbol{footnote}}
\renewcommand{\thefootnote}{\fnsymbol{footnote}}
\parindent=20pt

%%%%%%%%%%%%%%%%%%%%%%%%%%%%%%%%%%%%%%%%%%%%

\setcounter{section}{-1}

%%%%%%%%%%%%%%%%%%%%%%%%%%%%%%%%%%%%%%%%%%%%

%%%%%%%%%%%%%%%%%%%%%%%%%%%%%%%%%%%%%%%%%%%%%%%%%%%%%%%%%%%%%%%

%%%%%%%%%%%%%%%%%%%%%%%%%%%%%%%%%%%%%%%%%%%%%%%%%%%%%%%%%%%%%%%

\address{R. Kodera : Department of Mathematics and Informatics, Graduate School of Science, 
Chiba University}
\email{kodera@math.s.chiba-u.ac.jp}
		
\address{K. Wada : Department of Mathematics, Faculty of Science, Shinshu University, 
		Asahi 3-1-1, Matsumoto 390-8621, Japan}
		
\email{kwada@shinshu-u.ac.jp}

%%%%%%%%%%%%%%%%%%%%%%%%%%%%%%%%%%%%%%%%%%%%%%%%%%%%%%%%%%%%%%%

%%%%%%%%%%%%%%%%%%%%%%%%%%%%%%%%%%%%%%%%%%%%%%%%%%%%%%%%%%%%%%%

\medskip
\begin{center}
{\large \textbf{  Some finite dimensional representations of shifted quantum affine algebras of type $A$ }}  
\\
\vspace{1cm}
Ryosuke Kodera and Kentaro Wada  
\\[1em]
\end{center}

%%%%%%%%%%%%%%%%%%%%%%%%%%%%%%%%%%%%

\title{} 
%\author{Kentaro Wada} 
%\date{}
\maketitle 

\markboth{R. Kodera and K. Wada}{  Some finite dimensional representations of shifted quantum affine algebras of type $A$}

%%%%%%%%%%%%%%%%%%%%%%%%%%%%%%%%%%%%%%%%%%%%%%%%%%%%%%%%%%%%%%%

%%%%%%%%%%%%%%%%%%%%%%%%%%%%%%%%%%%%%%%%%%%%%%%%%%%%%%%%%%%%%%%
\begin{abstract}
In this paper, we study finite dimensional representations of shifted quantum affine algebras of type $A$. 
We give an explicit description of the tensor product of simple evaluation modules of the quantum loop algebra 
and a one-dimensional representation of the shifted quantum affine algebra under the separation condition. 
As a consequence,  we give the $q$-characters of some finite dimensional simple modules of the shifted quantum affine algebra.  
\end{abstract}

\tableofcontents 
%%%%%%%%%%%%%%%%%%%%%%%%%%%%%%%%%%%%%%%%%%%%%%%%%%%%%%%%%%%%%%%

%%%%%%%%%%%%%%%%%%%%%%%%%%%%%%%%%%%%%%%%%%%%%%%%%%%%%%%%%%%%%%%
\section{Introduction} 

\para 
The shifted quantum affine algebras were introduced in \cite{FT19} 
which were related to the study of the quantized K-theoretic Coulomb branches of $3d \, \cN= 4$ SUSY quiver gauge theories. 
The representation theory of the shifted quantum affine algebra was studied in \cite{Her23}.  

In this paper, we study finite dimensional representations of the shifted quantum affine algebra 
associated with the special linear Lie algebra $\Fsl_m$. 
The shifted quantum affine algebra $U_{q,[\bb]}=U_{q,[\bb]} (L\Fsl_m)$ of the shift $\bb=(b_i)_{1 \leq i \leq m-1} \in \ZZ^{m-1}$ 
is defined as a certain shift of the quantum loop algebra $U_q=U_q (L \Fsl_m)$ 
(see Definition \ref{def SQA} and Remark \ref{rem SQA} for notation).  
By \cite{FT19}, it is known that $U_{q,[\bb]}$ is regarded as a right (resp. left) coideal subalgebra of $U_q $ 
through the shift homomorphism if the shift $\bb$ is dominant (i.e. $\bb \in (\ZZ_{\geq 0})^{m-1}$).
On the other hand, in \cite{Her23}, it is proven that $U_{q,[\bb]}$ admits non-zero finite dimensional representations 
if and only if the shift $\bb$ is dominant. 
Then, we assume that the shift $\bb$ is dominant through this paper. 
We study representations of the quantum loop algebra $U_q$ and also of the shifted quantum affine algebra 
$U_{q,[\bb]}$ under the Drinfeld-Jimbo coproduct of $U_q $ 
although the deformed Drinfeld coproduct of $U_q $ is used in \cite{Her23}. 

%%%%%
\para 
One of the differences between the representations of $U_{q}$ and of $U_{q,[\bb]}$ 
appears in one-dimensional representations. 
The quantum loop algebra $U_q $ has only a few one-dimensional representations (almost \lq\lq trivial" representations), 
and its tensor product with a finite dimensional $U_q $-module gives only a change of type. 
On the other hand, the shifted quantum affine algebra $U_{q,[\bb]}$ has a large number of one-dimensional representations, 
and its tensor product with  a finite dimensional $U_q $-module has a non-trivial structure. 
For example, it does not preserve irreducibility in general. 
Furthermore, we can obtain any finite dimensional simple $U_{q,[\bb]} $-module 
as a subquotient of the tensor product of a finite dimensional simple $U_q$-module 
and a one-dimensional $U_{q,[\bb]}$-module thanks to the argument in \cite{Her23}. 

%%%%%%
\para 
A classification of finite dimensional simple $U_{q,[\bb]}$-modules is given in \cite{Her23} 
which is  a generalization of the classification for the quantum loop algebra $U_q$ in \cite{CP91} and \cite{CP95}. 
We recall some of them. 
Put 
\begin{align*}
\DD = \{ \bP = (P_i(z))_{1 \leq i \leq m-1}  \in \CC [z]^{m-1} \mid  P_i(0) =1\}. 
\end{align*}
The polynomials $P_i(z)$ ($1 \leq i \leq m-1$)  are called Drinfeld polynomials. 
Then, the isomorphism classes of finite dimensional simple $U_q $-modules of type $1$  
are indexed by $\DD$. 
We denote the simple module corresponding to $\bP \in \DD$ by $L(\bP)$ 
(see \ref{par simple Uq}). 

For the dominant shift $\bb =(b_i)_{1 \leq i \leq m-1} \in (\ZZ_{\geq 0})^{m-1}$, 
we consider the set 
\begin{align*}
\SS_{[\bb]} = \{ \bS = (S_i(z))_{1 \leq i \leq m-1} \in \CC [z]^{m-1} \mid \deg S_i(z) = b_i, \, S_i(0) \not=0\}. 
\end{align*}
Then, the isomorphism classes of one-dimensional $U_{q,[\bb]}$-modules are indexed by $\SS_{[\bb]}$. 
Moreover, any finite dimensional simple $U_{q,[\bb]}$-module is isomorphic to the module 
$L(\bS, \bP)$ determined by some $(\bS, \bP) \in \SS_{[\bb]} \times \DD$ (see \ref{simple Uqb}). 
We also call each polynomials $S_i(z)$ and $P_i(z)$ the Drinfeld  polynomials. 
We note that, 
for a finite dimensional simple $U_{q,[\bb]}$-module $L$, 
the Drinfeld polynomials $(\bS, \bP) \in \SS_{[\bb]} \times \DD$ such that $L \cong L(\bS, \bP)$ 
are not unique in general (see Remark \ref{Rem DP}).  

%%%%%%
\para 
The quantum loop algebra $U_q=U_q (L\Fsl_m)$ has the evaluation homomorphism 
$\ev_{\z} : U_q(L\Fsl_m) \ra U_q (\Fgl_m)$ for $\z \in \CC^{\times}$ 
given in \cite{Jim86}. 
For a partition $\la$ such that $\ell (\la) \leq m$, 
let $V(\la)$ be the simple highest weight $U_q(\Fgl_m)$-module of highest weight $\la$. 
We regard $V(\la)$ as a $U_q (L\Fsl_m)$-module through the evaluation homomorphism $\ev_{\z}$, 
and we denote it by $V(\la)^{\ev_\z}$. 
We call $V(\la)^{\ev_\z}$ the evaluation module of $V(\la)$ at the evaluation value $\z$. 
Then, we are interested in the $U_q$-module structure of the tensor product 
$V(\la^{(1)})^{\ev_{\z_1}} \otimes \dots \otimes V(\la^{(r)})^{\ev_{\z_r}}$ of evaluation modules. 

It is well known that the basis of  $V(\la)$ 
is indexed by the set of semi-standard tableaux of shape $\la$. 
Then, the tensor product $V(\la^{(1)})^{\ev_{\z_1}} \otimes \dots \otimes V(\la^{(r)})^{\ev_{\z_r}}$ 
has a basis indexed by the set of $r$-tuple of semi-standard tableaux of corresponding shapes. 
We denote such index set by $\SStd_m(\Bla)$, 
where $\Bla = (\la^{(1)}, \dots, \la^{(r)})$ 
(see \S \ref{comb} for combinatorial definitions).  
For evaluation values $\Bz=(\z_1,\dots, \z_r)$, 
we consider  
the separation condition 
\begin{align*}
\prod_{1 \leq i < j \leq r} \prod_{- n < k <n} (q^{2k} \z_i - \z_j) \not=0, 
\end{align*}  
where $n$ is the sum of the sizes of each partition $\la^{(k)}$ ($1 \leq k \leq r$). 
In this paper, 
we define a $U_q$-module $\D_{\Bz} (\Bla)$ such that 
\begin{itemize}
\item $\D_{\Bz} (\Bla)$ has a basis $\{ v_{\bT} \mid \bT \in \SStd_m (\Bla) \}$, 

\item the action of $U_q$ on $\D_{\Bz}(\Bla)$ is described explicitly by using combinatorics 
\end{itemize}
in Proposition \ref{Prop Def DBz Bla} 
under the separation condition. 
Then, in Theorem \ref{Thm DBz Bla}, we give the following statements under the separation condition.
\begin{itemize}
\item $\D_{\Bz} (\Bla) \cong V(\la^{(1)})^{\ev_{\z_1}} \otimes \dots \otimes V(\la^{(r)})^{\ev_{\z_r}}$ 
	as $U_q$-modules. 

\item Any $\ell$-weight of $\D_{\Bz}(\Bla)$ is multiplicity-free. 

\item The $U_q$-module $\D_{\Bz} (\Bla)$ is simple. 

\item Explicit description of the Drinfeld polynomials and $q$-character of $\D_{\Bz} (\Bla)$ 
	using combinatorics. 
\end{itemize}

We remark that the $q$-character of evaluation module $V(\la)^{\ev_\z}$ 
has been already computed in \cite{FM02}, 
and we already know the $q$-characters of the tensor products of evaluation modules. 
We have also known the simplicity of the tensor product $V(\la^{(1)})^{\ev_{\z_1}} \otimes \dots \otimes V(\la^{(r)})^{\ev_{\z_r}}$ 
under the separation condition 
by \cite{MTZ04} and \cite{Her10}. 
However, in Theorem \ref{Thm DBz Bla}, 
we can directly obtain such results (simplicity and the $q$-character of $\D_{\Bz} (\Bla)$) 
from the explicit description of $\D_{\Bz} (\Bla)$ in  Proposition \ref{Prop Def DBz Bla} 
without using any previous results. 
We also remark that a similar explicit description of the evaluation module $V(\la)^{\ev_\z}$ 
for the Yangian $Y(\Fgl_m)$ (resp. the quantum Yangian $Y_q (\Fgl_m)$) is given by using the Gelfand-Zetlin basis 
in \cite{Mol94} and \cite{NT94}.  
Furthermore, an explicit description of the tensor product of the evaluation modules for the Yangian $Y(\Fgl_m)$ is also given 
in \cite{NT98}. 
On the other hand, we prove Proposition \ref{Prop Def DBz Bla} by directly checking the defining relations, 
and the argument there differs from ones in the above papers for the Yangian $Y(\Fgl_m)$.  
We use the explicit description of $\D_{\Bz} (\Bla)$ 
when we study the tensor product with one-dimensional $U_{q,[\bb]}$-module later. 　

%%%%%
\para 
Let $\bb =(b_i)_{1 \leq i \leq m-1} \in (\ZZ_{\geq 0})^{m-1}$ be a dominant shift, 
and we assume that $\Bz=(\z_1,\dots, \z_r) \in (\CC^{\times})^r$ satisfy the separation condition. 
We can regard $U_{q,[\bb]}$ as a right (resp. left) coideal subalgebra of $U_q$ through the shift homomorphism. 
In particular, we have the algebra homomorphism $\D_{\bb,0} : U_{q,[\bb]} \ra U_{q,[\bb]} \otimes U_q$ 
(resp. $\D_{0, \bb} : U_{q,[\bb]} \ra U_q \otimes U_{q,[\bb]}$).  
Let $L_{\Bb} = \CC w$ be the  one-dimensional $U_{q,[\bb]}$-module 
corresponding to $\bS =(S_i(z))_{1 \leq i \leq m-1} \in \SS_{[\bb]}$,  
where $\Bb= (\b_{i, - b_i+s})_{1 \leq i \leq m-1}^{0 \leq s \leq b_i}$ is determined by 
$S_i(z) = \sum_{s=0}^{b_i} \b_{i, - b_i+s} z^s$ 
(see \eqref{def LBb} for the definition of $L_{\Bb}$). 
Then, we study the $U_{q,[\bb]}$-module 
$L_{\Bb} \otimes \D_{\Bz} (\Bla)$ (resp. $\D_{\Bz}(\Bla) \otimes L_{\Bb}$) 
through the algebra homomorphism $\D_{\bb,0}$ (resp. $\D_{0, \bb}$). 

Since a finite dimensional simple $U_{q,[\bb]}$-module is both $\ell_{[\bb]}$-highest weight module and $\ell_{[\bb]}$-lowest weight module, 
it is important to obtain singular vectors and $f$-singular vectors of 
$L_{\Bb} \otimes \D_{\Bz} (\Bla)$ (resp. $\D_{\Bz}(\Bla) \otimes L_{\Bb}$) 
(see \ref{ellbb hw} for such notions). 

In order to describe the $U_{q,[\bb]}$-module structure of $L_{\Bb} \otimes \D_{\Bz} (\Bla)$  (resp. $\D_{\Bz} (\Bla) \otimes L_{\Bb}$), 
we prepare some combinatorial notation (see \ref{bTi}, \ref{cBz} and \ref{Sing Bla} for details).  
For an $r$-tuple of partitions $\Bla=(\la^{(1)}, \dots, \la^{(r)})$,  
we denote the diagram of $\Bla$ by $[\Bla]$, 
namely we consider the set 
\begin{align*}
[\Bla]= \{(a,b,c) \in (\ZZ_{>0})^3 \mid 1 \leq b \leq \la_a^{(c)}, \, 1 \leq c \leq r\}.
\end{align*}
The elements of $[\Bla]$ are called nodes. 
Then, a tableau $\bT$ of shape $\Bla$ is a map $\bT : [\Bla] \ra \ZZ_{>0}$, 
and the notion of semi-standard tableau is defined in the usual manner. 
We note that the basis of $\D_{\Bz} (\Bla)$ is indexed by $\SStd_m(\Bla)$ 
which is the set of semi-standard tableaux of shape $\Bla$ such that the entries are less than or equal to $m$. 
For a tableau $\bT$ of shape $\Bla$ and $x \in [\Bla]$, 
we define the tableaux $\bT_x^+$ and $\bT_x^-$ by 
\begin{align*}
\bT_x^{\pm} (y) = \begin{cases} \bT(x) \pm 1 & \text{ if } y=x, \\ \bT(y) & \text{ if } y \not=x \end{cases} 
	\text{ for } y \in [\Bla], 
\end{align*}
where we do not consider the tableau $\bT_x^-$ if $\bT(x) =1$. 
For $\bT \in \SStd_m(\Bla)$, 
put 
\begin{align*}
&[\bT]_i =\{x \in [\Bla] \mid \bT(x) =i\} 
\quad (1 \leq i \leq m), 
\\
&[\bT]_{\a_i} = \{x \in [\bT]_{i+1} \mid \bT_x^- \in \SStd_m (\Bla) \} 
	\quad (1 \leq i \leq m-1), 
\\
& [\bT]_{-\a_i} = \{ x \in [\bT]_i \mid \bT_x^+ \in \SStd_m(\Bla) \} 
	\quad (1 \leq i \leq m-1), 
\\
& \SStd_m (\Bla ; \leq \bT) = \{\bS \in \SStd_m(\Bla) \mid \bS(x) \leq \bT(x) \text{ for all } x \in [\Bla]\}, 
\\
& \SStd_m (\Bla ; \geq \bT) = \{\bS \in \SStd_m(\Bla) \mid \bS(x) \geq \bT(x) \text{ for all } x \in [\Bla]\}. 
\end{align*}
For $\bS=(S_i(z))_{1 \leq i \leq m-1} \in \SS_{[\bb]}$ with $S_i(z) = \sum_{s=0}^{b_i} \b_{i, - b_i +s} z^s$ ($1 \leq i \leq m-1$), 
put 
\begin{align*}
&\Sing_m^{\Bb} (\Bla) 
\\
&= \{ \bT\in \SStd_m(\Bla) \mid 
	\prod_{x \in [\bT]_{\a_i}}  (1 - q^i c_{\Bz} (x) z)  \text{ divides } S_i(z)  \text{ for each }  1 \leq i \leq m-1  \}, 
\\
&\fSing_m^{\Bb} (\Bla) 
\\
&= \{ \bT\in \SStd_m(\Bla) \mid 
	\prod_{x \in [\bT]_{-\a_i}} (1 - q^i c_{\Bz} (x) z)  \text{ divides } S_i(z)  \text{ for each }  1 \leq i \leq m-1 \}, 
\end{align*}
where $c_{\Bz}(x)$ is the content of $x$ with respect to $\Bz$, 
namely we have 
$c_{\Bz} (x) = q^{2 (b-a)} \z_{c}$ when $x=(a,b,c)$. 
Then, 
we obtain the following statements for $U_{q,[\bb]}$-module $L_{\Bb} \otimes \D_{\Bz}(\Bla)$ 
(resp. $\D_{\Bz} (\Bla) \otimes L_{\Bb}$) 
in Theorem \ref{Thm simple LBb Dla} (resp. in Theorem \ref{Thm D Bla otimes LBb}). 
\begin{itemize}
\item 
Any $\ell_{[\bb]}$-weight of $L_{\Bb} \otimes \D_{\Bz} (\Bla)$ (resp. $\D_{\Bz} (\Bla) \otimes L_{\Bb}$) 
is multiplicity-free. 

\item 
- The singular vector of $L_{\Bb} \otimes \D_{\Bz}(\Bla)$ is unique up to scalar multiples.  
\\
- The $f$-singular vector of $\D_{\Bz} (\Bla) \otimes L_{\Bb}$ is unique up to scalar multiples. 
\\
(This statement follows from Lemma \ref{Lemma sing LBb M}.)  

\item 
- All $f$-singular vectors of $L_{\Bb} \otimes \D_{\Bz} (\Bla)$ 
	are given by $\{ w \otimes v_{\bT} \mid \bT \in \fSing_m^{\Bb} (\Bla)\}$ 
	up to scalar multiples.  
\\
- All singular vector of $\D_{\Bz}(\Bla) \otimes L_{\Bb}$ 
	are given by $\{ v_{\bT} \otimes w \mid \bT \in \Sing_m^{\Bb} (\Bla)\}$ 
	up to scalar multiples.
	
\item 
- For $\bT \in \fSing_m^{\Bb} (\Bla)$, 
	the subspace $L_{\Bb} \otimes \D_{\Bz} (\Bla ; \leq \bT)$ of $L_{\Bb} \otimes \D_{\Bz} (\Bla)$ 
	spanned by $\{w \otimes v_{\bS} \mid \bS \in \SStd_m (\Bla ; \leq \bT)\}$ 
	is a $U_{q,[\bb]}$-submodule of $L_{\Bb} \otimes \D_{\Bz} (\Bla)$. 
\\
- For $\bT \in \Sing_m^{\Bb} (\Bla)$, 
	the subspace $\D_{\Bz} (\Bla ; \geq \bT) \otimes L_{\Bb}$ of $\D_{\Bz} (\Bla) \otimes L_{\Bb}$ 
	spanned by $\{v_{\bS} \otimes w \mid \bS \in \SStd_m (\Bla ; \geq \bT)\}$ 
	is a $U_{q,[\bb]}$-submodule of $\D_{\Bz} (\Bla) \otimes L_{\Bb}$. 

\item 
- Any $U_{q,[\bb]}$-submodule of $L_{\Bb} \otimes \D_{\Bz} (\Bla)$  coincides with 
	 $L_{\Bb} \otimes \D_{\Bz} (\Bla ; \leq \bT)$ for some $\bT \in \fSing_m^{\Bb} (\Bla)$. 
\\
- Any $U_{q,[\bb]}$-submodule of $\D_{\Bz} (\Bla) \otimes L_{\Bb}$ coincides with 
	$\D_{\Bz} (\Bla ; \geq \bT) \otimes L_{\Bb}$ for some $\bT \in \Sing_m^{\Bb} (\Bla)$. 

\item 
- There exists a unique semi-standard tableau $\bT_{\Bla}^{\Bb} \in \fSing_m^{\Bb} (\Bla)$ 
	such that the $U_{q,[\bb]}$-submodule $L_{\Bb} \otimes \D_{\Bz} (\Bla ; \leq \bT_{\Bla}^{\Bb})$ is simple. 
\\
- There exists a unique semi-standard tableau $\bT^{\Bla}_{\Bb} \in \Sing_m^{\Bb} (\Bla)$ 
	such that the $U_{q,[\bb]}$-submodule $\D_{\Bz} (\Bla ; \geq \bT^{\Bla}_{\Bb}) \otimes L_{\Bb}$ is simple.
\item 
We describe the Drinfeld polynomials and $q$-character of the simple $U_{q,[\bb]}$-module 
	$L_{\Bb} \otimes \D_{\Bz} (\Bla ; \leq \bT_{\Bla}^{\Bb})$ 
	(resp.  the simple $U_{q,[\bb]}$-module  $\D_{\Bz} (\Bla ; \geq \bT^{\Bla}_{\Bb}) \otimes L_{\Bb}$). 
\end{itemize}
In particular, we obtain the explicit description and $q$-character of the finite dimensional simple $U_{q,[\bb]}$-module 
whose Drinfeld polynomials are $(\bS, \bP =(P_i(z))_{1 \leq i \leq m-1})$ such that 
$\bS \in \SS_{[\bb]}$ and 
\begin{align*}
P_i(z) = \prod_{c=1}^r \prod_{k=1}^{\la_i^{(c)} - \la_{i+1}^{(c)}} (1 - q^{2 \la_i^{(c)} - i +1 - 2k} \z_{c} z) 
	\quad (1 \leq i \leq m-1) 
\end{align*}
for some $r$-tuple of partitions $\Bla=(\la^{(1)}, \dots, \la^{(r)})$ 
and $\Bz =(\z_1,\dots, \z_r)$ satisfying the separation condition.

%%%%%%
\para
Finally, 
we study cell modules of cyclotomic $q$-Schur algebras 
via Schur-Weyl dualities established in \cite{Wad24}. 
By \cite{DJM98}, 
it is known that the cyclotomic $q$-Schur algebra $\sS_{n,r}(\bm)$ is a cellular algebra in the sense of \cite{GL96}. 
Then, cell modules of $\sS_{n,r}(\bm)$ have central role in the representation theory of $\sS_{n,r}(\bm)$. 
On the other hand, we can regard $\sS_{n,r}(\bm)$-modules as $U_{q,[\bb_{\bm}]}$-modules 
through the Schur-Weyl duality given in \cite{Wad24}, 
where $\bb_{\bm}$ is a specific shift determined by $\bm=(m_1,\dots, m_r) \in (\ZZ_{>0})^r$ 
(see \ref{uqnnm Scnr} for details). 

In Theorem \ref{Thm qcha cell}, 
we give the $q$-character of  a cell module of $\sS_{n,r}(\bm)$, 
and also give Drinfeld polynomials of the unique simple quotients of the cell module 
when we regard the cell module as a $U_{q, [\bb_{\bm}]}$-module through the Schur-Weyl duality. 
We also give several realizations of the cell module as  a subquotient of the tensor product of a one-dimensional 
$U_{q,[\bb_{[\bm]}]}$-module and evaluation modules in Proposition \ref{Prop WBla D Bla} and Proposition \ref{Prop WBla D wh Bla}. 

The representation theory of the cyclotomic $q$-Schur algebra is related to the parabolic category $\cO$ of 
the affine Lie algebra $\wh{\Fgl}_N$, and also related to the category $\cO$ of the rational Cherednik algebra of typr $G(r,1,n)$ 
by \cite{RSVV16} and \cite{Los16}. 
Therefore, the above results on cell modules may be useful in future works. 
\\

{\bf Acknowledgements:} 
The first author was supported by JSPS KAKENHI Grant Number JP21K03155 and JP25K06916. 
The second author was supported by JSPS KAKENHI Grant Number JP21K03178. 
This work was supported by the Research Institute for Mathematical Sciences, 
an International Joint Usage/Research Center located in Kyoto University.

%%%%%%%%%%%%%%%%%%%%%%%%%%%%%%%%%%%%%%%%%%%%%%%%%%%%%%%%%%%%%%%

%%%%%%%%%%%%%%%%%%%%%%%%%%%%%%%%%%%%%%%%%%%%%%%%%%%%%%%%%%%%%%%
\section{Shifted quantum affine algebras}

Throughout this paper, we assume that $q \in \CC^{\times}$ is not a root of unity. 
\begin{definition}[{\cite[\S 5.1]{FT19}}] 
\label{def SQA} 
For $\bb =(b_1,b_2,\dots, b_{m-1}) \in \ZZ^{m-1}$,  
the shifted quantum affine algebra $U_{q,[\bb]} = U_{q,[\bb]}(L\Fsl_m)$ 
associated with the special linear Lie algebra $\Fsl_m$ 
is the associative algebra over $\CC$ defined by the following generators and defining relations: 
\begin{description}
\item[generators]  
$e_{i,t}, \, f_{i,t}, \, \psi_{i, -b_i +s}^+, \, (\psi^+_{i, - b_i})^{-1}, \, \psi_{i, -s}^-, \, (\psi_{i,0}^-)^{-1} $ 
($1\leq i \leq m-1$, $t \in \ZZ$, $s \in \ZZ_{\geq 0}$) 

\item[defining relations]  
\begin{align*}
&\tag{U1} 
\psi_{i, - b_i}^+ (\psi_{i, - b_i}^+)^{-1} = 1 = (\psi_{i, -b_i}^+)^{-1} \psi_{i,-b_i}^+,  
	\quad 
	\psi_{i,0}^- (\psi_{i,0}^-)^{-1} = 1 = (\psi_{i,0}^-)^{-1} \psi_{i,0}^-, 
	\\
	& [\psi_{i}^{\ve} (z), \psi_j^{\ve'} (w)] =0 \quad (\ve, \ve' \in \{ \pm\}), 
\\
&\tag{U2} 
(w - q^{a_{ij}} z) e_i (z) e_j (w) = (q^{a_{ij}} w - z) e_j (w) e_i (z), 
\\
&\tag{U3} 
(q^{a_{ij}} w - z) f_i(z) f_j(w) = (w- q^{a_{ij}} z) f_j(w) f_i(z), 
\\
&\tag{U4} 
(w - q^{a_{ij}} z) \psi_i^{\ve} (z) e_j(w) = (q^{a_{ij}} w - z) e_j (w) \psi_i^{\ve} (z) 
	\quad (\ve \in \{\pm\}), 
\\
&\tag{U5}  
(q^{a_{ij}} w - z) \psi_i^{\ve} (z) f_j (w) = (w - q^{a_{ij}} z) f_j (w) \psi_i^{\ve} (z) 
	\quad ( \ve \in \{\pm\}), 
\\
&\tag{U6} 
[e_i(z), f_j(w)] 
	= \frac{\d_{i,j}}{q-q^{-1}} \d \left( \frac{w}{z} \right) (\psi_i^+ (z) - \psi_i^- (z)), 
\\
&\tag{U7} 
e_i(z) e_j(w) = e_j(w) e_i(z)  \text{ if } j \not= i,  i \pm 1, 
	\\ & 
	e_{i\pm 1}(w) \big( e_i(z_1) e_i(z_2) + e_i(z_2) e_i(z_1) \big) + \big( e_i(z_1) e_i(z_2) + e_i(z_2) e_i(z_1) \big) e_{i\pm 1}(w) 
	\\ & \quad 
	= [2] \big( e_i (z_1) e_{i \pm 1} (w) e_i (z_2) + e_i (z_2) e_{i \pm 1} (w) e_i (z_1) \big),   
\\
&\tag{U8}   
f_i(z) f_j(w) = f_j(w) f_i(z)  \text{ if } j \not= i,  i \pm 1, 
	\\ & 
	f_{i\pm 1}(w) \big( f_i(z_1) f_i(z_2) + f_i(z_2) f_i(z_1) \big) + \big( f_i(z_1) f_i(z_2) + f_i(z_2) f_i(z_1) \big) f_{i\pm 1}(w) 
	\\ & \quad 
	= [2] \big( f_i (z_1) f_{i \pm 1} (w) f_i (z_2) + f_i (z_2) f_{i \pm 1} (w) f_i (z_1) \big),   
\end{align*}
where $(a_{ij})_{1\leq i,j \leq m-1}$ is the Cartan matrix of type $A_{m-1}$, 
and we consider the generating series 
\begin{align*}
&e_i(z) = \sum_{t \in \ZZ} e_{i,t} z^{t}, 
\quad 
f_i(z) = \sum_{t \in \ZZ} f_{i,t} z^{t},  
\\
&\psi_i^+(z) = \sum_{s \geq 0} \psi_{i, - b_i +s}^+ z^{- b_i +s}, 
\quad 
\psi_i^- (z) = \sum_{s \geq 0} \psi_{i,-s}^- z^{ - s}, 
\quad 
\d(z) = \sum_{t \in \ZZ} z^t.
\end{align*}
\end{description}
\end{definition}

%%%%
\remark 
\label{rem SQA} 
In \cite{FT19}, 
the shifted quantum affine algebra $U_{q,[\bb]}$ is denoted by 
$\cU^{\mathbf{sc}}_{\mu,0}$ with the coweight $\mu$ such that $b_i = \a_i^{\vee} ( \mu)$. 

%%%%%
\para  
We define elements $h_{i, \pm t} \in U_{q,[\bb]}$ ($1\leq i \leq m-1$, $t \in \ZZ_{>0}$) by 
\begin{align*}
& (\psi_{i,- b_i}^+ z^{ - b_i} )^{-1} \psi_i^+(z)  = \exp \big( (q-q^{-1}) \sum_{t >0} h_{i,t} z^{t} \big), 
\\
& (\psi_{i,0}^-)^{-1} \psi_i^-(z) = \exp \big( - (q-q^{-1}) \sum_{t >0} h_{i,-t}z^{-t} \big). 
\end{align*}
%%%
Then, the relations (U4) and (U5) are replaced by 
\begin{align*}
\tag{U4'} 
& \psi_{i,-b_i}^+ e_{j,s} (\psi_{i,-b_i}^+)^{-1} = q^{a_{ij}} e_{j,s}, 
	\quad 
	\psi_{i,0}^- e_{j,s} ( \psi_{i,0} )^{-1} = q^{- a_{ij}} e_{j,s}, 
	\quad 
	[h_{i,t}, e_{j,s}] = \frac{ [ t a_{ij}] }{t} e_{j,s+t}, 
\\
\tag{U5'} 
& \psi_{i,-b_i}^+ f_{j,s} (\psi_{i,-b_i}^+)^{-1} = q^{ - a_{ij}} f_{j,s}, 
	\quad 
	\psi_{i,0}^- f_{j,s} ( \psi_{i,0} )^{-1} = q^{ a_{ij}} f_{j,s}, 
	\quad 
	[h_{i,t}, f_{j,s}] = - \frac{ [ t a_{ij}] }{t} f_{j,s+t} 
\end{align*}
respectively. 
In particular, we have 
\begin{align}
\label{ei spm1 h e}
e_{i,s \pm 1} =  \frac{1}{[2]} [h_{i, \pm 1}, e_{i,s}], 
\quad 
f_{i, s \pm 1} = - \frac{1}{[2]} [h_{i, \pm 1}, f_{i, s }]  
\quad (s \in \ZZ). 
\end{align}
We also have 
\begin{align}
\label{psiis+}
\begin{split}
&\psi_{i,s}^+ = (q-q^{-1}) [e_{i,s}, f_{i,0}] \quad (s >0), 
\quad 
\psi_{i,s}^- = - (q-q^{-1}) [e_{i,s}, f_{i,0}] \quad ( s < - b_i), 
\\
& \psi_{i,s}^+ - \psi_{i,s}^- = (q-q^{-1}) [e_{i, s}, f_{i,0}] \quad ( - b_i \leq s \leq 0)
\end{split}
\end{align}
by the relation (U6). 

%%%%%
\para  
We can easily check that 
the elements $\psi_{i, - b_i}^+ \psi_{i,0}^-$ ($1\leq i \leq m-1$) are central elements of $U_{q,[\bb]}$. 
For $\Be = (\eta_1, \eta_2,\dots, \eta_{m-1}) \in (\CC^{\times})^{m-1}$, 
let $\cI_{q,[\bb]}^{\Be}$ be the two-sided ideal of $U_{q,[\bb]}$ generated by 
$\{\psi_{i, - b_i}^+ \psi_{i,0}^- - (- \eta_i )^{b_i} \mid 1\leq i \leq m-1 \} $, 
and we consider the quotient algebra 
\begin{align*}
U_{q,[\bb]}^{\Be} = U_{q,[\bb]} / \cI_{q,[\bb]}^{\Be}. 
\end{align*}

In the case where $\bb= \mathbf{0} = (0,0, \dots,0)$, 
the ideal $\cI_{q,[\bb]}^{\Be}$ does not depend on the choice of $\Be$, 
and we have 
\begin{align}
\label{iso Uq0 UqLsl}
U_{q,[\mathbf{0}]}^{\Be} 
\cong U_q (L \Fsl_m) 
\end{align} 
as algebras, 
where $U_q (L \Fsl_m)$ is the quantum loop algebra associated with $\Fsl_m$. 
We denote corresponding generators of $U_q (L \Fsl_m)$ via the above isomorphism 
by the same symbols. 

%%%%%
\para 
Let 
$U_{q,[\bb]}^{>}$, $U_{q,[\bb]}^{<}$ and $U_{q,[\bb]}^0$ be 
the subalgebra of $U_{q,[\bb]}$ 
generated by 
$\{e_{i,t} \mid 1\leq i \leq m-1, t\in \ZZ\}$, 
$\{f_{i,t} \mid 1\leq i \leq m-1, t\in \ZZ\}$ 
and 
$\{ \psi_{i, -b_i +s}^+, (\psi_{i, - b_i}^+)^{-1}, \psi_{i, -s}^-, (\psi_{i,0}^-)^{-1} \mid 1 \leq i \leq m-1, s \in \ZZ_{\geq 0} \}$ 
respectively.  
Then we have the following triangular decomposition. 

\begin{prop}[{\cite[Proposition 5.1]{FT19}}] 
The multiplication map 
\begin{align}
\label{tri decom}
U_{q,[\bb]}^{<} \otimes U_{q,[\bb]}^0 \otimes U_{q,[\bb]}^{>} 
\ra 
U_{q,[\bb]}
\end{align}
is an isomorphism of $\CC$-vector spaces. 
\end{prop}

%%%
\para 
For $\bb =(b_1,\dots, b_{m-1}) \in \ZZ^{m-1}$, 
$\bc =(c_1,\dots, c_{m-1}), \bd=(d_1,\dots, d_{m-1}) \in (\ZZ_{\geq 0})^{m-1}$ 
and $\Be =( \eta_1, \dots, \eta_{m-1}) \in (\CC^{\times})^{m-1}$, 
there exists an injective algebra homomorphism 
\begin{align*}
\iota^{\Be}_{[\bc, \bd]} : U_{q,[\bb]} \ra U_{q,[\bb - (\bc+ \bd)]}
\end{align*} 
such that 
\begin{align*}
&e_i(z) \mapsto (1 - \eta_i  z^{-1})^{c_i} e_i(z), 
\quad 
f_i(z) \mapsto (1 - \eta_i  z^{-1})^{d_i} f_i(z), 
\\
&\psi_i^{\pm}(z) \mapsto (1- \eta_i  z^{-1})^{c_i + d_i} \psi_i^{\pm} (z) 
\end{align*}
by \cite[Lemma 10.18, Theorem 10.19]{FT19} (see also \cite[\S 4.5]{Her23}). 
More precisely, we have 
\begin{align}
\label{exp iotacdeta}
\begin{split}
&\iota_{[\bc, \bd]}^{\Be} (e_{i,t}) = \sum_{k=0}^{c_i} \begin{pmatrix} c_i \\ k \end{pmatrix} (-\eta_i)^k e_{i, t+k}, 
	\quad 
\iota_{[\bc, \bd]}^{\Be} (f_{i,t} ) = \sum_{k=0}^{d_i} \begin{pmatrix} d_i \\ k \end{pmatrix} ( - \eta_i)^k f_{i, t+k}, 
\\
& \iota_{[\bc, \bd]}^{\Be} (\psi_{i, -b_i +s}^+) 
	= \sum_{ k = \max \{ (c_i+d_i) - s, \, 0\}}^{c_i+d_i} 
	\begin{pmatrix} c_i + d_i \\ k \end{pmatrix} (- \eta_i)^k \psi_{i, -b_i +s +k}^+, 
\\
& \iota_{[\bc, \bd]}^{\Be} ( \psi_{i, - s}^-) 
	= \sum_{k= \max\{(c_i + d_i) - s, \, 0\}}^{c_i+d_i} \begin{pmatrix} c_i + d_i \\ k \end{pmatrix} 
	(- \eta_i)^{(c_i+d_i) -k} \psi_{i, - s+ (c_i+d_i) -k}^-. 
\end{split}
\end{align}
In particular, we have 
\begin{align}
\label{iota cdeta psii-bi+}
\iota_{[\bc, \bd]}^{\Be} ( \psi_{i, -b_i}^+)= ( - \eta_i )^{c_i+d_i} \psi_{i, - b_i + (c_i+d_i)}^+, 
\quad 
\iota_{[\bc, \bd]}^{\Be} (\psi_{i,0}^-) = \psi_{i,0}^-. 
\end{align}
By \eqref{iota cdeta psii-bi+}, we see that 
$\iota_{[\bc,\bd]}^{\Be} ( \cI_{q,[\bb]}^{\Be}) = \cI_{q,[\bb-(\bc+\bd)]}^{\Be} \cap \Im \iota_{[\bc, \bd]}^{\Be}$, 
and the homomorphism $\iota_{[\bc, \bd]}^{\Be} : U_{q,[\bb]} \ra U_{q, [\bb-(\bc+ \bd)]}$ induces an injective homomorphism 
\begin{align*}
\iota_{[\bc,\bd]}^{\Be} : U_{q,[\bb]}^{\Be} \ra U_{q,[\bb - (\bc+\bd)]}^{\Be}. 
\end{align*}

%%%
\para 
We denote by $U_q(\wh{\Fsl}_m)$ the quantum affine algebra associated with the affine Lie algebra $\wh{\Fsl}_m$. 
The quantum affine algebra $U_q (\wh{\Fsl}_m)$ is defined by using Chevalley generators 
$e_i, f_i, k_i^{\pm} $ ($0 \leq i \leq m-1$) with the usual defining relations for the Drinfeld-Jimbo quantum groups. 
The quantum affine algebra $U_q (\wh{\Fsl}_m)$ has the coproduct 
$\D : U_q (\wh{\Fsl}_m) \ra U_q (\wh{\Fsl}_m) \otimes U_q (\wh{\Fsl}_m)$ such that 
\begin{align*}
\D (e_i) = e_i \otimes k_i^+ + 1 \otimes e_i, 
\quad 
\D(f_i) = f_i \otimes 1 + k_i^- \otimes f_i, 
\quad 
\D(k_i^{\pm}) = k_i^{\pm} \otimes k_i^{\pm}. 
\end{align*}

%%%
Put $c = k_0^+ k_1^+ \dots k_{m-1}^+ \in U_q (\wh{\Fsl}_m)$, 
then $c$ is the canonical central element of $U_q (\wh{\Fsl}_m)$. 
By \cite{Dri87} and \cite{Bec94}, it is known that 
there exists an isomorphism of algebras 
\begin{align}
\label{iso Uq whsl UqLsl}
U_q (\wh{\Fsl}_m)/ (c-1) \ra U_q (L \Fsl_m)
\end{align} 

%%%
In \cite[Theorem 10.13]{FT19}, 
the coproduct $\D$ of $U_q (\wh{\Fsl}_m)$ is described by using some generators of $U_q ( L \Fsl_m)$  
through the isomorphism $U_q (\wh{\Fsl}_m)/ (c-1) \cong U_q (L \Fsl_m)$, 
and we can naturally lift this coproduct to the coproduct of $U_{q,[\mathbf{0}]}$. 
We denote this coproduct of $U_{q,[\mathbf{0}]}$ by the same symbol $\D$. 

%%%%%
\para 
For $\bb, \bc, \bd \in (\ZZ_{\geq 0})^{m-1}$ such that $\bb = \bc + \bd$ and $\Be \in (\CC^{\times})^{m-1}$,  
there exists an algebra homomorphism 
\begin{align}
\label{hom Ddc}
\D_{\bd, \bc} : U_{q,{[ \bb]}} \ra U_{q,{[ \bd ]}} \otimes U_{q, { [ \bc ]}}
\end{align}
such that the diagram
\begin{align}
\label{comm diagram Ddc}
\xymatrix{
U_{q,[\bb]} \ar[r]^{\D_{\bd, \bc} \quad \quad }  \ar[d]_{\iota_{[\bc, \bd]}^{\Be}} 
	& U_{q,[\bd]} \otimes U_{q, [\bc]} \ar[d]^{\iota_{[\mathbf{0}, \bd]}^{\Be} \otimes \iota_{[\bc, \mathbf{0}]}^{\Be}}
\\
U_{q, [\mathbf{0}]} \ar[r]^{\D \qquad } & U_{q, [\mathbf{0}]} \otimes U_{q, [\mathbf{0}]  } 
}
\end{align}
commutes by \cite[Theorem 10.20]{FT19}. 

We can obtain the following proposition from the corresponding results for $U_q ( L \Fsl_m)$ 
(e.g. \cite[Proposition 4.4]{CP91})  
by using \eqref{exp iotacdeta} and the commutative diagram \eqref{comm diagram Ddc}. 
%%%%%%%%%%

\begin{prop}[{cf. \cite[Lemma 12.4]{Zha24}}]
\label{Prop Ddc}
For $\bb, \bc, \bd \in (\ZZ_{\geq 0})^{m-1}$ such that $\bb = \bc + \bd$ and $\Be \in (\CC^{\times})^{m-1}$, 
we have the following. 
\begin{enumerate}
\item For $1\leq i \leq m-1$ and $s \geq 0$, 
\begin{align*}
&\D_{\bd, \bc} (\psi_{i, - b_i +s}^+) 
	\equiv \sum_{k=0}^s \psi_{i, - d_i +k}^+ \otimes \psi_{i, - c_i + s -k}^+ 
		\mod \fX^+_{[\bd]} \otimes \fX_{[\bc]}^-, 
\\
& \D_{\bd, \bc} (\psi_{i, -s}^-) 
	\equiv \sum_{k=0}^s \psi_{i, - k}^- \otimes \psi_{i, - s +k}^-  
		\mod \fX^+_{[\bd]} \otimes \fX_{[\bc]}^-, 
\end{align*}
where 
$\fX_{[\bd]}^+$ (resp. $\fX_{[\bc]}^-$) is the left ideal of $U_{q,[\bd]}$ (resp. $U_{q,[\bc]}$) 
generated by $\{ e_{i,t} \mid 1\leq i \leq m-1, \, t \in \ZZ\}$ 
(resp. $\{ f_{i,t} \mid 1\leq i \leq m-1, \, t \in \ZZ\}$). 

\item 
For $1\leq  i \leq m-1$ and $t \in \ZZ$, 
\begin{align*}
\D_{\bd, \bc} (e_{i,t}) 
	&\equiv \begin{cases}
		\dis 
		1 \otimes e_{i,t} + \sum_{k=1}^t e_{i, t-k} \otimes \psi_{i,k}^+ + \sum_{k=0}^{c_i} e_{i, t+k} \otimes \psi_{i, -k}^+ 
			& \text{ if } t \geq 0, 
		\\
		\dis 
		1 \otimes e_{i,t} + \sum_{k=0}^{-t-1} e_{i, t +k} \otimes \psi_{i, -k}^- 
			+ \d_{( - t < c_i)} \sum_{k=-t}^{c_i} e_{i,t+k} \otimes \psi_{i, -k}^+ 
		& \text{ if } t <0 
	\end{cases}
	\\
	& \mod \fX_{[\bd]}^{+2} \otimes \fX_{[\bc]}^-, 
\end{align*}
where $\fX_{[\bd]}^{+2}$  is the left ideal of $U_{q,[\bd]}$  
generated by $\{ e_{i_1,t_1} e_{i_2,t_2} \mid 1\leq i_1, i_2\leq m-1, \, t_1,t_2 \in \ZZ\}$. 

\item 
For $1\leq  i \leq m-1$ and $t \in \ZZ$,  
\begin{align*}
\D_{\bd, \bc} (f_{i,t}) 
	& \equiv \begin{cases}
		\dis f_{i,t} \otimes 1 + \sum_{k=0}^{t-1} \psi_{i,k}^+ \otimes f_{i, t-k} + \sum_{k=1}^{d_i} \psi_{i, -k}^+ \otimes f_{i, t+k} 
			& \text{ if } t >0, 
		\\
		\dis 
		f_{i,t} \otimes 1 + \sum_{k=0}^{-t} \psi_{i, -k}^- \otimes f_{i, t+k} 
			+ \d_{( -  t < d_i)} \sum_{k=-t+1}^{d_i} \psi_{i, -k}^+ \otimes f_{i, t+k} 
			& \text{ if } t \leq 0 
	\end{cases}
	\\
	& \mod \fX_{[\bd]}^+ \otimes \fX_{[\bc]}^{-2}, 
\end{align*}
where $\fX_{[\bd]}^{-2}$  is the left ideal of $U_{q,[\bd]}$  
generated by $\{ f_{i_1,t_1} f_{i_2,t_2} \mid 1\leq i_1, i_2\leq m-1, \, t_1,t_2 \in \ZZ\}$. 
\end{enumerate}
\end{prop}

%%%%%%%%%%%%%%%%%%%%%%%%%%%%%%%%%%%%%%%%%%%%%%%%%%%%%%%%%%%%%%%

%%%%%%%%%%%%%%%%%%%%%%%%%%%%%%%%%%%%%%%%%%%%%%%%%%%%%%%%%%%%%%%

\section{Finite dimensional simple modules of shifted quantum affine algebras} 

In this section, we recall some results for finite dimensional simple modules of the shifted quantum affine algebra 
$U_{q,[\bb]} = U_{q, [\bb]}(L\Fsl_m)$ 
obtained in \cite{Her23}. 
By \cite[Proposition 6.3]{Her23}, 
there is no non-zero finite dimensional modules of $U_{q,[\bb]}$ if $\bb \not\in (\ZZ_{\geq 0})^{m-1}$, 
and we assume that $\bb \in (\ZZ_{\geq 0})^{m-1}$ in the remainder of this paper. 

\para 
First, we give a classification of one-dimensional modules. 
Put 
\begin{align*}
\BB_{[\bb]} = \{ \Bb = (\b_{i, - b_i +s})^{0 \leq s \leq b_i}_{1 \leq i \leq m-1} 
 	\mid \b_{i, - b_i}, \b_{i,0} \in \CC^{\times}, \, \b_{i, - b_i +s} \in \CC \, (0 <s< b_i) \}. 
\end{align*}
For $\Bb \in \BB_{[\bb]}$, 
we can define a one-dimensional $U_{q,[\bb]}$-module 
$L_{\Bb} = \CC w$ by 
\begin{align}
\label{def LBb} 
\begin{split}
&e_{i,t} \cdot w = f_{i,t} \cdot w =0 \quad (1\leq i \leq m-1, \, t \in \ZZ), 
\\
& \psi_{i, - b_i+s}^+ \cdot w = 
	\begin{cases} 
		\b_{i, - b_i +s} w & \text{ if } 0 \leq s \leq b_i, 
		\\
		0 & \text{ if } s > b_i 
	\end{cases} 
	\quad ( 1 \leq i \leq m-1, \, s \in \ZZ_{\geq 0} ), 
\\
& \psi_{i, -s }^- \cdot w = 
	\begin{cases}
	\b_{i, -s } w & \text{ if } 0 \leq s \leq b_i, 
	\\
	0 & \text{ if } s > b_i
	\end{cases} 
	\quad ( 1 \leq i \leq m-1, \, s \in \ZZ_{\geq 0} ).
\end{split}
\end{align}
Then we see that 
$\{ L_{\Bb} \mid \Bb \in \BB_{[\bb]}\}$ gives a set of all one-dimensional $U_{q,[\bb]}$-modules up to isomorphism 
by direct calculations. 

%%%%%
\para \label{ellbb hw} 
We can consider the following notion of highest weight modules (resp. lowest weight modules) 
associated with the triangular decomposition \ref{tri decom}. 

Let $M$ be a $U_{q,[\bb]}$-module. 
For $\Bg =(\g_{i, - b_i+s}^+, \g_{i, -s}^-)^{s \in \ZZ_{\geq 0}}_{1 \leq i \leq m-1}$, 
we say that $v \in M$ is an $\ell_{[\bb]}$-weight vector of $\ell_{[\bb]}$-weight $\Bg$ if 
\begin{align*}
\psi_{i, - b_i +s}^+ \cdot v = \g_{i, - b_i +s}^+ v, 
\quad 
\psi_{i,-s}^- \cdot v = \g_{i, -s}^- v 
\quad (1\leq i \leq m-1, \, s \in \ZZ_{\geq 0}), 
\end{align*}
and say that $v \in M$ is a singular vector (resp. an $f$-singular vector) 
if $e_{i,t} \cdot v =0$ (resp. $f_{i,t} \cdot v =0$) for all $1\leq i \leq m-1$ and $t \in \ZZ$. 
We also say that 
$M$ is an $\ell_{[\bb]}$-highest weight module of $\ell_{[\bb]}$-highest weight $\Bg$ 
(resp. an $\ell_{[\bb]}$-lowest weight module of $\ell_{[\bb]}$-lowest weight $\Bg$) 
if 
there exists a vector $v_0 \in M$ satisfying the following conditions: 
\begin{enumerate}
\item 
$M$ is generated by $v_0$ as a $U_{q,[\bb]}$-module. 

\item 
$v_0$ is an $\ell_{[\bb]}$-weight vector of $\ell_{[\bb]}$-weight $\Bg$. 

\item 
$v_0$ is a singular vector (resp. an $f$-singular vector). 
\end{enumerate}
In this case, we say that $v_0$ is an $\ell_{[\bb]}$-highest weight vector (resp. an $\ell_{[\bb]}$-lowest weight vector). 
We sometimes drop \lq\lq$\ell_{[\bb]}$" in the above notions if there is no risk of confusion. 
We also denote \lq\lq $\ell_{[\bb]}$" by \lq \lq $\ell$" simply for $U_q (L\Fsl_m)$-modules. 

By standard arguments, 
we see that any finite dimensional simple $U_{q,[\bb]}$-module is both an $\ell_{[\bb]}$-highest weight module 
and an $\ell_{[\bb]}$-lowest weight module. 

%%%%%
\para 
\label{par simple Uq} 
We recall a classification of finite dimensional simple modules of the quantum loop algebra $U_q(L\Fsl_m)$ in 
\cite{CP91} and \cite{CP95}.
The subalgebra of $U_{q}(L\Fsl_m)$ generated by 
$\{e_{i,0}, f_{i,0}, \psi_{i,0}^{\pm} \mid 1 \leq i \leq m-1\}$ is isomorphic to the quantum group 
$U_q(\Fsl_m)$ associated with the special linear Lie algebra $\Fsl_m$. 
We define the type of a $U_q (L\Fsl_m)$-module by the type as $U_q (\Fsl_m)$-module through the restriction. 

Recall that a finite dimensional simple $U_q (L\Fsl_m)$-module of type $1$ is a highest weight module, and 
let $\Bg =(\g_{i, s}^+, \g_{i, -s}^-)^{s \in \ZZ_{\geq 0}}_{1 \leq i \leq m-1}$ be its highest weight.  
Then there exists a tuple of polynomials $\bP = (P_i (z))_{1 \leq i \leq m-1} $ such that 
$P_i(0)=1$ and 
\begin{align}
\label{hw simple UqLlm} 
\sum_{s \geq 0 } \g_{i,s}^+ z^s = q^{\deg P_i(z)} \frac{P_i(q^{-1} z)}{P_i(qz)} = \sum_{s \geq 0} \g_{i, -s}^- z^{-s}. 
\end{align}
On the other hand, 
for any tuple of polynomials $\bP = (P_i (z))_{1 \leq i \leq m-1} $ such that $P_i(0)=1$, 
there exists a finite dimensional simple $U_q (L\Fsl_m)$-module of type $1$ with the highest weight $\Bg$ given by 
\eqref{hw simple UqLlm}. 
We denote by $L(\bP)$ the finite dimensional simple $U_q (L\Fsl_m)$ of type $1$ corresponding to $\bP$. 
Each polynomial $P_i(z)$ is called the Drinfeld polynomial. 
Put 
\begin{align*}
\DD = \{ \bP = (P_i(z))_{1 \leq i \leq m-1} \mid P_i(z) \in \CC [z] \text{ such that } P_i(0) =1\}, 
\end{align*}
and we have the following. 

\begin{thm}[{\cite{CP91},  \cite{CP95}}]  
The set of isomorphism classes of finite dimensional simple $U_q (L\Fsl_m)$-modules of type $1$ is given by 
$\{L(\bP) \mid \bP \in \DD\}$. 
\end{thm}

%%%%%
\para 
\label{simple Uqb}
A classification of finite dimensional simple modules of the shifted quantum affine algebra $U_{q,[\bb]}=U_{q,[\bb]}(L\Fsl_m)$ 
is given in \cite{Her23}, and we reformulate it as follows. 

For a one-dimensional $U_{q,[\bb]}$-module $L_{\Bb} = \CC w$ ($\Bb \in \BB_{[\bb]}$) 
and a finite dimensional simple $U_q (L\Fsl_m)$-module $L(\bP)$ ($\bP \in \DD$), 
we obtain a $U_{q,[\bb]}$-module $L_{\Bb} \otimes L(\bP)$ through the homomorphism 
$\D_{\bb,0} : U_{q,[\bb]} \ra U_{q,[\bb]} \otimes U_{q,[0]}$ in \eqref{hom Ddc}. 
Let $v_0 \in L(\bP)$ be an $\ell_{[0]}$-highest weight vector of $\ell_{[0]}$-highest weight 
$\Bg = (\g_{i,s}^+, \g_{i,-s}^-)^{s \in \ZZ_{\geq 0}}_{1 \leq i \leq m-1}$. 
By Proposition \ref{Prop Ddc}, we see that 
$w \otimes v_0 \in L_{\Bb} \otimes L(\bP)$ is both an $\ell_{[\bb]}$-singular vector and  an $\ell_{[\bb]}$-weight vector. 
Let $\wt{\Bg} = (\wt{\g}_{i, - b_i +s}^+, \wt{\g}_{i, -s}^-)^{s \in \ZZ_{\geq 0}}_{1 \leq i \leq m-1}$ 
be the $\ell_{[\bb]}$-weight of $w \otimes v_0$. 
Then the simple quotient of $U_{q,[\bb]} \cdot (w \otimes v_0)$ 
is a finite dimensional simple highest weight module of $\ell_{[\bb]}$-highest weight $\wt{\Bg}$. 
By direct calculations using Proposition \ref{Prop Ddc} together with \eqref{hw simple UqLlm}, 
we have 
\begin{align}
\label{hw simple Uqb}
\sum_{ t \geq - b_i} \wt{\g}_{i,t} z^t 
= q^{\deg P_i} \frac{P_i(q^{-1} z)}{P_i(qz)} \frac{S_i(z)}{z^{b_i}} 
= \sum_{t \geq 0} \wt{\g}_{i, -t} z^{-t},  
\end{align}
where 
$S_i(z) = \sum_{r=0}^{b_i} \b_{i, - b_i +r} z^r$.

Put 
\begin{align*}
\SS_{[\bb]} = \{ \bS = (S_i(z))_{1\leq i \leq m-1} \in \CC[z]^{m-1} \mid 
	\deg S_i(z) = b_i, \, S_i(0) \not= 0 \}, 
\end{align*} 
and we have a bijection 
$\BB_{[\bb]} \ra \SS_{[\bb]} $ ($\Bb=(\b_{i,- b_i +s})^{0 \leq s \leq b_i}_{1 \leq i \leq m-1}  \mapsto \bS=(S_i(z))_{1\leq i \leq m-1}$ ) 
by $S_i(z) = \sum_{r=0}^{b_i} \b_{i, - b_i +r} z^r$. 
As a consequence of the above argument,  
we see that there exists a finite dimensional simple $U_{q,[\bb]}$-module whose highest weight is given by \eqref{hw simple Uqb} 
for $(\bS, \bP) \in \SS_{[\bb]} \times \DD$. 
We denote this simple module by $L(\bS, \bP)$. 
In \cite{Her23}, the converse statement is also proven, and we have the following theorem. 
%%%%%%
\begin{thm}[{\cite[Theorem 6.5]{Her23}}] 
\label{Thm simple Uqb}
For each $(\bS, \bP) \in \SS \times \DD$, there exists a finite dimensional simple $U_{q,[\bb]}$-module $L(\bS, \bP)$ 
whose highest weight is given by \eqref{hw simple Uqb}. 
Conversely, any finite dimensional simple $U_{q,[\bb]}$-module is isomorphic to $L( \bS, \bP)$ 
for some $(\bS, \bP) \in \SS_{[\bb]} \times \DD$. 
\end{thm}

We also call  each polynomials $S_i(z)$ and $P_i(z)$  the Drinfeld polynomials. 

\remark 
\label{Rem DP} 
For $P(z), S(z) \in \CC[z]$ such that $P(0)=1$, $\deg S(z) = b$, $\deg S(0) \not=0$, 
if $P(qz)$ and $S(z)$ has a common divisor $(1 - az)$, 
we have 
\begin{align*}
q^{\deg P} \frac{ P (q^{-1} z)}{ P(qz)} \frac{ S(z)}{z^{b_i}} 
	= q^{\deg \wt{P}} \frac{ \wt{P} (q^{-1} z)}{ \wt{P} (q z)} \frac{ \wt{S}(z)}{z^{b_i}}, 
\end{align*}
where 
\begin{align*}
\wt{P} (z) = \frac{ P(z)}{(1- q^{-1} a z)}, 
\quad 
\wt{S}(z) = \frac{ q ( 1 - q^{-2} az) S(z)}{(1 - a z)}.
\end{align*}
Thus, the surjective map 
\begin{align*}
\SS_{[\bb]} \times \DD \ra \{ \text{finite dimensional simple $U_{q,[\bb]}$-module} \}/_{\cong}, 
\quad 
(\bS, \bP) \mapsto L(\bS, \bP)
\end{align*}
is not injective if $\bb\not= (0, \dots, 0)$. 

%%%%%%%%%%%%%%%%%%%%%%%%%%%%%%%%%%%%%%%%%%%%%%%%%%%%%%%%%%%%%%%

%%%%%%%%%%%%%%%%%%%%%%%%%%%%%%%%%%%%%%%%%%%%%%%%%%%%%%%%%%%%%%%
\section{$q$-characters for finite dimensional modules of shifted quantum affine algebras} 

\para 
First, we recall some statements for $q$-characters of finite dimensional modules of the quantum loop algebra $U_{q}(L\Fsl_m)$ 
according to \cite{FR99} and \cite{FM01}. 
For a finite dimensional $U_q(L\Fsl_m)$-module $M$ of type $1$, 
we consider the decomposition of $M$ into generalized simultaneous eigenspaces for the actions of 
$\{\psi_{i, \pm s }^{\pm} \mid 1 \leq i \leq m-1, \, s \in \ZZ_{\geq 0}\}$:
\begin{align}
\label{decom wll-wt sp}
\begin{split}
&M = \bigoplus_{\Bg=(\g^{\pm}_{i, \pm s} )^{s \in \ZZ_{\geq 0}}_{1 \leq i \leq m-1}} M_{\Bg}, 
\\
& M_{\Bg} = \left\{ v \in M \mid (\psi_{i, \pm s}^{\pm} - \g_{i, \pm s}^{\pm})^{N_{i, \pm s}} \cdot v =0 
	 \begin{array}{l}
	 \text{ for some } 
	 N_{i, \pm s} \in \ZZ_{>0} 
	 \\  
	 \quad 
	 (1 \leq i \leq m-1, \, s \in \ZZ_{\geq 0} )
	 \end{array}
	 \right\}. 
\end{split} 
\end{align}
Then, we say that $\Bg = (\g_{i, \pm s}^{\pm})^{s \in \ZZ_{\geq 0}}_{1 \leq i \leq m-1} $ is an $\ell$-weight of $M$ if 
$M_{\Bg} \not=0$, 
and we denote the set of $\ell$-weights of $M$ by $\lwt M $. 
In \cite{FR99}, it is proven that, 
for $\Bg = (\g_{i, \pm s}^{\pm}) \in \lwt M$, 
there exist polynomials $Q_i(z) , R_i(z) \in \CC[z]$ ($1 \leq i \leq m-1$) such that $Q_i(0)=R_i(0) =1$ and 
\begin{align}
\label{Qi Ri} 
\sum_{s \geq 0} \g_{i,s}^+ z^s = q^{\deg Q_i - \deg R_i} 
	\frac{ Q_i (q^{-1} z)}{Q_i(qz)} \frac{ R_i(qz) }{R_i(q^{-1} z)} 
= \sum_{s \geq 0} \g_{i,-s}^- z^{-s}. 
\end{align}

For $\Bg \in \lwt M$, 
we define a monomial 
\begin{align*}
Y^{\Bg} = \prod_{i=1}^{m-1} \Big( \prod_{j=1}^{k_i} Y_{i, \z_{i,j}} \prod_{j'=1}^{l_i} Y_{i, \xi_{i, j'}}^{-1} \Big) 
\end{align*}
in the Laurent polynomial ring $\ZZ [Y_{i,\z}^{\pm}]^{\z \in \CC^{\times}}_{1 \leq i \leq m-1}$ 
if the polynomials $Q_i(z)$ and $R_i(z)$  satisfying \eqref{Qi Ri} are written as 
\begin{align*}
\begin{split} 
& Q_i(z) = (1 - \z_{i,1} z) (1 - \z_{i,2} z) \dots (1 -\z_{i, k_i} z), 
\\
& R_i (z) = (1 - \xi_{i,1} z) (1 - \xi_{i,2} z) \dots (1 - \xi_{i, l_i} z)
\end{split} 
\quad 
(\z_{i,j}, \xi_{i,j'} \in \CC^{\times}). 
\end{align*}

For a finite dimensional $U_q(L \Fsl_m)$-module $M$ of type $1$, 
we define the $q$-character  of $M$ as 
\begin{align*}
\chi_q (M) = \sum_{\Bg \in \lwt M} ( \dim M_{\Bg} ) Y^{\Bg} 
\in \ZZ [Y_{i,\z}^{\pm}]^{\z \in \CC^{\times}}_{1 \leq i \leq m-1}. 
\end{align*}
Then we have the following proposition.  

\begin{prop}[{\cite{FR99}, \cite{FM01}}] 
Let $U_q (L\Fsl_m) \cmod_f^1$ be the category of finite dimensional $U_q (L\Fsl_m)$-modules of type $1$, 
and $R_0 (U_q (L\Fsl_m) \cmod_f^1)$ be its Grothendieck ring. 
Then the map 
\begin{align*}
\chi_q : R_0 (U_q (L\Fsl_m)\cmod_f^1) \ra \ZZ [Y_{i,\z}^{\pm}]^{\z \in\CC^{\times}}_{1 \leq i \leq m-1} 
\quad 
(M \mapsto \chi_q (M) )
\end{align*}
is an injective ring homomorphism. 
\end{prop}

%%%%%%%%%%%%%
\para 
We introduce a similar notation of $q$-characters for finite dimensional modules of the shifted quantum affine algebra 
$U_{q,[\bb]} (L\Fsl_m)$. 
For finite dimensional $U_{q,[\bb]}$-module $M$, 
we consider the generalized simultaneous  eigenspace decomposition of $M$ 
for the actions of $\{ \psi_{i, - b_i +s}^+, \, \psi_{i -s}^- \mid 1 \leq i \leq m-1, \, s \in \ZZ_{\geq 0} \} $, 
and those simultaneous eigenvalues 
$\Bg = (\g_{i, - b_i +s}^+, \g_{i,-s}^-)^{s \in \ZZ}_{1 \leq i \leq m-1}$ are called $\ell_{[\bb]}$-weights of $M$. 
We denote the set of $\ell_{[\bb]}$-weights of $M$ by $\lbwt M$. 
Then we have the following proposition. 
%%%
\begin{prop}
\label{Prop lb-weight} 
Let $M$ be a finite dimensional $U_{q,[\bb]}$-module. 
For $\Bg \in \lbwt M$, 
there exist polynomials $Q_i(z)$, $R_i(z)$ and $S_i(z)$ ($1 \leq i \leq m-1$) 
such that 
$Q_i(0)=R_i(0)=1$, $S_i ( 0) \not=0$, $\deg S_i(z) =b_i$ 
and the generating functions of $\ell_{[\bb]}$-weight $\Bg = (\g_{i, - b_i +s}^+, \g_{i,-s}^-)^{s\in \ZZ}_{1 \leq i \leq m-1}$ 
are given as 
\begin{align}
\label{ratinal func ellb wt}
\sum_{s \geq 0} \g_{i, - b_i+s}^+ z^{- b_i+s} 
	= q^{\deg Q_i - \deg R_i} \frac{ Q_i ( q^{-1} z)}{ Q_i (q z)} \frac{ R_i (q z)}{ R_i (q^{-1} z)} \frac{ S_i (z) }{ z^{b_i}} 
	= \sum_{ s \geq 0} \g_{i, -s}^- z^{-s} 
\end{align}
for $1 \leq i \leq m-1$. 

\begin{proof}
As seen in the previous section, 
any finite dimensional simple $U_{q,[\bb]}$-module is obtained as a subquotient of the module 
$L_{\Bb} \otimes L(\bP)$, 
where $L_{\Bb} = \CC w$ is the one-dimensional $U_{q,[\bb]}$-module corresponding to 
$\Bb =(\b_{i, - b_i +s})^{ 0 \leq s \leq b_i}_{1 \leq i \leq m-1} \in \BB_{[\bb]}$ 
and  $L(\bP)$ is the finite dimensional simple $U_q (L\Fsl_m)$-module corresponding to $\bP \in \DD$. 
Thus, it is enough to show the statement for $\ell_{[\bb]}$-weights of $L_{\Bb} \otimes L (\bP)$.  
Let 
$L(\bP) = \bigoplus_{ \wt{\Bg}  \in \lwt L(\bP)} L(\bP)_{\wt{\Bg}} $ 
be the decomposition of $ L(\bP)$ as in \eqref{decom wll-wt sp}, 
and we have the decomposition 
\begin{align*}
L_{\Bb} \otimes L(\bP) = \bigoplus_{\wt{\Bg} \in \lwt L(\bP)} ( L_{\Bb} \otimes L(\bP)_{\wt{\Bg}})
\end{align*}
as vector spaces. 
%%%%%%
For $\wt{\Bg} = (\wt{\g}_{i, \pm s}^{\pm}) \in \lwt L(\bP)$, 
we take $\Bg = (\g_{i, - b_i +s}^+, \g_{i, -s}^-)^{s \in \ZZ_{\geq 0}}_{1 \leq i \leq m-1}$ as 
\begin{align}
\label{def lbwt LbB otimes LP}
\g_{i, - b_i +s}^+ = \sum_{p=0}^{\min \{ s, b_i \}} \b_{i, - b_i +p} \wt{\g}_{i, s -p}^+, 
\quad 
\g_{i, -s}^- = \sum_{p=0}^{\min \{ s, b_i\}} \b_{i, - p} \wt{\g}_{i, - s+p}^-. 
\end{align}
For $v \in L(\bP)_{\wt{\Bg}}$, we have 
\begin{align*}
( \psi_{i, - b_i +s}^+ - \g_{i, - b_i+s}^+ ) \cdot (w \otimes v) 
= \sum_{p=0}^{\min \{s, b_i\}} \b_{i,- b_i +p} \big( w \otimes (\psi_{i, s-p}^+ - \wt{\g}_{i, s -p}^+) \cdot v \big)
\end{align*} 
by direct calculation using Proposition \ref{Prop Ddc}. 
Note that $(\psi_{i, s-p}^+ - \wt{\g}_{i, s-p}^+) \cdot v $  belongs to $L(\bP)_{\wt{\Bg}}$ again. 
Then we have 
\begin{align*}
&(\psi_{i, - b_i+s}^+ - \g_{i, - b_i +s}^+ )^N \cdot (w \otimes v) 
\\
&= \sum_{(n_0, n_1,\dots, n_l) \in (\ZZ_{\geq 0})^{l+1} \atop n_0 + n_1 + \dots + n_l =N} 
	\b_{(n_0, n_1, \dots, n_l)} w \otimes \big( \prod_{p=0}^l (\psi_{i, s- p}^+ - \wt{\g}_{i, s-p}^+)^{n_p} \cdot v \big) 
\end{align*}
with $l = \min\{s, b_i\}$ and  $\b_{(n_0, n_1, \dots, n_l)} \in \CC$. 
Thus, we see that 
\begin{align*}
(\psi_{i, - b_i+s}^+ - \g_{i, - b_i+s}^+)^N \cdot (w \otimes v)=0
\end{align*} 
for large enough  $N$. 
Similarly, we can check that $(\psi_{i, -s}^- - \g_{i, -s}^-)^N \cdot (w \otimes v) =0$ for large enough  $N$. 
As a consequence, we have 
$L_{\Bb} \otimes L(\bP)_{\wt{\Bg}} \subset \big( L_{\Bb} \otimes L(\bP) \big)_{\Bg}$, 
and we see that 
any $\ell_{[\bb]}$-weight $\Bg = (\g_{i, - b_i +s}^+, \g_{i, -s}^-)$ 
of $L_{\Bb} \otimes L(\bP)$ 
is written as in \eqref{def lbwt LbB otimes LP}. 

For $\wt{\Bg} = (\wt{\g}_{i, \pm s}^{\pm}) \in \lwt L(\bP)$, 
let $Q_i(z)$ and $R_i(z)$ ($1 \leq i \leq m-1$) be the polynomials 
which give the generating functions of $\ell$-weight $\wt{\Bg}$ as in  \eqref{Qi Ri}, 
and we take $\Bg = (\g_{i, - b_i +s}^+, \g_{i, -s}^-)^{s \in \ZZ_{\geq 0}}_{1 \leq i \leq m-1}$ as \eqref{def lbwt LbB otimes LP}. 
Then, we have 
\begin{align*}
\sum_{s \geq 0} \g_{i, - b_i +s}^+ z^{- b_i +s} 
&= \sum_{s \geq 0} \big( \sum_{p=0}^{\min\{ s, b_i\}} \b_{i, - b_i +p} \wt{\g}_{i, s-p}^+ \big) z^{- b_i +s} 
\\
&= \sum_{ p=0}^{b_i} \sum_{ s \geq p} \b_{i, - b_i +p} \wt{\g}_{i, s-p}^+  z^{(- b_i +p) + (s -p)} 
\\
&= \sum_{p=0}^{b_i} \b_{i,- b_i +p} z^{- b_i +p} \sum_{ s \geq p} \wt{\g}_{i, s-p}^+ z^{s-p}  
\\
&= \Big( \sum_{p=0}^{b_i} \b_{i, - b_i+p} z^{- b_i +p} \Big) \Big( \sum_{s \geq 0} \wt{\g}_{i, s}^+ z^s \Big)  
\\
&= \frac{S_i(z)}{z^{b_i}} \cdot q^{\deg Q_i - \deg R_i} \frac{ Q_i(q^{-1} z)}{Q_i(qz)} \frac{R_i (qz)}{R_i (q^{-1} z)}, 
\end{align*}
where $S_i(z) = \sum_{p=0}^{b_i} \b_{i, - b_i +p} z^p$. 
We remark that $S_i(0) \not=0$ and $\deg S_i (z) = b_i$ since $\Bb \in \BB_{[\bb]}$. 
Similarly, we have 
\begin{align*}
\sum_{s \geq 0} \g_{i, -s}^- z^{-s} 
= \frac{S_i(z)}{z^{b_i}} \cdot q^{\deg Q_i - \deg R_i} \frac{ Q_i(q^{-1} z)}{Q_i(qz)} \frac{R_i (qz)}{R_i (q^{-1} z)}
\end{align*}
in the expansion at $z = \infty$. 
Then the proposition is proven. 
\end{proof}
\end{prop}

\remark 
For the one-dimensional $U_{q,[\bb]}$-module $L_{\Bb} = \CC w$ ($\Bb \in \BB_{[\bb]}$), 
the generating functions of the only one $\ell_{[\bb]}$-weight $\Bg $ is written as 
\begin{align*}
\sum_{s \geq 0} \g_{i, - b_i +s}^+ z^{- b_i +s} = \frac{S_i(z)}{z^{b_i}} = \sum_{ s \geq 0} \g_{i, -s}^- z^{- s}, 
\end{align*} 
where $S_i(z) = \sum_{p=0}^{b_i} \b_{i, - b_i +p} z^p$. 

%%%%%%%%%
\para 
As in the case of the quantum loop algebra $U_q (L\Fsl_m)$, 
we assign a monomial with an $\ell_{[\bb]}$-weight of a finite dimensional $U_{q,[\bb]}$-module  by using 
\eqref{ratinal func ellb wt} as follows. 

For $\z, \eta, a \in \CC^{\times}$, 
we assign 
\begin{align*}
Y_{i,\z} = q \frac{1 - \z q^{-1} z}{1 - \z q z}, 
\quad 
Z_{i, \eta} =\frac{1 - \eta z}{z}, 
\quad 
\tau_{i,a} =a. 
\end{align*}
Then we can express the rational function in \eqref{ratinal func ellb wt} as 
\begin{align}
\label{monomial QRS} 
 q^{\deg Q_i - \deg R_i} \frac{ Q_i ( q^{-1} z)}{ Q_i (q z)} \frac{ R_i (q z)}{ R_i (q^{-1} z)} \frac{ S_i (z) }{ z^{b_i}} 
 = \tau_{i, a_i} \big( \prod_{j=1}^{k_i} Y_{i, \z_{i,j}} \big) \big( \prod_{j'=1}^{l_i} Y_{i, \xi_{i, j'}}^{-1} \big) 
 	\big( \prod_{p=1}^{b_i} Z_{i, \eta_{i,p}}  \big) 
\end{align}
if 
\begin{align*}
& Q_i (z) = (1 - \z_{i,1} z) (1 - \z_{i,2} z) \dots (1 - \z_{i, k_i} z), 
\\
& R_i(z) = (1 - \xi_{i,1} z) (1 - \xi_{i,2} z) \dots ( 1 - \xi_{i, l_i} z), 
\\
& S_i (z) = a_i (1 - \eta_{i,1} z) (1 - \eta_{i,2} z) \dots (1 - \eta_{i, b_i} z).
\end{align*}
%%%
We should note that 
the description of the monomial in \eqref{monomial QRS} is not uniquely determined by the above correspondence  
for a given $\ell_{[\bb]}$-weight $\Bg$,   
namely the monomial depends on a choice of polynomials $Q_i(z)$, $R_i(z)$ and $S_i(z)$ satisfying \eqref{ratinal func ellb wt} 
as following cases. 

\begin{enumerate}
\item 
In the case where $Q_i(z)$ and $R_i(z)$ have a common divisor $(1 - \z z)$, 
we can write $Q_i(z) = (1 - \z z) Q'_i(z)$ and $R_i (z) = (1- \z z) R'_i(z)$. 
Then we have 
\begin{align}
\label{common Q R}
q^{\deg Q_i - \deg R_i} \frac{Q_i (q^{-1} z)}{Q_i(qz)} \frac{ R_i( q z)}{ R_i (q^{-1} z)} \frac{S_i(z)}{z^{b_i}}
= q^{ \deg Q'_i - \deg R'_i} \frac{Q'_i (q^{-1} z)}{Q'_i(qz)} \frac{ R'_i( q z)}{ R'_i (q^{-1} z)} \frac{S_i(z)}{z^{b_i}}. 
\end{align}
This equality corresponds to $Y_{i, \eta} Y_{i, \eta}^{-1} =1$ in monomials.

\item 
In the case where $S_i(z)$ and $Q_i(qz)$ have a common divisor $(1 - \eta z)$, 
put $\wt{Q}_i(z) = (1- q^{-1} \eta  z)^{-1} Q_i(z)$ and $\wt{S}_i(z) = q (1 - q^{-2} \eta z) (1 - \eta z)^{-1} S_i(z)$. 
Then we have 
\begin{align}
\label{common S Q} 
q^{\deg Q_i - \deg R_i} \frac{Q_i (q^{-1} z)}{Q_i(qz)} \frac{ R_i( q z)}{ R_i (q^{-1} z)} \frac{S_i(z)}{z^{b_i}} 
=
q^{\deg \wt{Q}_i - \deg R_i} \frac{\wt{Q}_i (q^{-1} z)}{\wt{Q}_i(qz)} \frac{ R_i( q z)}{ R_i (q^{-1} z)} \frac{\wt{S}_i(z)}{z^{b_i}}. 
\end{align}
This equality corresponds to $Y_{i, q^{-1} \eta} Z_{i, \eta} = \tau_{i, q} Z_{i, q^{-2} \eta}$ in monomials.  

\item 
In the case where $S_i(z)$ and $R_i (q^{-1} z)$ have a common divisor $(1 - \eta z)$, 
put $\wt{R}_i(z) = (1 - q \eta z)^{-1} R_i(z)$ and $\wt{S}_i(z) = q^{-1} (1 - q^2 \eta z) (1 - \eta z)^{-1} S_i(z)$. 
Then we have 
\begin{align*}
q^{\deg Q_i - \deg R_i} \frac{Q_i (q^{-1} z)}{Q_i(qz)} \frac{ R_i( q z)}{ R_i (q^{-1} z)} \frac{S_i(z)}{z^{b_i}} 
= q^{\deg Q_i - \deg \wt{R}_i} \frac{Q_i (q^{-1} z)}{Q_i(qz)} \frac{ \wt{R}_i( q z)}{ \wt{R}_i (q^{-1} z)} \frac{\wt{S}_i(z)}{z^{b_i}}. 
\end{align*} 
This equality corresponds to $Y_{i, q\eta}^{-1} Z_{i, \eta} = \tau_{i, q^{-1}} Z_{i, q^2 \eta}$ in monomials.  

\item 
The variable $\tau_{i,a}$ just represents  the constant term of the polynomial $S_i(z)$, 
but we would treat it as a variable when we consider the $q$-character. 
Then we also have $\tau_{i,a} \tau_{i,b} = \tau_{i, ab}$. 
\end{enumerate}

Thanks to the above arguments, 
we take an ideal $\fI$ of the polynomial ring 
$\ZZ [  Y_{i, \z}^{\pm}, Z_{i, \eta}, \tau_{i,a}]^{\z, \eta, a \in \CC^{\times}}_{1 \leq i \leq m-1} $ 
generated by 
\begin{align*}
Y_{i, q^{-1} \eta} Z_{i, \eta} - \tau_{i,q} Z_{i, q^{-2} \eta}, 
\quad 
Y_{i, q \eta}^{-1} Z_{i, \eta} - \tau_{i, q^{-1}} Z_{i, q^2 \eta}, 
\quad 
\tau_{i,a} \tau_{i,b} - \tau_{i, ab}
\end{align*}
for $1 \leq i \leq m-1$ and $\eta, a,b \in \CC^{\times}$. 

For a finite dimensional $U_{q,[\bb]}$-module $M$ and an $\ell_{[\bb]}$-weight $\Bg$ of $M$, 
we define 
\begin{align*}
X^{\Bg} = \prod_{i=1}^{m-1} \Big( 
	\tau_{i, a_i} \big( \prod_{j=1}^{k_i} Y_{i, \z_{i,j}} \big) \big( \prod_{j'=1}^{l_i} Y_{i, \xi_{i, j'}}^{-1} \big) 
 	\big( \prod_{p=1}^{b_i} Z_{i, \eta_{i,p}}  \big) \Big) 
\in \ZZ [  Y_{i, \z}^{\pm}, Z_{i, \eta}, \tau_{i,a}]^{\z, \eta, a \in \CC^{\times}}_{1 \leq i \leq m-1} / \fI
\end{align*}
by \eqref{ratinal func ellb wt} and \eqref{monomial QRS}. 
Let $M= \bigoplus_{\Bg \in \lbwt M} M_{\Bg}$ be the decomposition of $M$ into the generalized $\ell_{[\bb]}$-weight spaces, 
and we define the $q$-character of $M$ as 
\begin{align*}
\chi_{[\bb]} (M) = \sum_{\Bg \in \lbwt M} (\dim M_{\Bg}) X^{\Bg} 
\in \ZZ [  Y_{i, \z}^{\pm}, Z_{i, \eta}, \tau_{i,a}]^{\z, \eta, a \in \CC^{\times}}_{1 \leq i \leq m-1} / \fI. 
\end{align*}
Then we have the following proposition. 
\begin{prop} 
Let $U_{q,[\bb]}\cmod_f$ be the category of finite dimensional $U_{q,[\bb]}$-modules, 
and $K_0 (U_{q,[\bb]} \cmod_f)$ be its Grothendieck group. 
Then the map 
\begin{align*}
\chi_{[\bb]} : K_0 (U_{q,[\bb]} \cmod_f) 
	\ra \ZZ [  Y_{i, \z}^{\pm}, Z_{i, \eta}, \tau_{i,a}]^{\z, \eta, a \in \CC^{\times}}_{1 \leq i \leq m-1} / \fI
\quad 
(M \mapsto \chi_{[\bb]} (M) ) 
\end{align*} 
is an injective group homomorphism. 
\begin{proof}
We have already seen that the map $\chi_{[\bb]}$ is well-defined in the above argument. 
It is clear that $\chi_{[\bb]}$ is a group homomorphism, 
and the injectivity follows from \cite[Corollary 5.1]{Her23}. 
\end{proof}
\end{prop}

%%%%%%%%%%%%%%%%%%%%%%%%%%%%%%%%%%%%%%%%%%%%%%%%%%%%%%%%%%%%%%%

%%%%%%%%%%%%%%%%%%%%%%%%%%%%%%%%%%%%%%%%%%%%%%%%%%%%%%%%%%%%%%%

\section{Combinatorics} 
\label{comb}
In this section, we prepare some notation for combinatorics which are used to describe 
some finite dimensional simple $U_{q,[\bb]}(L\Fsl_m)$-modules in subsequent sections. 

\para 
For a positive integer $m$, 
let $\vL_n(m)$ be the set of $m$-tuples of non-negative integers such that the sum of entries is equal to $n$, 
namely, 
\begin{align*}
\vL_n (m) = \{ \mu = (\mu_1,\mu_2, \dots, \mu_m) \in (\ZZ_{\geq 0})^m \mid \sum_{i=1}^m \mu_i =n\}. 
\end{align*}
An element of $\vL_n(m)$ is called a composition of $n$ with $m$ parts. 

Let $P= \bigoplus_{i=1}^m \ZZ \ve_i$ be the weight lattice of the general linear Lie algebra $\Fgl_m$. 
Put $\a_i = \ve_i - \ve_{i+1}$ ($1\leq i \leq m-1$), 
then $\{\a_i \mid 1 \leq i \leq m-1\}$ gives a set of simple roots of $\Fgl_m$, and also of $\Fsl_m$. 
We regard the set $\vL_n(m)$ as a subset of the weight lattice $P$ by the injection 
$\vL_n(m) \ra P$ ($(\mu_1,\dots, \mu_m) \mapsto \sum_{i=1}^m \mu_i \ve_i$). 
We denote the dominance order on $P$ by $\geq$. Then, for $\mu,\nu \in \vL_n(m)$,  we have 
\begin{align*}
\mu \geq \nu \text{ if and only if }
\sum_{i=1}^j \mu_i \geq \sum_{i=1}^j \nu_i \text{ for all } 1 \leq j \leq m. 
\end{align*}
%%%
\para 
A partition is  a non-increasing sequence $\la =(\la_1, \la_2, \dots )$ of non-negative integers with only finitely many non-zero terms. 
The size of a partition $\la =(\la_1,\la_2, \dots)$ is the sum of all entries, and we denote it by $|\la| = \sum_{i \geq 1} \la_i$. 
The length of a partition $\la$ is the number of non-zero terms in $\la$, and we denote it by $\ell (\la)$. 
We also denote by $\la \vdash n$ if $\la$ is a partition of  size $n$. 

For positive integers $n$ and $r$, put 
\begin{align*}
\vL_{n,r}^+ = \{ \Bla = (\la^{(1)}, \la^{(2)}, \dots, \la^{(r)}) \mid \la^{(k)} \text{ is a partition } (1 \leq k \leq r), \, 
	\sum_{k=1}^r |\la^{(k)}| =n \}.
\end{align*}
An element of $\vL_{n,r}^+$ is called an $r$-partition of size $n$. 
For positiove integers $n$, $r$ and $m$, we also put 
\begin{align*}
\vL_{n,r}^+ (m) = \{\Bla=(\la^{(1)}, \la^{(2)}, \dots, \la^{(r)}) \in \vL_{n,r}^+ \mid \ell (\la^{(k)}) \leq m \, (1 \leq k \leq r)\}.
\end{align*}

In the case where $r=1$, we denote $\vL_{n,1}^+$ by $\vL_n^+$ (resp. $\vL_{n,1}^+(m)$ by $\vL_n^+(m)$) simply. 
Then, the set $\vL_n^+(m)$ is a subset of $\vL_n(m) \subset P$.
%%%%%

\para 
For $\Bla =(\la^{(1)}, \dots, \la^{(r)}) \in \vL_{n,r}^+$, put 
\begin{align*}
[\Bla] = \{(a,b,c) \in (\ZZ_{>0})^3 \mid 1 \leq b \leq \la_a^{(c)}, \, 1 \leq c \leq r\}, 
\end{align*}
and we call it the diagram of $\Bla$. 
The elements of $[\Bla]$ are called nodes. 
In the case where $r=1$, 
we usually drop the third entries of nodes since it is clear. 
We represent the diagram $[\Bla]$ by arranging the boxes in usual manner. 
For an example, if $\Bla= ((3,2), (2,2,1), (4,1))$, then we represent as 
\begin{align*}
[\Bla] = \left( \, \begin{array}{|c|c|c|} \hline \, & \, & \, \\ \hline \, & \, \\ \cline{1-2} \multicolumn{1}{c}{} \end{array} \, , 
	\, 
	\begin{array}{|c|c|c|} \hline \, & \, \\ \hline \, & \, \\ \hline \, \\ \cline{1-1} \end{array} \, , 
	\, 
	\begin{array}{|c|c|c|c|} \hline \, & \, & \, & \, \\ \hline \, \\ \cline{1-1} \multicolumn{1}{c}{} \end{array} 
	\, 
	\right). 
\end{align*}

%%%
For $\Bla \in \vL_{n,r}^+$, 
a tableau $\bT$ of shape $\Bla$ is a map 
$\bT : [\Bla] \ra \ZZ_{>0} $. 
We also call it $\Bla$-tableau simply. 
For a $\Bla$-tableau $\bT$, put 
\begin{align*}
\owt \bT = (\mu_1,\mu_2,\dots ), 
\text{ where } \mu_i = \sharp \{ x \in [\Bla] \mid \bT(x) =i \} \, (i >0), 
\end{align*}
and we call it the weight of $\bT$. 
From the definition, 
for $\Bla \in \vL_{n,r}^+$ and a $\Bla$-tableau $\bT$, 
we have $\owt \bT \in \vL_n(m) $ for some $m \in \ZZ_{>0}$. 
We can represent a $\Bla$-tableau $\bT$ by writing the number $\bT(x)$ in the corresponding box of the diagram $[\Bla]$ 
for each $x \in [\Bla]$. 
For example, if 
\begin{align}
\label{example bT} 
\bT= \left( \, \begin{array}{|c|c|c|} \hline 1 & 3 & 2 \\ \hline 2 & 4 \\ \cline{1-2} \multicolumn{1}{c}{} \end{array} \, , 
	\, 
	\begin{array}{|c|c|c|} \hline 5 & 1 \\ \hline 3 & 2 \\ \hline 2 \\ \cline{1-1} \end{array} \, , 
	\, 
	\begin{array}{|c|c|c|c|} \hline 1 & 2 & 3 & 4 \\ \hline 2 \\ \cline{1-1} \multicolumn{1}{c}{} \end{array} 
	\, 
	\right), 
\end{align}
we have $\bT((1,1,1))=1$, $\bT((1,2,1)) =3$, $\bT((1,3,1)) =2$, $\bT(2,1,1)=2$, $\bT(2,2,1)=4$, 
$\bT((1,1,2))=5$, $\bT((1,2,2))=1$, $\bT((2,1,2))=3$ and so on. 
In this example, we have 
$\owt \bT = (3, 6, 3, 2,1)$. 

For $\Bla \in \vL_{n,r}^+$ and a $\Bla$-tableau $\bT$, 
we say that $\bT$ is semi-standard  
if $\bT$ satisfies the following two conditions; 
\begin{enumerate} 
\item 
$\bT((a,b,c)) \leq \bT((a, b+1,c))$  if $ (a,b+1,c) \in [\Bla]$ for $(a,b,c) \in [\Bla]$, 

\item 
$\bT((a,b,c)) <  \bT((a+1, b,c))$  if $ (a+1,b,c) \in [\Bla]$ for $(a,b,c) \in [\Bla]$.  
\end{enumerate}
We see that the tableau $\bT$ in \eqref{example bT} is not semi-standard, 
but the tableau 
\begin{align}
\label{ex semi-std}
\bS= \left( \, \begin{array}{|c|c|c|} \hline 1 & 2 & 2 \\ \hline 2 & 4 \\ \cline{1-2} \multicolumn{1}{c}{} \end{array} \, , 
	\, 
	\begin{array}{|c|c|c|} \hline 1 & 1 \\ \hline 2 & 2 \\ \hline 3 \\ \cline{1-1} \end{array} \, , 
	\, 
	\begin{array}{|c|c|c|c|} \hline 1 & 2 & 3 & 4 \\ \hline 2 \\ \cline{1-1} \multicolumn{1}{c}{} \end{array} 
	\, 
	\right) 
\end{align}
is semi-standard. 

For $\Bla \in \vL_{n,r}^+$ and $\mu \in \vL_n(m)$, put 
\begin{align*}
\SStd (\Bla,\mu) = \{ \text{ semi-standard tableau of shape $\Bla$ with weight $\mu$} \},  
\end{align*}
and, for $\Bla \in \vL_{n,r}^+$ and $m \in \ZZ_{>0}$, put 
\begin{align*}
\SStd_m (\Bla) = \bigcup_{\mu \in \vL_n(m)} \SStd (\Bla,\mu). 
\end{align*}
By the definition of semi-standard tableaux, 
we have 
$\Bla \in \vL_{n,r}^+(m)$ if $\SStd_m (\Bla) \not= \emptyset$. 

%%%%%%%%%%%%%%%
\para 
Let $\CC [x_1,x_2,\dots, x_k]$ be the polynomial ring over $\CC$ with indeterminate variables 
$x_1,x_2,\dots, x_k$. 
For a partition $\la$ such that $\ell (\la) \leq k$,  
let $\bm_{\la}(x_1, \dots, x_k) \in \CC [x_1,\dots, x_k]^{\fS_k}$ be the monomial symmetric polynomial associated with $\la$. 
For $t \in \ZZ_{>0}$, we define symmetric polynomials $\Phi_t^{\pm} (x_1, \dots, x_k) \in \CC[x_1,\dots,x_k]^{\fS_k}$ by 
\begin{align*}
\Phi_t^{\pm} (x_1,\dots, x_k) 
= \sum_{\la \vdash t \atop \ell (\la) \leq k} (1 - q^{\mp 2})^{\ell (\la) -1} \bm_{\la} (x_1,\dots, x_k), 
\end{align*}
and we consider a generating function 
\begin{align*}
\wh{\Phi}^{\pm}_{(k)} (z) = 1 + (1- q^{\mp 2}) \sum_{ t >0} \Phi_t^{\pm} (x_1,\dots, x_k) z^t. 
\end{align*}
Then we have the following lemma.  
\begin{lem}[{\cite[Proposition 6.6]{KW21}}] 
For $k \in \ZZ_{>0}$, we have 
\begin{align*}
\wh{\Phi}_{(k)}^{\pm} (z) 
	= \frac{(1 - q^{\mp 2} x_1 z) ( 1 - q^{\mp 2} x_2 z) \dots ( 1 - q^{\mp 2} x_k z)}{ (1 - x_1 z) (1 - x_2 z) \dots ( 1 - x_k z)}. 
\end{align*}
\end{lem}

%%%
\para 
For $k,l,t \in \ZZ_{>0}$, 
we define a polynomial $\Phi_t (x_1,\dots, x_k \mid y_1,\dots, y_l) $ with indeterminate variables 
$x_1,\dots,x_k, y_1,\dots, y_l$ by 
\begin{align*}
&\Phi_t (x_1,\dots, x_k \mid y_1,\dots y_l) 
\\
&= q^{-1} \Phi_t^+ (x_1,\dots,x_k) - q \Phi_t^- (y_1, \dots, y_l) 
	- (q-q^{-1}) \sum_{p=1}^{t-1} \Phi_{t-p}^+ (x_1,\dots x_k) \Phi_p^- (y_1,\dots, y_l).
\end{align*}
For  convenience, we also put 
\begin{align*}
\begin{cases}
\Phi_t (x_1,\dots, x_k \mid 0) = q^{-1} \Phi_t^+ (x_1,\dots, x_k) & \text{ if } l=0, 
\\
\Phi_t (0 \mid y_1,\dots, y_l) = - q \Phi_t^-( y_1,\dots, y_l) & \text{ if } k=0, 
\\
\Phi_t (0 \mid 0) =0 & \text{ if } k=l=0.
\end{cases}
\end{align*}

For $k,l \in \ZZ_{\geq 0}$, we consider a generating function 
\begin{align*}
\wh{\Phi}_{(k|l)} (z) 
= 1 + (q-q^{-1}) \sum_{ t >0} \Phi_t (x_1,\dots, x_k \mid y_1,\dots, y_l) z^t. 
\end{align*}
By direct calculation, we see that 
$\wh{\Phi}_{(k)}^+(z) \wh{\Phi}_{(l)}^-(z) = \wh{\Phi}_{(k|l)} (z)$, 
and we have the following proposition. 

%%%
\begin{prop}
\label{Prop whPhiklz}
For $k,l \in \ZZ_{\geq 0}$, we have 
\begin{align*}
&\wh{\Phi}_{(k|l)} (z) 
\\
&=\frac{ (1- q^{-2} x_1 z) (1- q^{-2} x_2 z) \dots ( 1 - q^{-2} x_k z)}{ (1- x_1 z) (1 - x_2 z) \dots (1 - x_k z)} 
	\cdot 
	\frac{ ( 1- q^2 y_1 z) ( 1- q^2 y_2 z) \dots (1 - q^2 y_l z)}{ (1- y_1 z) (1 - y_2 z) \dots (1 - y_l z)}.
\end{align*}
\end{prop}
%%%%%%%%%%%%%%%%%%%%%%%%%%%%%%%%%%%%%%%%%%%%%%%%%%%%%%%%%%%%%%%

%%%%%%%%%%%%%%%%%%%%%%%%%%%%%%%%%%%%%%%%%%%%%%%%%%%%%%%%%%%%%%%

\section{Evaluation representations and their tensor products} 
In this section, we describe finite dimensional simple evaluation representations of the quantum loop algebra $U_q (L \Fsl_m)$ 
using combinatorics explicitly, 
and we also describe their tensor products under a certain separation condition. 

\para 
Let $U_q (\Fgl_m)$ be the quantum group associated with the general linear Lie algebra $\Fgl_m$, 
and we denote the Chevalley generators of $U_q (\Fgl_m)$ by $E_i$, $F_i$ ($1 \leq i \leq m-1$) and $K_j^{\pm}$ ($1 \leq j \leq m$).  
For $\z \in \CC^{\times} $, 
let $\ev_{\z} : U_q (\wh{\Fsl}_m) \ra U_q (\Fgl_m)$ be the evaluation homomorphism at $\z$ defined in \cite{Jim86}. 
More precisely, the homomorphism $\ev_{\z}$ is defined by 
\begin{align*}
\begin{split} 
& e_i \mapsto E_i, \quad f_i \mapsto F_i, \quad k_i^+ \mapsto K_i^+ K_{i+1}^- \quad (1 \leq i \leq m-1), 
\\
& k_0^+ \mapsto K_1^- K_m^+, 
\\
& e_0 \mapsto \z q^{-1} (K_1^+ K_m^+) [F_{m-1}, [F_{m-2}, \dots, [F_2, F_1]_{q^{-1}} \dots ]_{q^{-1}} ]_{q^{-1}}, 
\\
& f_0 \mapsto (-1)^m \z^{-1} q^{m-1} (K_1^- K_m^-) [E_{m-1}, [E_{m-2}, \dots, [E_2,E_1]_{q^{-1}} \dots ]_{q^{-1}}]_{q^{-1}}, 
\end{split} 
\end{align*}
where $[X,Y]_{q^{-1}} = XY - q^{-1} YX$ for $X,Y \in U_q (\Fgl_m)$. 

Recall that we regard $\vL_n^+(m)$ as a subset of the weight lattice of $\Fgl_m$, 
and an element of $\vL_n^+(m)$ is an integral dominant weight of $\Fgl_m$.  
For $\la \in \vL_n^+(m)$, let $V (\la)$ be the simple highest weight $U_q (\Fgl_m)$-module of highest weight $\la$. 
When we regard $V(\la)$ as a $U_q (\wh{\Fsl}_m)$-module through the homomorphism $\ev_{\z}$, 
we denote it by $V(\la)^{\ev_{\z}}$. 
We consider $V(\la)^{\ev_{\z}}$ as a $U_q (L\Fsl_m)$-module through the isomorphism \eqref{iso Uq whsl UqLsl}. 

It is well known that a basis of $V(\la)$ is indexed by the set $\SStd_m(\la)$. 
Thus, for $\Bla =(\la^{(1)}, \la^{(2)}, \dots, \la^{(r)}) \in \vL_{n,r}^+(m)$ and $\Bz =(\z_1, \z_2,\dots, \z_r) \in (\CC^{\times})^r$,  
we see that a basis of $V(\la^{(1)})^{\ev_{\z_1}} \otimes V(\la^{(2)})^{\ev_{\z_2}} \otimes \dots \otimes V (\la^{(r)})^{\ev_{\z_r}}$ 
is indexed by the set $\SStd_m (\Bla)$. 

In the remainder of this section, 
we give an explicit description of  the $U_q (L\Fsl_m)$-action on 
$V(\la^{(1)})^{\ev_{\z_1}} \otimes \dots \otimes V (\la^{(r)})^{\ev_{\z_r}}$ 
using notion of semi-standard tableaux under a separation condition 
\begin{align}
\label{separation cond} 
\prod_{1 \leq i < j \leq r} \prod_{ - n < k < n} (q^{2k} \z_i - \z_j) \not=0.
\end{align}

%%%
\para \label{bTi} 
For $\Bla \in \vL_{n,r}^+(m)$ and $\bT \in \SStd_m(\Bla)$, 
put $\mu_i (\bT) = \sharp \{ x \in [\Bla] \mid \bT(x) =i\}$  ($1\leq i \leq m$). 
Then we have  $\owt \bT =(\mu_1(\bT), \mu_2(\bT), \dots, \mu_m(\bT)) \in \vL_n(m)$. 

For a $\Bla$-tableau $\bT$ and $x \in [\Bla]$, 
we define $\Bla$-tableaux $\bT_x^+$ and $\bT_x^-$ by  
\begin{align*}
\bT_x^{\pm} (y) = \begin{cases} \bT(x) \pm 1 & \text{ if } y=x, \\ \bT(y) & \text{ if } y \not=x \end{cases} 
\quad (y \in [\Bla]), 
\end{align*}
where we do not consider the tableau $\bT_x^-$ if $\bT(x) =1$. 

For $x=(a,b,c) \in [\Bla]$, put 
\begin{align*}
&\fn_a (x) = (a-1, b,c), 
\quad 
\fn_b (x) = (a+1,b,c), 
\\
&\fn_l (x) = (a, b-1,c), 
\quad 
\fn_r (x) = (a, b+1,c). 
\end{align*}
Namely, $\fn_a(x)$ (resp. $\fn_b(x)$, $\fn_l(x)$, $\fn_r(x)$) is the above node 
(resp. the below node, the left node, the right node) of $x$ if it is contained in $[\Bla]$. 
For convenience, we put 
$\bT( \fn_a(x)) = - \infty$, $\bT (\fn_b (x)) = \infty$, $\bT(\fn_l(x)) = - \infty$ and $\bT ( \fn_r(x)) = \infty$ 
if $\fn_{*} (x) \not\in [\Bla]$ ($* \in \{a,b,l,r\}$) respectively.  

For $\bT \in \SStd_m(\Bla)$ and $1 \leq i \leq m$, put 
\begin{align*}
& [\bT]_i = \{ x \in [\Bla] \mid \bT(x) =i\}, 
\\
& [\bT]_i^L = \{ x \in [\Bla] \mid \bT (x) =i, \, \bT(\fn_l(x)) <i\}, 
\\
& [\bT]_i^R = \{ x \in [\Bla] \mid \bT(x) =i, \, \bT(\fn_r(x)) >i\}. 
\end{align*}
Thus, the set 
$[\bT]_i^L$ (resp. $[\bT]_i^R$) 
is a subset of $[\bT]_i$ consisting of the leftmost node $x$ (resp. rightmost node $x$) such that $\bT(x) =i$ 
in each row of $[\Bla]$. 

For $\bT \in \SStd_m(\Bla)$ and $1 \leq i \leq m-1$, we also put  
\begin{align*}
& [\bT]_{\a_i} = \{x \in [\bT]_{i+1} \mid \bT_x^- \in \SStd_m (\Bla)\}, 
\\
& [\bT]_{-\a_i} = \{ x \in [\bT]_i \mid \bT_x^+ \in \SStd_m(\Bla) \}. 
\end{align*}
From the definition of semi-standard tableau, 
we have 
\begin{align*}
& [\bT]_{\a_i} = \{ x \in [\Bla] \mid \bT(x) =i+1, \, \bT (\fn_a (x)) <i, \, \bT(\fn_l(x)) \leq i \} \subset [\bT]_{i+1}^L, 
\\
& [\bT]_{-\a_i} = \{ x \in [\Bla] \mid \bT(x) =i, \, \bT(\fn_b(x)) >i+1, \, \bT( \fn_r(x)) \geq i+1\} \subset [\bT]_i^R. 
\end{align*}
We  also have $\owt \bT_x^- = \owt \bT + \a_i$ for $x \in [\bT]_{\a_i}$,  
and $\owt \bT_x^+ = \owt \bT - \a_i$ for $x \in [\bT]_{-\a_i}$ 
in the weight lattice $P=\bigoplus_{i=1}^m \ZZ \ve_i$.

%%%
\para \label{cBz} 
We take $\Bz = (\z_1, \dots, \z_r) \in (\CC^{\times})^r$ satisfying the condition \eqref{separation cond}, and fix it. 
For $\Bla \in \vL_{n,r}^+(m)$ and $x =(a,b,c) \in [\Bla]$, 
we define the content $c_{\Bz} (x)$ of $x$ with respect to $\Bz$ by 
\begin{align*}
c_{\Bz} (x) = q^{2 (b-a)} \z_c.
\end{align*}
By the condition \eqref{separation cond}, 
we see that, 
for $x=(a,b,c), y =(a', b', c') \in [\Bla]$, 
\begin{align}
\label{czx = czy}
c_{\Bz} (x) = c_{\Bz} (y) 
\text{ if and only if } 
\begin{cases} 
	a'= a+k, 
	\\
	b'=b+k,  
	\\
	c=c' 
\end{cases} 
\text{ for some } k \in \ZZ.  
\end{align}
For $x =(a,b,c) \in [\bT]_{\a_i}$, we have $\bT(x) =i+1$ and $\bT(\fn_a (x)) <i$. 
Then the definition of semi-standard tableaux implies that 
$\bT((a +k, b+k, c)) >i+1$ and $\bT((a-k, b-k, c)) <i$ for $k \in \ZZ_{>0}$. 
By combining with \eqref{czx = czy}, 
we have 
\begin{align}
\label{czx not= czy ai}
c_{\Bz} (x) \not= c_{\Bz} (y) \text{ for any } 
x \in [\bT]_{\a_i} \text{ and } y \in [\bT]_i.
\end{align}
Similarly, we have 
\begin{align}
\label{czx not= czy -ai} 
c_{\Bz} (x) \not= c_{\Bz}(y) 
	\text{ for any } 
	x \in [\bT]_{-\a_i} \text{ and } y \in [\bT]_{i+1}. 
\end{align}

For $\bT \in \SStd_m(\Bla)$ and $x \in [\bT]_{\a_i}$ ($1 \leq i \leq m-1$), we set 
\begin{align*}
B_{\bT}^x = q^{\mu_i (\bT)} 
	\frac{ \prod_{ y \in [\bT]_i^L} \big( c_{\Bz} (x) - q^{-2} c_{\Bz} (y) \big)}
		{ \prod_{y' \in [\bT]_i^R} \big( c_{\Bz} (x) - c_{\Bz} (y') \big)}, 
\end{align*}
where we note that $\prod_{y' \in [\bT]_i^R} \big( c_{\Bz} (x) - c_{\Bz} (y') \big) \not=0$ by \eqref{czx not= czy ai}. 

For $x \in [\bT]_{-\a_i}$ ($1 \leq i \leq m-1$), we also set 
\begin{align*}
C_{\bT}^x = q^{- \mu_{i+1} (\bT)} 
	\frac{ \prod_{y' \in [\bT]_{i+1}^R} \big( c_{\Bz}(x) - q^2 c_{\Bz} (y') \big)}
		{ \prod_{ y \in [\bT]_{i+1}^L} \big( c_{\Bz} (x) - c_{\Bz} (y) \big)}, 
\end{align*}
where we note that $\prod_{ y \in [\bT]_{i+1}^L} \big( c_{\Bz} (x) - c_{\Bz} (y) \big) \not=0$ by \eqref{czx not= czy -ai}.
By direct calculation, we see that 
\begin{align}
\label{BbTx CbTx}
B_{\bT}^x = q^{\mu_i (\bT)} \prod_{y \in [\bT]_i} \frac{(c_{\Bz}(x) - q^{-2} c_{\Bz}(y))}{( c_{\Bz}(x) - c_{\Bz}(y))},  
\quad 
C_{\bT}^x = q^{-\mu_{i+1}(\bT)} \prod_{y \in [\bT]_{i+1}} \frac{ (c_{\Bz}(x) - q^2 c_{\Bz}(y))}{(c_{\Bz}(x) - c_{\Bz}(y))}. 
\end{align}

Finally, for $\bT \in \SStd_m (\Bla)$, $1 \leq i \leq m-1$ and $s \in \ZZ_{>0}$, 
we set 
\begin{align}
\label{def PhiisbT}
\begin{split}
&\Phi_{i,s} (\bT) = \Phi_s \big( c_{\Bz} (x_1), c_{\Bz} (x_2), \dots, c_{\Bz}(x_{\mu_i(\bT)}) 
	\mid c_{\Bz} (y_1), c_{\Bz} (y_2), \dots, c_{\Bz} (y_{\mu_{i+1}(\bT)}) \big), 
\\
& \Phi_{i,-s} (\bT) = \Phi_s \big( c_{\Bz} (y_1)^{-1}, c_{\Bz} (y_2)^{-1}, \dots, c_{\Bz} (y_{\mu_{i+1} (\bT)})^{-1} 
	\mid c_{\Bz} (x_1)^{-1}, c_{\Bz} (x_2)^{-1}, \dots, c_{\Bz} (x_{\mu_i(\bT)})^{-1}  \big)
\end{split}
\end{align} 
when $[\bT]_i =\{x_1,x_2,\dots, x_{\mu_i(\bT)}\}$ and $[\bT]_{i+1} = \{y_1,y_2,\dots, y_{\mu_{i+1}(\bT)}\}$. 

Then we can define the following $U_q (L\Fsl_m)$-module $\D_{\Bz} (\Bla)$. 

\begin{prop}
\label{Prop Def DBz Bla} 
Assume that $\Bz \in (\CC^{\times})^r$ satisfies the separation condition \eqref{separation cond}. 
For $\Bla \in \vL_{n,r}^+(m)$, let $\D_{\Bz} (\Bla)$ be a $\CC$-vector space 
with a basis $\{ v_{\bT} \mid \bT \in \SStd_m(\Bla)\}$. 
Then we can define the action of $U_q (L\Fsl_m)$ on $\D_{\Bz} (\Bla)$ by 
\begin{align*}
& e_{i,t} \cdot v_{\bT} 
	= \sum_{ x \in [\bT]_{\a_i}} B_{\bT}^x (q^i c_{\Bz} ( x))^t  \, v_{\bT_x^-} 
	\quad ( t \in \ZZ), 
\\
& f_{i,t} \cdot v_{\bT} 
	= \sum_{x \in [\bT]_{-\a_i}} C_{\bT}^x (q^i c_{\Bz} (x))^t \, v_{\bT_x^+} 
	\quad (t \in \ZZ),  
\\
& \psi_{i,s}^+ \cdot v_{\bT} 
	= q^{\mu_i (\bT) - \mu_{i+1} (\bT) + s \cdot i} \cdot 
		\begin{cases}
		v_{\bT} & \text{ if } s =0, 
		\\
		(q-q^{-1}) \Phi_{i,s} (\bT) \,  v_{\bT} & \text{ if } s >0, 
		\end{cases}
\\
& \psi_{i, -s}^- \cdot v_{\bT} 
	= q^{- \mu_i (\bT) + \mu_{i+1} (\bT) - s \cdot i} \cdot 
		\begin{cases}
		v_{\bT} & \text{ if } t=0, 
		\\
		(q-q^{-1}) \Phi_{i, -s} (\bT) \, v_{\bT} & \text{ if } s >0. 
		\end{cases} 
\end{align*}
\begin{proof}
It is enough to check that the action in the proposition satisfies the defining relations (U1) - (U8) of $U_q (L\Fsl_m)$, 
and 
we give details of them in \S \, \ref{section proof DBzBla}. 
\end{proof}
\end{prop}

%%%%%%%
\begin{cor}
\label{Cor BbTx not=0} 
In the setting of Proposition \ref{Prop Def DBz Bla}, 
we have the following. 
\begin{enumerate}
\item 
\begin{enumerate}
\item 
For $\bT \in \SStd_m (\Bla)$ and $x \in [\bT]_{\a_i}$, 
we have $B_{\bT}^x \not=0$. 

\item 
For $\bT \in \SStd_m(\Bla)$ and $x \in [\bT]_{-\a_i}$, 
we have $C_{\bT}^x \not=0$. 
\end{enumerate}

\item 
For $1 \leq i \leq m-1$ and $\bT \in \SStd_m (\Bla)$, we have 
\begin{align*}
\psi_i^+(z) \cdot v_{\bT} 
	&= q^{\mu_i(\bT) - \mu_{i+1} (\bT)} 
		\prod_{x \in [\bT]_i} \frac{(1 - q^{i-1} c_{\Bz}(x) q^{-1} z)}{ (1 - q^{i-1} c_{\Bz} (x) q z)} 
		\prod_{ y \in [\bT]_{i+1}} \frac{(1 - q^{i+1} c_{\Bz} (y) q z)}{( 1 - q^{i+1} c_{\Bz} (y) q^{-1} z)} v_{\bT}, 
\\
\psi_i^-(z) \cdot v_{\bT} 
	&= q^{ - \mu_i(\bT) + \mu_{i+1} (\bT)} 
		\\ & \quad \times 
		\prod_{ y \in [\bT]_{i+1}} \frac{ ( 1 -  q^{- (i +1)} c_{\Bz} (y)^{-1}  (qz)^{-1} )}{ ( 1 - q^{-(i+1)} c_{\Bz} (y)^{-1} (q^{-1} z)^{-1})}
		\prod_{ x \in [\bT]_i} \frac{(1 -  q^{-(i-1)} c_{\Bz}(x)^{-1}  (q^{-1}z)^{-1})}{(1 - q^{-(i-1)} c_{\Bz}(x)^{-1}  (qz)^{-1})}.
\end{align*}
\end{enumerate}
\begin{proof}
(\roi). 
For $x \in [\bT]_{\a_i}$, 
we have $ \bT(\fn_a(x)) = \bT_x^- (\fn_a(x)) < \bT_x^-(x) =i$ and 
$\bT(\fn_r(x)) \geq \bT(x) =i+1$ by definitions. 
These imply that $c_{\Bz}(x) \not= q^{-2} c_{\Bz}(y)$ for all $y \in [\bT]_i$, 
and we have $B_{\bT}^x \not=0$. 
Similarly, we have $C_{\bT}^x \not=0$ for $x \in [\bT]_{-\a_i}$ from 
$\bT(\fn_b(x)) = \bT_x^+(\fn_b(x)) > \bT_x^+(x) =i+1$ and $\bT(\fn_l(x)) \leq \bT(x) =i$. 

(\roii) follows from Proposition \ref{Prop whPhiklz} and \eqref{def PhiisbT}. 
\end{proof}
\end{cor}

%%%%%
\para 
For $\Bla = (\la^{(1)}, \la^{(2)}, \dots, \la^{(r)}) \in \vL_{n,r}^+(m)$, 
we define the semi-standard $\Bla$-tableaux $\bT^{\Bla}$ and $\bT_{\Bla}$ by 
\begin{align*}
\bT^{\Bla} ((a,b,c)) =a 
\text{ and } 
\bT_{\Bla} ((a,b,c)) = m - (\la^{(c)})'_{b} +a 
\quad ((a,b,c) \in [\Bla]) 
\end{align*}
respectively, 
where $(\la^{(c)})'$ is the conjugate of $\la^{(c)}$. 
Namely, $(\la^{(c)})'_b$ is the number of nodes in the $b$-th column of the diagram $[\la^{(c)}]$. 
For  example, if $m=5$ and $\Bla= ((3,2), (2,2,1), (4,1))$, then we have 
\begin{align*}
&\bT^{\Bla} 
= \left( \, \begin{array}{|c|c|c|} \hline 1 & 1 & 1 \\ \hline 2 & 2 \\ \cline{1-2} \multicolumn{1}{c}{} \end{array} \, , 
	\, 
	\begin{array}{|c|c|c|} \hline 1 & 1 \\ \hline 2 & 2 \\ \hline 3 \\ \cline{1-1} \end{array} \, , 
	\, 
	\begin{array}{|c|c|c|c|} \hline 1 & 1 & 1& 1 \\ \hline 2 \\ \cline{1-1} \multicolumn{1}{c}{} \end{array} 
	\, 
	\right) 
\text{ and } 
\\
&\bT_{\Bla} 
= \left( \, \begin{array}{|c|c|c|} \hline 4 & 4 & 5 \\ \hline 5 & 5 \\ \cline{1-2} \multicolumn{1}{c}{} \end{array} \, , 
	\, 
	\begin{array}{|c|c|c|} \hline 3 & 4 \\ \hline 4 & 5 \\ \hline 5 \\ \cline{1-1} \end{array} \, , 
	\, 
	\begin{array}{|c|c|c|c|} \hline 4 & 5 & 5 & 5 \\ \hline 5 \\ \cline{1-1} \multicolumn{1}{c}{} \end{array} 
	\, 
	\right).
\end{align*}

%%%%%%

\begin{thm}
\label{Thm DBz Bla}
Assume that $\Bz \in (\CC^{\times})^r$ satisfies the separation condition \eqref{separation cond}. 
For $\Bla =(\la^{(1)},\dots, \la^{(r)}) \in \vL_{n,r}^+(m)$, 
the $U_q (L\Fsl_m)$-module $\D_{\Bz} (\Bla)$ satisfies the following. 

\begin{enumerate}
\item 
\begin{enumerate} 
\item The vector $v_{\bT^{\Bla}}$ is a unique singular vector of $\D_{\Bz} (\Bla)$ up to scalar multiples. 

\item 
The vector $v_{\bT_{\Bla}}$ is a unique $f$-singular vector of $\D_{\Bz} (\Bla)$  up to scalar multiples. 
\end{enumerate} 

\item 
The $q$-character of $\D_{\Bz} (\Bla)$ is given by 
\begin{align*}
&\chi_q(\D_{\Bz}(\Bla)) = \sum_{\bT \in \SStd_m(\Bla)} Y^{\bT}, 
\\
&\text{ where }
Y^{\bT} = \prod_{i=1}^{m-1} 
	\Big( \prod_{x \in [\bT]_i} Y_{i, q^{i-1} c_{\Bz} (x)} \prod_{ y \in [\bT]_{i+1}} Y_{i, q^{i+1} c_{\Bz} (y)}^{-1} \Big).
\end{align*} 

\item 
Any $\ell$-weight of $\D_{\Bz} (\Bla)$ is multiplicity-free. 

\item 
The $U_q (L \Fsl_m)$-module $\D_{\Bz} (\Bla)$ is simple. 
In particular, the vector $v_{\bT^{\Bla}}$ (resp. $v_{\bT_{\Bla}}$) is an $\ell$-highest weight vector 
(resp. an $\ell$-lowest weight vector) of $\D_{\Bz} (\Bla)$. 

\item 
The Drinfeld polynomials $\bP =(P_i(z))_{1 \leq i \leq m-1}$ of $\D_{\Bz} (\Bla)$ are given by 
\begin{align*}
P_i(z) = \prod_{c=1}^r\prod_{k=1}^{\la_i^{(c)} - \la_{i+1}^{(c)}} (1 -   q^{2 \la_i^{(c)} - i+1 -2k} \z_c z)
\quad (1 \leq i \leq m-1).
\end{align*}

\item 
We have 
$\D_{\Bz} (\Bla) \cong V(\la^{(1)})^{\ev_{\z_1}}  \otimes \dots \otimes V (\la^{(r)})^{\ev_{\z_r}}$ as $U_q (L\Fsl_m)$-modules. 
\end{enumerate}

\begin{proof}
(\roi)-(a). 
By definitions, it is clear that $[\bT^{\Bla}]_{\a_i} = \emptyset$ for any $1 \leq i \leq m-1$, 
and we see that $v_{\bT^{\Bla}}$ is a singular vector. 
On the other hand, for $\bT \in \SStd_m(\Bla)$ such that $\bT\not=\bT^{\Bla}$,  
let $x =(a,b,c) \in [\Bla]$ be the topmost and leftmost node among the nodes such that $\bT(x) \not= \bT^{\Bla}(x)$. 
Then we see that 
\begin{align*}
\bT(x) > a, 
\quad 
&\bT(\fn_a (x)) = \begin{cases} a-1 & \text{ if } \fn_a (x) \in [\Bla], \\ - \infty & \text{ if } \fn_a (x) \not\in [\Bla], \end{cases} 
\text{ and }
\\
&\bT(\fn_l(x)) = \begin{cases} a & \text{ if } \fn_l(x) \in [\Bla] \\ - \infty & \text{ if } \fn_l(x) \not\in [\Bla],  \end{cases}
\end{align*} 
and we have $x \in [\bT]_{\a_i}$ if $\bT(x) =i+1$. 
We also have $B_{\bT}^x \not=0$ by Corollary \ref{Cor BbTx not=0} (\roi). 
As a consequence, we see that the vector $v_{\bT}$ is not  singular if $\bT \not= \bT^{\Bla}$, 
and $v_{\bT^{\Bla}}$ is a unique singular vector of $\D_{\Bz}(\Bla)$ up to scalar multiples. 
Similarly we can prove (\roi)-(b). 

%%%
\ps
We obtain the statement (\roii)  from Corollary \ref{Cor BbTx not=0} (\roii). 

%%%
\ps 
(\roiii). 
Thanks to the separation condition \eqref{separation cond}, 
for $x \in [\bT]_i$ and $y \in [\bT]_{i+1}$, 
we see that $q^{i-1} c_{\Bz} (x) = q^{i+1} c_{\Bz}(y)$ if and only if $y = \fn_b (x)$ since $\bT$ is semi-standard. 
Then we can check that (\roii) implies (\roiii). 

%%%
\ps 
(\roiv). 
Thanks to (\roi), 
to prove that $\D_{\Bz}(\Bla)$ is simple, 
it is sufficient to show that $\D_{\Bz} (\Bla)$ is generated by $v_{\bT^{\Bla}}$ as a $U_q (L\Fsl_m)$-module. 
Let $M$ be the $U_q(L\Fsl_m)$-submodule of $\D_{\Bz} (\Bla)$ generated by $v_{\bT^{\Bla}}$, 
and we prove that $v_{\bT} \in M$ for all $\bT \in \SStd_m(\Bla)$ 
by the induction on $\owt \bT$  
with respect to the dominance order. 
It is clear that $\owt \bT^{\Bla}$ is highest in $\{ \owt \bT \mid \bT \in  \SStd_m(\Bla)\}$, 
and $v_{\bT^{\Bla}} \in M$. 

For $\bT \in \SStd_m(\Bla)$ such that $\bT \not= \bT^{\Bla}$, 
by the proof of (\roi)-(a), 
we have $[\bT]_{\a_i} \not= \emptyset$ for some $1 \leq i \leq m-1$. 
For $x \in [\bT]_{\a_i}$, 
we have 
$\bT_x^- \in \SStd_m(\Bla)$ and $\owt \bT_x^- > \owt \bT$ 
since $\owt \bT_x^- = \owt \bT + \a_i$. 
Thus, the vector $v_{\bT_x^-}$ belongs to $M$ by the induction hypothesis. 
We also see that $x \in [\bT_x^-]_{-\a_i}$ since $\bT_x^-(x) = \bT(x) -1 =i$ and $(\bT_x^-)_{x}^+ = \bT \in \SStd_m(\Bla)$. 
Then we have 
\begin{align}
\label{fit vTx-}
f_{i,t} \cdot v_{\bT_x^-} 
= C_{\bT_x^-}^x (q^i c_{\Bz} (x))^t v_{\bT} 
	+ \sum_{ y \in [\bT_x^-]_{-\a_i} \setminus \{x\}} C_{\bT_x^-}^y (q^i c_{\Bz} (y))^t v_{(\bT_x^-)_y^+} 
	\in M.
\end{align}
We remark that the vectors $v_{\bS}$ ($\bS \in \SStd_m(\Bla)$) are $\ell$-weight vevtors, and 
any $\ell$-weight of $\D_{\Bz} (\Bla)$ is multiplicity-free by (\roiii). 
Then \eqref{fit vTx-} implies that $v_{\bT} \in M$, 
and we conclude (\roiv). 

%%%
\ps 
(\rov). 
By the above argument,  
the vector $v_{\bT^{\Bla}}$ is an $\ell$-highest weight vector.
By Corollary \ref{Cor BbTx not=0} (\roii), we have 
\begin{align*}
&\psi_i^+(z) \cdot v_{\bT^{\Bla}} 
\\
&= q^{\mu_i(\bT^{\Bla}) - \mu_{i+1} (\bT^{\Bla})}  
	\prod_{x \in [\bT^{\Bla}]_i} \frac{ (1 -  q^{i-1} c_{\Bz} (x)  q^{-1}z )}{ ( 1 - q^{i-1}c_{\Bz} (x)  qz)} 
	\prod_{y \in [\bT^{\Bla}]_{i+1} } \frac{ (1 - q^{i+1} c_{\Bz} (y) q z)}{ (1 - q^{i+1} c_{\Bz} (y) q^{-1} z)}
	\, v_{\bT}
\\
&= q^{\sum_{c=1}^r \la_i^{(c)} - \sum_{c=1}^r \la_{i+1}^{(c)}} 
	\\ & \qquad \times 
	\Big( \prod_{c=1}^r \prod_{b=1}^{\la_i^{(c)} } 
		\frac{ (1 - q^{i-1} q^{2 ( b - i)} \z_c q^{-1} z) }{ (1 - q^{i-1} q^{2 (b-i)} \z_c q z)} \Big) 
	\Big( \prod_{c=1}^r \prod_{b=1}^{\la_{i+1}^{(c)}} 
		\frac{ (1 - q^{i+1} q^{2 (b - (i+1))} \z_c q z)}{ ( 1 - q^{i+1} q^{2 ( b - (i+1))} \z_c q^{-1} z)} \Big) 
	\, v_{\bT} 
\\
&= \prod_{c=1}^r \Big( q^{\la_i^{(c)} - \la_{i+1}^{(c)}} \prod_{k=1}^{\la_i^{(c)} - \la_{i+1}^{(c)}} 
	\frac{ (1 - q^{i-1} q^{2 ( \la_i^{(c)} - (k-1) -i)} \z_c q^{-1} z)}{ (1 - q^{i-1} q^{2 ( \la_i^{(c)} - (k-1) -i)} \z_c q z)} 
	\Big) v_{\bT} 
\\
&= \prod_{c=1}^r \Big( q^{\la_i^{(c)} - \la_{i+1}^{(c)}} \prod_{k=1}^{\la_i^{(c)} - \la_{i+1}^{(c)}} 
	\frac{ (1 - q^{ 2 \la_i^{(c)} - i +1- 2 k  } \z_c  \, (q^{-1} z))}{ (1 - q^{ 2 \la_i^{(c)} - i +1- 2 k  } \z_c \, ( q z))} 
	\Big) v_{\bT}. 
\end{align*}
Then we obtain (\rov).  
 
%%%
\ps 
(\rovi).  
Let $v_0^{(c)}$ be an $\ell$-highest weight vector of $V(\la^{(c)})^{\ev_{\z_c}}$ ($1 \leq c \leq r$), 
and the submodule of $V(\la^{(1)})^{\ev_{\z_1}} \otimes \dots \otimes V (\la^{(r)})^{\ev_{\z_r}}$ 
generated by $v_0^{(1)} \otimes \dots \otimes v_0^{(r)}$ is an $\ell$-highest weight module 
with an $\ell$-highest weight vector $v_0^{(1)} \otimes \dots \otimes v_0^{(r)}$. 
By \cite{CP94a, FM02}, 
it is known that the Drinfeld polynomials $(P^{(c)}_i(z))_{1 \leq i \leq m-1}$ of the evaluation module $V(\la^{(c)})^{\ev_{\z_c}}$ 
($1 \leq c \leq r$) 
are given by 
\begin{align*}
P_i^{(c)} (z) = \prod_{k=1}^{\la_i^{(c)} - \la_{i+1}^{(c)}} ( 1 - q^{i-1} q^{2 (\la_i^{(c)} - (k-1) -i)} \z_{c} z ) 
\quad (1 \leq i \leq m-1). 
\end{align*}
Thus, there exists a surjective homomorphism from $U_q (L\Fsl_m) \cdot (v_0^{(1)} \otimes \dots \otimes v_0^{(r)})$ 
to $\D_{\Bz}(\Bla)$ by (\roiv) and (\rov). 
By comparing dimensions, the homomorphism gives an isomorphism 
$V(\la^{(1)})^{\ev_{\z_1}} \otimes \dots \otimes V(\la^{(r)})^{\ev_{\z_r}} \cong \D_{\Bz} (\Bla)$. 
\end{proof}
\end{thm}

%%%%%
\remarks\ 
\begin{enumerate}
\item 
For the Yangian $Y(\Fgl_m)$ (resp. the quantum Yangian $Y_q (\Fgl_m)$), 
the action of $Y(\Fgl_m)$ (resp. $Y_q (\Fgl_m)$) on an evaluation module is described explicitly 
by using the Gelfand-Zetlin basis in  \cite{Mol94, NT94}. 

\item 
The $q$-character of the evaluation module $V(\la)^{\ev_{\z}}$ 
has already computed in \cite[Lemma 4.7]{FM02} by using Gelfand-Zetlin schemes. 

\item 
The simplicity of the tensor product $V (\la^{(1)})^{\ev_{\z_1}} \otimes \dots \otimes V(\la^{(r)})^{\ev_{\z_r}} $ 
of evaluation modules under the separation condition \eqref{separation cond} also follows from 
\cite[Theorem 2.2]{MTZ04} and \cite[Theorem1.1]{Her10}. 
\end{enumerate}

%%%%%%%%%%%%%%%%%%%%%%%%%%%%%%%%%%%%%%%%%%%%%%%%%%%%%%%%%%%%%%%

%%%%%%%%%%%%%%%%%%%%%%%%%%%%%%%%%%%%%%%%%%%%%%%%%%%%%%%%%%%%%%%

\section{Some finite dimensional simple modules of $U_{q,[\bb]}(L \Fsl_m)$} 

In this section, 
we discuss  $U_{q,[\bb]}(L \Fsl_m)$-modules 
obtained from the $U_q (L\Fsl_m)$-module $\D_{\Bz} (\Bla)$ by tensoring a one-dimensional $U_{q,[\bb]}$-module. 
Then we describe some finite dimensional simple $U_{q,[\bb]}$-modules explicitly, and give their $q$-characters. 

\begin{lem}
\label{Lemma sing LBb M}
For $\bb, \bc \in (\ZZ_{\geq 0})^{m-1}$, 
we consider a one-dimensional $U_{q,[\bb]}$-module $L_{\Bb} = \CC w$ ($\Bb \in \BB_{[\bb]}$) 
and a finite dimensional simple $U_{q,[\bc]}$-module $M$. 
Let $v_0 $ be an $\ell_{[\bc]}$-highest weight vector of $M$, 
and $\wt{v}_0 $ be an $\ell_{[\bc]}$-lowest weight vector of $M$. 

\begin{enumerate}
\item 
For the $U_{q,[\bb + \bc]}$-module $L_{\Bb} \otimes M$ through the homomorphism 
$\D_{\bb, \bc} : U_{q,[\bb + \bc]} \ra U_{q,[\bb]} \otimes U_{q, [\bc]}$, 
the vector $w \otimes v_0$ is a unique singular vector of $L_{\Bb} \otimes M$ up to scalar multiples. 
In particular, any submodule of $L_{\Bb} \otimes M$ contains $w \otimes v_0$, 
and $L_{\Bb} \otimes M$ has a unique simple submodule. 

\item 
For the $U_{q,[\bb+ \bc]}$-module $M \otimes L_{\Bb}$ through the homomorphism 
$\D_{\bc, \bb} : U_{q, [\bb+\bc]} \ra U_{q, [\bc]} \otimes U_{q, [\bb]}$, 
the vector $\wt{v}_0 \otimes w$ is a unique $f$-singular vector of $M \otimes L_{\Bb}$ up to scalar multiples. 
In particular, any submodule of $M \otimes L_{\Bb} $ contains $ \wt{v}_0 \otimes w $, 
and $M \otimes L_{\Bb}$  has a unique simple submodule. 
\end{enumerate}

\begin{proof}
By Proposition \ref{Prop Ddc} together with \eqref{def LBb}, 
we see that the element $e_{i,t} \in U_{q,[\bb + \bc]}$  ($1 \leq i \leq m-1$, $t \in \ZZ$) acts on $L_{\Bb} \otimes M$ 
by $1 \otimes e_{i,t}$. 
Then the simplicity of $M$ implies (\roi). 
Similarly, we see that the element $f_{i,t} \in U_{q,[\bb+\bc]}$ acts on $M \otimes L_{\Bb}$ by $f_{i,t} \otimes 1$, 
and we obtain (\roii). 
\end{proof} 
\end{lem}

%%%
\para \label{Sing Bla} 
For $\Bla \in \vL_{n,r}^+(m)$ and $\bT \in \SStd_m(\Bla)$, 
put 
\begin{align*}
&\SStd_m(\Bla; \leq \bT) = \{ \bS \in \SStd_m(\Bla) \mid \bS (x) \leq \bT(x) \text{ for all } x \in [\Bla]\}, 
\\
&\SStd_m(\Bla ; \geq \bT)= \{ \bS \in \SStd_m(\Bla) \mid \bS (x) \geq \bT(x) \text{ for all } x \in [\Bla]\}.
\end{align*}
  
For $\Bla \in \vL_{n,r}^+(m)$ and  $\Bb = (\b_{i, - b_i +s})_{1 \leq i \leq m-1}^{0 \leq s \leq b_i}\in \BB_{[\bb]}$, 
we define subsets  
$\Sing_m^{\Bb} (\Bla)$ and $\fSing_m^{\Bb} (\Bla)$  of $\SStd_m (\Bla)$ by 
\begin{align*}
&\Sing_m^{\Bb} (\Bla) 
\\
&= \{ \bT\in \SStd_m(\Bla) \mid 
	\prod_{x \in [\bT]_{\a_i}}  (1 - q^i c_{\Bz} (x) z)  \text{ divides } S_i(z)  \text{ for each }  1 \leq i \leq m-1  \}, 
\\
&\fSing_m^{\Bb} (\Bla) 
\\
&= \{ \bT\in \SStd_m(\Bla) \mid 
	\prod_{x \in [\bT]_{-\a_i}} (1 - q^i c_{\Bz} (x) z)  \text{ divides } S_i(z)  \text{ for each }  1 \leq i \leq m-1 \} 
\end{align*}
respectively, 
where 
$S_i(z) = \sum_{s=0}^{b_i} \b_{i, - b_i +s} z^s$ ($1 \leq i \leq m-1$). 
It is clear that $\bT^{\Bla} \in \Sing_m^{\Bb} (\Bla)$ (resp. $\bT_{\Bla} \in \fSing_m^{\Bb} (\Bla)$) 
since $[\bT^{\Bla}]_{\a_i} =\emptyset$ (resp. $[\bT_{\Bla}]_{-\a_i} = \emptyset$) for all $1 \leq i \leq m-1$. 
%%%

\begin{thm}
\label{Thm simple LBb Dla} 
Assume that $\Bz \in (\CC^{\times})^r$ satisfies the separation condition \eqref{separation cond}. 
For  $\Bb \in \BB_{[\bb]}$ ($\bb \in (\ZZ_{\geq 0})^{m-1}$), let $L_{\Bb} = \CC w$ 
be the corresponding one-dimensional $U_{q,[\bb]}$-module. 
For $\Bla \in \vL_{n,r}^+(m)$, we consider the $U_{q,[\bb]}$-module 
$L_{\Bb} \otimes \D_{\Bz} (\Bla)$
through the homomorphism $\D_{\bb, \mathbf{0}}$. 
Then we have the following. 
\begin{enumerate}
\item 
The vector $w \otimes v_{\bT}$ ($\bT \in \SStd_m(\Bla)$) is an $\ell_{[\bb]}$-weight vector. 
Moreover, any $\ell_{[\bb]}$-weight of $L_{\Bb} \otimes \D_{\Bz} (\Bla)$ is multiplicity-free. 

\item 
For $\bT \in \SStd_m (\Bla)$, 
the vector $w \otimes v_{\bT}$ is $f$-singular 
if and only if 
$\bT \in \fSing_m^{\Bb} (\Bla)$. 

\item 
For $\bT \in \fSing_m^{\Bb} (\Bla)$, 
let $L_{\Bb} \otimes \D_{\Bz}(\Bla ; \leq \bT)$  be the subspace of $L_{\Bb} \otimes \D_{\Bz}(\Bla)$ 
spanned by $\{ w \otimes v_{\bS} \mid \bS \in \SStd_m (\Bla ; \leq \bT) \}$. 
Then $L_{\Bb} \otimes \D_{\Bz} (\Bla ; \leq \bT)$ is a $U_{q,[\bb]}$-submodule of $L_{\Bb} \otimes \D_{\Bz} (\Bla)$. 

\item 
There exists a unique semi-standard tableau $\bT_{\Bla}^{\Bb}  \in \fSing_m^{\Bb} (\Bla)$ 
such that $\fSing_m^{\Bb} (\Bla) \cap \SStd_m (\Bla ; \leq \bT_{\Bla}^{\Bb}) = \{ \bT_{\Bla}^{\Bb} \}$. 

\item 
Put $\D_{\Bz}^{\Bb} (\Bla) = L_{\Bb} \otimes \D_{\Bz}(\Bla ;\leq \bT_{\Bla}^{\Bb})$. 
Then $\D_{\Bz}^{\Bb} (\Bla)$ is a simple $U_{q,[\bb]}$-module with 
an $\ell_{[\bb]}$-highest weight vector $w \otimes v_{\bT^{\Bla}}$ 
and an $\ell_{[\bb]}$-lowest weight vector $w \otimes v_{\bT_{\Bla}^{\Bb}}$. 

\item 
The Drinfeld polynomials $(\bS=(S_i(z))_{1 \leq i \leq m-1}, \bP = (P_i(z))_{1 \leq i \leq m-1})$ 
of $\D_{\Bz}^{\Bb} (\Bla)$ are given by 
\begin{align*}
S_i(z) = \sum_{s=0}^{b_i} \b_{i, - b_i +s} z^s, 
\quad 
P_i(z) = \prod_{c=1}^r\prod_{k=1}^{\la_i^{(c)} - \la_{i+1}^{(c)}} (1 -   q^{2 \la_i^{(c)} - i+1 -2k} \z_c z)
\end{align*}
for $1 \leq i \leq m-1$. 

\item 
Let $S_i(z) = a_i (1 - \eta_{i,1} z) (1 - \eta_{i,2} z) \dots (1 - \eta_{i, b_i} z)$ for $1 \leq i \leq m-1$. 
Then, the $q$-character of $\D_{\Bz}^{\Bb} (\Bla)$ is given by 
\begin{align*}
&\chi_{[\bb]} (\D_{\Bz}^{\Bb} (\Bla)) 
= \sum_{\bT \in \SStd_m(\Bla ; \leq \bT_{\Bla}^{\Bb}) }  X^{\bT}, 
\\
&\text{where } 
X^{\bT} = \prod_{i=1}^{m-1} \Big( \tau_{i, a_i} 
	\prod_{x \in [\bT]_i} Y_{i, q^{i-1} c_{\Bz}(x)} 
	\prod_{y \in [\bT]_{i+1}} Y_{i, q^{i+1} c_{\Bz} (y)}^{-1} 
	\prod_{k=1}^{b_i} Z_{i, \eta_{i,k}} \Big). 
\end{align*}
\end{enumerate}

\begin{proof}
(\roi) follows from Proposition \ref{Prop Ddc} and Theorem \ref{Thm DBz Bla} 
together with  calculations in the proof of Proposition \ref{Prop lb-weight}. 

%%%
\ps 
(\roii). 
By (\roi), 
the vector $w \otimes v_{\bT}$ is an eigenvector for the action of $h_{i, \pm 1}$. 
Then, for $1 \leq i \leq m-1$ and $t \in \ZZ$, 
we see that $f_{i,t} \cdot (w \otimes v_{\bT}) =0$ if and only if $f_{i,0} \cdot (w \otimes v_{\bT})=0$ 
thanks to \eqref{ei spm1 h e}.  
On the other hand,  
we have 
\begin{align*}
f_{i,0} \cdot (w \otimes v_{\bT}) 
&= (\psi_{i,0}^- \cdot w ) \otimes (f_{i, 0} \cdot v_{\bT}) + \sum_{k=1}^{b_i} (\psi_{i, -k}^{+} \cdot w) \otimes ( f_{i, k} \cdot v_{\bT}) 
\end{align*}
by Proposition \ref{Prop Ddc}. 
Then, by direct computation using \eqref{def LBb} and Proposition \ref{Prop Def DBz Bla}, 
we have 
\begin{align}
\label{fi0 w otimes vbT}  
f_{i,0} \cdot (w \otimes v_{\bT}) 
= \sum_{ x \in [\bT]_{-\a_i}} C_{\bT}^x \big( \sum_{k=0}^{b_i} \b_{i,- k} (q^i c_{\Bz} (x))^k \big) \, w \otimes v_{\bT_x^+}, 
\end{align}
where we note that $C_{\bT}^x \not=0$ for $x \in [\bT]_{-\a_i}$. 
Then, we see that 
$f_{i,0} \cdot (w \otimes v_{\bT}) =0$ if and only if 
$q^i c_{\Bz} (x)$ is a root of the polynomial $\sum_{k=0}^{b_i} \b_{i,-k} z^k$
for all $x \in [\bT]_{-\a_i}$.
On the other hand, 
we see that $\sum_{k=0}^{b_i} \b_{i,-k} z^k = z^{b_i} S_i(z^{-1})$. 
Then, we have 
\begin{align*}
S_i(z) &= z^{ b_i} \b_{i, - b_i} (z^{-1} - \g_1) (z^{-1} - \g_2) \dots (z^{-1} - \g_{b_i}) 
\\
&= \b_{i, - b_i} (1 -\g_1 z) (1 - \g_2 z) \dots (1 - \g_{b_i} z)
\end{align*}
if $\sum_{k=0}^{b_i} \b_{i,-k} z^k = \b_{i, - b_i} (z- \g_1) (z -\g_2) \dots (z - \g_{b_i})$. 
Moreover, for $x,x' \in [\bT]_{-\a_i}$, 
we see that $c_{\Bz} (x) \not= c_{\Bz} (x')$  if $x \not=x'$. 
As a consequence, we obtain (\roii). 

%%%
\ps 
(\roiii).  
It is clear that  the space $L_{\Bb} \otimes \D_{\Bz} (\Bla ; \leq \bT)$ is invariant under the actions  
of $e_{i,t}$, $\psi_{i, - b_i +s}^+$, $\psi_{i, -s}^-$ ($1 \leq i \leq m-1$, $t \in \ZZ$, $s \in \ZZ_{\geq 0}$) 
by Proposition \ref{Prop Ddc} and Proposition \ref{Prop Def DBz Bla}. 
For $\bS \in \SStd_m(\Bla; \leq \bT)$, we have 
\begin{align}
\label{fio w otimes vbS} 
f_{i,0} \cdot (w \otimes v_{\bS}) 
= \sum_{ x \in [\bS]_{-\a_i}} C_{\bS}^x \big( \sum_{k=0}^{b_i} \b_{i,- k} (q^i c_{\Bz} (x))^k \big) \, w \otimes v_{\bS_x^+}. 
\end{align} 
On the other hand, 
if $\bS_x^+ \not\in \SStd_m(\Bla; \leq \bT)$ for $x \in [\bS]_{-\a_i}$,  
then we have 
\begin{align*}
\bT(x) = \bS(x) =i, 
\quad 
\bT(\fn_b(x)) \geq \bS (\fn_b (x)) > i+1 
\text{ and } 
\bT(\fn_r (x)) \geq \bS(\fn_r(x)) \geq i+1, 
\end{align*}
and we conclude that $x \in [\bT]_{-\a_i}$. 
In this case, the value 
$q^i c_{\Bz}(x)$ is a root of the polynomial $\sum_{k=0}^{b_i} \b_{i,-k} z^k$ 
by the argument in (\roii) since $\bT \in \fSing_m^{\Bb} (\Bla)$. 
By combining with \eqref{fio w otimes vbS}, 
we have $f_{i,0} \cdot (w \otimes v_{\bS}) \in L_{\Bb} \otimes \D_{\Bz} (\Bla ; \leq \bT)$ 
for $\bS \in \SStd_m (\Bla ; \leq \bT)$. 
Then, we see that the space $L_{\Bb} \otimes \D_{\Bz}(\Bla; \leq \bT)$ is also invariant under the actions of 
$f_{i,t}$ ($1 \leq i \leq m-1$, $t \in \ZZ$) by the induction using \eqref{ei spm1 h e}. 

%%%
\ps 
(\roiv).   
By Lemma \ref{Lemma sing LBb M}, 
$L_{\Bb} \otimes \D_{\Bz} (\Bla)$ has a unique simple submodule, 
and we denote it by $M_0$. 
Then, there exists a unique  semi-standard tableau $\bT_0  \in \fSing_m^{\Bb} (\Bla)$ 
such that $w \otimes v_{\bT_0}$ is an $\ell_{[\bb]}$-lowest weight vector of $M_0$ 
by (\roi) and (\roii).  

If $\fSing_m^{\Bb} (\Bla) \cap \SStd_m(\Bla ; \leq \bT_0)$ contains a semi-standard tableau $\bT_1$ 
which is different from $\bT_0$, 
then a simple submodule $M_1$ of $L_{\Bb} \otimes \D_{\Bz} (\Bla ; \leq \bT_1)$ 
does not contain the vector $w \otimes v_{\bT_0}$ since $\bT_1 (x) < \bT_0 (x)$ for some $x \in [\Bla]$,  
 and we have $M_1 \not= M_0$.  
This contradicts  the uniqueness of the simple submodule $M_0$ of $L_{\Bb} \otimes \D_{\Bz} (\Bla)$. 
Thus, we have $\fSing_m^{\Bb} (\Bla) \cap \SStd_m(\Bla ; \leq \bT_0) = \{\bT_0\}$. 

Similarly, if a semi-standard tableau $\bT_2 \in \fSing_m^{\Bb} (\Bla)$ which is different from $\bT_0$ 
satisfies the condition $\fSing_m^{\Bb} (\Bla) \cap \SStd_m(\Bla ; \leq \bT_2) = \{ \bT_2\}$, 
then we see that a simple submodule $M_2$ of $L_{\Bb} \otimes \D_{\Bz} (\Bla ; \leq \bT_2)$  
does not contain the vector $w \otimes v_{\bT_0}$, 
and we have $M_2 \not=M_0$.  
This is a contradiction. 
As a consequence, we conclude (\roiv). 

%%% 
\ps 
(\rov).  
In order to prove (\rov), 
it is enough to show that $w \otimes v_{\bT} \in M_0$ for all $\bT \in \SStd_m (\la ; \leq \bT_{\Bla}^{\Bb})$. 
We show it by the induction on $\owt \bT$ with respect to the dominance order. 
It is clear that $ w \otimes v_{\bT_{\Bla}^{\Bb}} \in M_0$,  
and we see that 
$\owt \bT_{\Bla}^{\Bb}$ is lowest in $\{ \owt \bT \mid \bT \in \SStd_m(\Bla ; \leq \bT_{\Bla}^{\Bb} ) \}$. 

Suppose that  $\bT \in \SStd_m (\Bla ; \leq \bT_{\Bla}^{\Bb})$ is an element different from $\bT_{\Bla}^{\Bb}$. 
Then, $\bT$ is not $f$-singular by (\roiv), 
and there exists a node $x \in [\bT]_{-\a_i}$ for some $1 \leq i \leq m-1$ 
such that the value $q^i c_{\Bz}(x)$ is not a root of $\sum_{k=0}^{b_i} \b_{i,-k} z^k$ 
by the argument in (\roii). 
This implies  that $\bT_{x}^+ \in \SStd_m (\Bla ; \leq \bT_{\Bla}^{\Bb})$ 
by the argument in (\roiii). 
Since $\owt \bT_x^+ = \owt \bT - \a_i$, we see that $\owt \bT_x^+ < \owt \bT$, 
and we have $w \otimes v_{\bT_x^+} \in M_0$ by the induction hypothesis. 
Note that $x \in [\bT_x^+]_{\a_i}$, and we also have 
\begin{align*}
e_{i,0} \cdot (w \otimes v_{\bT_x^+} )  
= B_{\bT_x^+}^x w \otimes v_{\bT}  
	+ \sum_{y \in [\bT_x^+]_{\a_i} \setminus \{x\} } B_{\bT_x^+}^y  w \otimes  v_{(\bT_x^+)_y^-}
\end{align*}
by Proposition \ref{Prop Ddc} and  Proposition \ref{Prop Def DBz Bla}.  
Since $w \otimes v_{\bT_x^+} \in M_0$ and $B_{\bT_x^+}^x \not=0$,  
we have $w \otimes v_{\bT} \in M_0$ thanks to (\roi). 

%%% 
\ps  
The statement (\rovi) follows from \eqref{hw simple Uqb} together with Theorem \ref{Thm DBz Bla}, 
and the statement  (\rovii) follows from the calculations in the proof of Proposition \ref{Prop lb-weight}  
with Theorem \ref{Thm DBz Bla}.
\end{proof}
\end{thm} 

%%%
\begin{cor}
\label{Cor simple LBb otimes D}
We consider the same setting as in Theorem \ref{Thm simple LBb Dla}. 
Then, we have that 
the $U_{q,[\bb]}$-module $L_{\Bb} \otimes \D_{\Bz} (\Bla)$ is simple if and only if 
$\fSing_m^{\Bb} (\Bla) = \{ \bT_{\Bla}\}$. 
\end{cor} 
%%%%% 

The following theorem is a version of Theorem \ref{Thm simple LBb Dla} which we swap the tensor product. 

\begin{thm}
\label{Thm D Bla otimes LBb}
Under the same setting as in Theorem \ref{Thm simple LBb Dla}, 
for $\Bla \in \vL_{n,r}^+(m)$, 
we consider the $U_{q,[\bb]}$-module 
$\D_{\Bz} (\Bla) \otimes L_{\Bb} $
through the homomorphism $\D_{\mathbf{0}, \bb}$. 
Then we have the following. 
\begin{enumerate}
\item 
The vector $v_{\bT} \otimes w$ ($\bT \in \SStd_m(\Bla)$) is an $\ell_{[\bb]}$-weight vector. 
Moreover, any $\ell_{[\bb]}$-weight of $\D_{\Bz} (\Bla) \otimes L_{\Bb} $ is multiplicity-free. 

\item 
For $\bT \in \SStd_m (\Bla)$, 
the vector $ v_{\bT} \otimes w$ is singular 
if and only if 
$\bT \in \Sing_m^{\Bb} (\Bla)$. 

\item 
For $\bT \in \Sing_m^{\Bb} (\Bla)$, 
let $\D_{\Bz}(\Bla ; \geq \bT) \otimes L_{\Bb} $  be the subspace of $\D_{\Bz}(\Bla) \otimes L_{\Bb} $ 
spanned by $\{  v_{\bS} \otimes w \mid \bS \in \SStd_m (\Bla ; \geq \bT) \}$. 
Then $\D_{\Bz} (\Bla ; \geq \bT) \otimes L_{\Bb} $ is a $U_{q,[\bb]}$-submodule of $\D_{\Bz} (\Bla) \otimes L_{\Bb} $. 

\item 
There exists a unique semi-standard tableau $\bT^{\Bla}_{\Bb}  \in \Sing_m^{\Bb} (\Bla)$ 
such that \break 
$\Sing_m^{\Bb} (\Bla) \cap \SStd_m (\Bla ; \geq \bT^{\Bla}_{\Bb}) = \{ \bT^{\Bla}_{\Bb} \}$. 

\item 
Put $\D_{\Bz}^{\Bb*} (\Bla) = \D_{\Bz}(\Bla ; \geq \bT^{\Bla}_{\Bb}) \otimes L_{\Bb} $. 
Then $\D_{\Bz}^{\Bb *} (\Bla)$ is a simple $U_{q,[\bb]}$-module with 
an $\ell_{[\bb]}$-highest weight vector $ v_{\bT^{\Bla}_{\Bb}} \otimes w$ 
and an $\ell_{[\bb]}$-lowest weight vector $w \otimes v_{\bT_{\Bla}}$. 

\item 
The Drinfeld polynomials $(\wt{\bS}=( \wt{S}_i(z))_{1 \leq i \leq m-1}, \wt{\bP} = (\wt{P}_i(z))_{1 \leq i \leq m-1})$ 
of $\D_{\Bz}^{\Bb*} (\Bla)$ are given by 
\begin{align*}
&\wt{S}_i(z) = q^{\mu_i (\bT_{\Bb}^{\Bla}) - \mu_{i+1} (\bT_{\Bb}^{\Bla}) - \sharp [\bT_{\Bb}^{\Bla}]_i^{\ddag}} 
	\cdot S_i(z) \cdot 
	\frac{ \prod_{y \in [\bT_{\Bb}^{\Bla}]_{i+1}^{\dag R}} (1 - q^{i+2} c_{\Bz}(y) z)}
		{\prod_{ y \in  [\bT_{\Bb}^{\Bla}]_{\a_i}} (1 - q^i c_{\Bz} (y) z)}, 
\\
& \wt{P}_i(z) = \prod_{x \in [\bT_{\Bb}^{\Bla}]_i^{\ddag}} (1 - q^{i-1} c_{\Bz} (x) z) 
\end{align*}
for $1 \leq i \leq m-1$, where $S_i(z) = \sum_{s=0}^{b_i} \b_{i, - b_i +s} z^s$, 
\begin{align*}
&[\bT_{\Bb}^{\Bla}]_{i+1}^\dag =\{ (a,b,c) \in [\bT_{\Bb}^{\Bla}]_{i+1} \mid (a,b-k,c) \in [\bT_{\Bb}^{\Bla}]_{\a_i} 
	\text{ for some } k \in \ZZ_{\geq 0} \}, 
\\
&[\bT_{\Bb}^{\Bla}]_{i+1}^{\dag L} = [\bT_{\Bb}^{\Bla}]_{i+1}^{\dag} \cap [\bT_{\Bb}^{\Bla}]_{i+1}^L 
\quad ([\bT_{\Bb}^{\Bla}]_{i+1}^{\dag L} = [\bT_{\Bb}^{\Bla}]_{\a_i} \text{ by definition}), 
\\
&
[\bT_{\Bb}^{\Bla}]_{i+1}^{\dag R} = [\bT_{\Bb}^{\Bla}]_{i+1}^{\dag} \cap [\bT_{\Bb}^{\Bla}]_{i+1}^R,  
\\
&[\bT_{\Bb}^{\Bla}]_i^{\ddag} 
= [\bT_{\Bb}^{\Bla}]_i \setminus \{ \fn_a (y) \mid y \in [\bT_{\Bb}^{\Bla}]_{i+1} \setminus [\bT_{\Bb}^{\Bla}]_{i+1}^{\dag}\}. 
\end{align*}  

\item 
Let $S_i(z) = a_i (1 - \eta_{i,1} z) (1 - \eta_{i,2} z) \dots (1 - \eta_{i, b_i} z)$ for $1 \leq i \leq m-1$. 
Then, the $q$-character of $\D_{\Bz}^{\Bb*} (\Bla)$ is given by 
\begin{align*}
&\chi_{[\bb]} (\D_{\Bz}^{\Bb *} (\Bla)) 
= \sum_{\bT \in \SStd_m(\Bla ; \geq \bT^{\Bla}_{\Bb}) }  X^{\bT}, 
\\
&\text{where } 
X^{\bT} = \prod_{i=1}^{m-1} \Big( \tau_{i, a_i} 
	\prod_{x \in [\bT]_i} Y_{i, q^{i-1} c_{\Bz}(x)} 
	\prod_{y \in [\bT]_{i+1}} Y_{i, q^{i+1} c_{\Bz} (y)}^{-1} 
	\prod_{k=1}^{b_i} Z_{i, \eta_{i,k}} \Big). 
\end{align*}
\end{enumerate}

\begin{proof}
The statements except (\rovi) are proven in a similar way as in the proof of Theorem \ref{Thm simple LBb Dla}. 
We only give a proof of (\rovi). 
For $1 \leq i \leq m-1$ and $s \in \ZZ_{\geq 0}$, 
let $\g_{i,s}^+ $ be the complex number determined by 
$\psi_{i,s}^+ \cdot v_{\bT_{\Bb}^{\Bla}} = \g_{i,s}^+ v_{\bT_{\Bb}^{\Bla}}$ in $\D_{\Bz} (\Bla)$.  
Then, by direct calculations using Proposition \ref{Prop Ddc} and \eqref{def LBb}, we have  
\begin{align}
\label{psii+z v bTbbBla w} 
\psi^+_{i}(z) \cdot (v_{\bT_{\Bb}^{\Bla}} \otimes w) 
=  z^{-b_i} \big( \sum_{k=0}^{b_i}   \b_{i, - b_i +k} z^{ k}  \big)
	\big( \sum_{s \geq 0} \g_{i, s'}^+ z^s \big) 
	v_{\bT_{\Bb}^{\Bla}} \otimes w.  
\end{align}
Then,  Corollary \ref{Cor BbTx not=0} (\roii) implies that  
\begin{align}
\label{lbwt BBb Bla} 
\begin{split}
&z^{-b_i} \big( \sum_{k=0}^{b_i}   \b_{i, - b_i +k} z^{ k}  \big)
	\big( \sum_{s' \geq 0} \g_{i, s'}^+ z^s \big)  
\\
&= \frac{S_i(z)}{z^{b_i}}  \cdot 
	q^{\mu_i (\bT_{\Bb}^{\Bla}) - \mu_{i+1} (\bT_{\Bb}^{\Bla})} 
	\prod_{x \in [\bT_{\Bb}^{\Bla}]_i}  \frac{(1 - q^{i-1} c_{\Bz} (x) q^{-1} z)}{ (1 - q^{i-1} c_{\Bz}(x) q z)} 
	\prod_{y \in [\bT_{\Bb}^{\Bla} ]_{i+1}} \frac{ (1 - q^{i+1} c_{\Bz} (y) q z )}{ (1 - q^{i+1} c_{\Bz} (y) q^{-1} z )}. 
\end{split}
\end{align}
We also have 
\begin{align}
\label{decom bTBbBla i+1} 
\begin{split}
&\prod_{y \in [\bT_{\Bb}^{\Bla} ]_{i+1}} \frac{ (1 - q^{i+1} c_{\Bz} (y) q z )}{ (1 - q^{i+1} c_{\Bz} (y) q^{-1} z )} 
\\
&= \prod_{y \in [\bT_{\Bb}^{\Bla} ]_{i+1}^{\dag} } \frac{ (1 - q^{i+1} c_{\Bz} (y) q z )}{ (1 - q^{i+1} c_{\Bz} (y) q^{-1} z )} 
	\cdot 
	\prod_{y \in [\bT_{\Bb}^{\Bla}]_{i+1} \setminus [\bT_{\Bb}^{\Bla}]_{i+1}^{\dag} } 
		\frac{ (1 - q^{i+1} c_{\Bz} (y) q z )}{ (1 - q^{i+1} c_{\Bz} (y) q^{-1} z )}, 
\end{split}
\end{align}
and 
\begin{align}
\label{prod bBb Bla i+1 dag}
\prod_{y \in [\bT_{\Bb}^{\Bla} ]_{i+1}^{\dag} } \frac{ (1 - q^{i+1} c_{\Bz} (y) q z )}{ (1 - q^{i+1} c_{\Bz} (y) q^{-1} z )}  
= \frac{ \prod_{y \in [\bT_{\Bb}^{\Bla}]_{i+1}^{\dag R}} (1 - q^{i+2} c_{\Bz}(y) z)}
		{\prod_{ y \in  [\bT_{\Bb}^{\Bla}]_{\a_i}} (1 - q^i c_{\Bz} (y) z)}. 
\end{align}

On the other hand, 
for $y \in ([\bT_{\Bb}^{\Bla}]_{i+1} \setminus [\bT_{\Bb}^{\Bla}]_{i+1}^{\dag}) \cap [\bT_{\Bb}^{\Bla}]_{i+1}^L $, 
we see that 
$\bT_{\Bb}^{\Bla}(\fn_a(y))=i$ and $ \bT_{\Bb}^{\Bla} (\fn_l(\fn_a(y))) < i$.  
This implies that $\bT_{\Bb}^{\Bla}(\fn_a(y)) = i$ for all $y \in [\bT_{\Bb}^{\Bla}]_{i+1} \setminus [\bT_{\Bb}^{\Bla}]_{i+1}^{\dag}$ 
thanks to the definition of semi-standard tableaux. 
Therefore,  we have 
\begin{align}
\label{prod bTBbBla i setminus i+1} 
\begin{split}
&\prod_{x \in [\bT_{\Bb}^{\Bla}]_i}  \frac{(1 - q^{i-1} c_{\Bz} (x) q^{-1} z)}{ (1 - q^{i-1} c_{\Bz}(x) q z)} 
	\cdot \prod_{y \in [\bT_{\Bb}^{\Bla}]_{i+1} \setminus [\bT_{\Bb}^{\Bla}]_{i+1}^{\dag} } 
		\frac{ (1 - q^{i+1} c_{\Bz} (y) q z )}{ (1 - q^{i+1} c_{\Bz} (y) q^{-1} z )}
\\
&= \prod_{x \in [\bT_{\Bb}^{\Bla}]_i^{\ddag} }  \frac{(1 - q^{i-1} c_{\Bz} (x) q^{-1} z)}{ (1 - q^{i-1} c_{\Bz}(x) q z)}, 
\end{split}
\end{align}
where we note that $c_{\Bz} (\fn_a (y)) = q^2 c_{\Bz} (y)$ 
for $y \in [\bT_{\Bb}^{\Bla}]_{i+1} \setminus [\bT_{\Bb}^{\Bla}]_{i+1}^{\dag}$. 
Then, the equations 
\eqref{psii+z v bTbbBla w}, 
\eqref{lbwt BBb Bla}, \eqref{decom bTBbBla i+1}, \eqref{prod bBb Bla i+1 dag} and \eqref{prod bTBbBla i setminus i+1}
imply that 
\begin{align*}
&\psi_{i}^+(z) \cdot (v_{\bT_{\Bb}^{\Bla}} \otimes w) 
\\
&= q^{\mu_i (\bT_{\Bb}^{\Bla}) - \mu_{i+1} (\bT_{\Bb}^{\Bla}) - \sharp [\bT_{\Bb}^{\Bla}]_i^{\ddag}} 
	\frac{ S_i(z)}{z^{b_i}} 
	\cdot \frac{ \prod_{y \in [\bT_{\Bb}^{\Bla}]_{i+1}^{\dag R}} (1 - q^{i+2} c_{\Bz}(y) z)}
		{\prod_{ y \in  [\bT_{\Bb}^{\Bla}]_{\a_i}} (1 - q^i c_{\Bz} (y) z)} 
	\\ & \qquad 
	\times 
	\prod_{x \in [\bT_{\Bb}^{\Bla}]_i^{\ddag} }  \frac{(1 - q^{i-1} c_{\Bz} (x) q^{-1} z)}{ (1 - q^{i-1} c_{\Bz}(x) q z)} 
	\, 
	v_{\bT_{\Bb}^{\Bla}} \otimes w, 
\end{align*}
where we note that 
the polynomial 
$\prod_{y \in  [\bT_{\Bb}^{\Bla}]_{\a_i}} (1 - q^i c_{\Bz} (y) z)$ 
divides $S_i(z)$ since $\bT_{\Bb}^{\Bla} \in \Sing_m^{\Bb} (\Bla) $. 
\end{proof}
\end{thm}

%%%
\begin{cor}
\label{Cor simple D otimes LBb}
We consider the same setting as in Theorem \ref{Thm simple LBb Dla}. 
Then, we have that 
the $U_{q,[\bb]}$-module $ \D_{\Bz} (\Bla) \otimes L_{\Bb}$ is simple if and only if 
$\Sing_m^{\Bb} (\Bla) = \{ \bT^{\Bla}\}$. 
\end{cor} 

%%%%%%%%%%%%
\example 
Let $m=3$, $r=1$, $n=3$ and $\la =(2,1) \in \vL_3^+(3)$. 
The set $\SStd_m(\la)$ consits of 
\begin{align*}
&\bT_1= \left( \, \begin{array}{|c|c|} \hline 1 & 1 \\ \hline 2 \\ \cline{1-1} \end{array}\, \right), 
\quad 
\bT_2= \left( \, \begin{array}{|c|c|} \hline 1 & 1 \\ \hline 3 \\ \cline{1-1} \end{array}\, \right), 
\quad 
\bT_3= \left( \, \begin{array}{|c|c|} \hline 1 & 2 \\ \hline 2 \\ \cline{1-1} \end{array}\, \right), 
\quad 
\bT_4= \left( \, \begin{array}{|c|c|} \hline 1 & 2 \\ \hline 3 \\ \cline{1-1} \end{array}\, \right), 
\\
&\bT_5= \left( \, \begin{array}{|c|c|} \hline 1 & 3 \\ \hline 2 \\ \cline{1-1} \end{array}\, \right), 
\quad 
\bT_6= \left( \, \begin{array}{|c|c|} \hline 1 & 3 \\ \hline 3 \\ \cline{1-1} \end{array}\, \right), 
\quad 
\bT_7= \left( \, \begin{array}{|c|c|} \hline 2 & 2 \\ \hline 3 \\ \cline{1-1} \end{array}\, \right), 
\quad 
\bT_8= \left( \, \begin{array}{|c|c|} \hline 2 & 3 \\ \hline 3 \\ \cline{1-1} \end{array}\, \right), 
\end{align*}
where $\bT^{\la} = \bT_1$ and $\bT_{\la} = \bT_8$.  
We also have 
\begin{align*}
c_{\z} ((1,1))= \z, \quad c_{\z} ((1,2)) = q^2 \z  \text{ and } c_{\z} ((2,1)) = q^{-2} \z. 
\end{align*}
Moreover, we see that 
\begin{align*}
&\begin{array}{llll}
[\bT_1]_{-\a_1} = \{(1,2)\}, 
&
[\bT_2]_{-\a_1} = \{(1,2)\}, 
&
[\bT_3]_{-\a_1} = \emptyset, 
&
[\bT_4]_{-\a_1} =\{(1,1)\}, 
\\[0.5em]
[\bT_5]_{-\a_1} = \emptyset, 
& 
[\bT_6]_{-\a_1} = \{(1,1)\}, 
& 
[\bT_7]_{-\a_1} = \emptyset, 
& 
[\bT_8]_{-\a_1} = \emptyset, 
\end{array}
\\[1em]
&\begin{array}{llll}
[\bT_1]_{-\a_2} = \{ (2,1)\}, 
&
[\bT_2]_{-\a_2} = \emptyset, 
&
[\bT_3]_{-\a_2} = \{ (1,2), (2,1)\}, 
&
[\bT_4]_{-\a_2} = \{(1,2)\}, 
\\[0.5em]
[\bT_5]_{-\a_2} = \{(2,1)\}, 
& 
[\bT_6]_{-\a_2} = \emptyset, 
& 
[\bT_7]_{-\a_2} = \{(1,2)\}, 
& 
[\bT_8]_{-\a_2} = \emptyset. 
\end{array}
\\[1em]
&\begin{array}{llll}
[\bT_1]_{\a_1} = \emptyset, 
&
[\bT_2]_{\a_1} = \emptyset, 
&
[\bT_3]_{\a_1} = \{(1,2)\}, 
&
[\bT_4]_{\a_1} = \{(1,2)\}, 
\\[0.5em]
[\bT_5]_{\a_1} = \emptyset, 
& 
[\bT_6]_{\a_1} = \emptyset, 
& 
[\bT_7]_{\a_1} = \{(1,1)\}, 
& 
[\bT_8]_{\a_1} = \{(1,1)\}, 
\end{array}
\\[1em]
&\begin{array}{llll}
[\bT_1]_{\a_2} = \emptyset, 
&
[\bT_2]_{\a_2} = \{(2,1)\}, 
&
[\bT_3]_{\a_2} = \emptyset, 
&
[\bT_4]_{\a_2} = \{(2,1)\}, 
\\[0.5em]
[\bT_5]_{\a_2} = \{(1,2)\}, 
& 
[\bT_6]_{\a_2} = \{ (1,2), (2,1)\}, 
& 
[\bT_7]_{\a_2} = \emptyset, 
& 
[\bT_8]_{\a_2} = \{(1,2)\}. 
\end{array}
\end{align*}

Let $\bb=(1,0)$, $\b_{1,-1}=1 $, $\b_{1,0}= - q \z$ and $\b_{2,0}=1$. 
We have $S_1(z) = (1 - q \z  z)$ and $S_2(z)=1$. 
In this case, we see that 
\begin{align*}
&\fSing_3^{\Bb} (\la) = \{ \bT_6, \bT_8\}, 
\\
&\SStd_3 (\la ; \leq \bT_6) = \{ \bT_1, \bT_2, \bT_3, \bT_4, \bT_5, \bT_6\}, 
\quad 
\SStd_3 (\la ; \leq \bT_8) = \SStd_3 (\la), 
\end{align*}
Then, we have 
$\bT_{\la}^{\Bb} = \bT_6$. 
The Drinfeld polynomials and the $q$-character of $\D_{\Bz}^{\Bb} (\la)$ are given by 
\begin{align*}
&S_1(z) = (1 - q\z  z), \quad P_1(z) = (1- q^{2} \z z), 
\quad 
S_2(z) =1, \quad P_2(z) = (1 - q^{-1} \z z), 
\\
& \chi_{[\bb]} (\D_{\Bz}^{\Bb}(\la)) 
= X^{\bT_1} + X^{\bT_2} + X^{\bT_3} + X^{\bT_4} + X^{\bT_5} + X^{\bT_6} 
\\& \hspace{5em} 
=  
	\big( Y_{1, q^2 \z}   Y_{2, q^{-1} \z} 
		+ Y_{1, \z} Y_{1, q^2 \z} Y_{2, q z}^{-1} 
		+ Y_{1, q^4 \z}^{-1} Y_{2, q^3 \z} Y_{2, q^{-1} \z} 
		\\ & \hspace{8em} 
		+ Y_{1, \z} Y_{1, q^4 \z}^{-1} Y_{2, q^3 \z} Y_{2, q \z}^{-1} 
		+ Y_{2, q^{-1} \z} Y_{2, q^5 \z}^{-1} 
		+ Y_{1, \z} Y_{2, q^5 \z}^{-1} Y_{2, q \z}^{-1} 
	\big) Z_{1, q \z}. 
\end{align*}
We also see that 
\begin{align*}
& \Sing_3^{\Bb} (\la) = \{ \bT_1, \bT_7\}, 
\\
& \SStd_3 (\la ; \geq \bT_1) = \SStd_3(\la), 
\quad 
\SStd_3 (\la ; \geq \bT_7) = \{ \bT_7, \bT_8\}. 
\end{align*}
Then, we have $\bT_{\Bb}^{\la} = \bT_7$, 
\begin{align*}
&[\bT_{\Bb}^{\la}]_2^{\dag} = \{ (1,1), (1,2)\}, 
\quad 
[\bT_{\Bb}^{\la}]_2^{\dag L} = \{(1,1)\}, 
\quad 
[\bT_{\Bb}^{\la}]_2^{\dag R} =\{(1,2)\}, 
\quad 
[\bT_{\Bb}^{\la}]_1^{\ddag} = \emptyset, 
\\
&[\bT_{\Bb}^{\la}]_3^{\dag} = [\bT_{\Bb}^{\la}]_3^{\dag L} = [\bT_{\Bb}^{\la}]_3^{\dag R} = \emptyset,  
\quad 
[\bT_{\Bb}^{\la}]_2^{\ddag} = \{ (1,2)\}.
\end{align*}
The Drinfeld polynomials and the $q$-character of $\D_{\Bz}^{\Bb *}(\la)$ are given by 
\begin{align*}
&\wt{S}_1(z) = q^{-2 } (1 - q^5 \z z), 
\quad 
\wt{P}_1(z)= 1, 
\quad 
\wt{S}_2(z) =1, 
\quad 
\wt{P}_2(z) = (1 - q^3 \z z), 
\\
& \chi_{[\bb]} (\D_{\Bz}^{\Bb*}(\la)) 
	= X^{\bT_7} + X^{\bT_8} 
	\\ & \hspace{5.5em} 
	= \big( Y_{1, q^2 \z}^{-1} Y_{1, q^4 \z}^{-1} Y_{2, q^3 \z} 
			+ Y_{1, q^2\z}^{-1}  Y_{2, q^5 \z}^{-1}  
	\big) Z_{1, q \z} 
	\\ & \hspace{5.5em} 
	= \ta_{1, q^{-2}} \big( Y_{2, q^3 \z} +   Y_{1, q^4 \z} Y_{2, q^5 \z}^{-1} \big) Z_{1, q^5 \z}. 
\end{align*}

\remark 
For representations of shifted Yangians, 
a claim  corresponding to Lemma \ref{Lemma sing LBb M} is stated in \cite[Theorem 4.8]{HZ24}. 
A claim corresponding to Corollary \ref{Cor simple LBb otimes D} and Corollary \ref{Cor simple D otimes LBb} 
is also stated in \cite[Corollary 5.10]{HZ24}.  

%%%%%%%%%%%%%%%%%%%%%%%%%%%%%%%%%%%%%%%%%%%%%%%%%%%%%%%%%%%%%%%

%%%%%%%%%%%%%%%%%%%%%%%%%%%%%%%%%%%%%%%%%%%%%%%%%%%%%%%%%%%%%%%
\section{A proof of Proposition \ref{Prop Def DBz Bla}}
\label{section proof DBzBla}
In this section,  
we check that the action of $U_q (L\Fsl_m)$ on $\D_{\Bz} (\Bla)$ given by Proposition \ref{Prop Def DBz Bla} 
satisfies the defining relations (U1) - (U8) of $U_q (L\Fsl_m)$. 
In order to check them, we prepare some technical lemmas as follows.
We obtain the following two lemmas immediately from definitions. 

%%%
\begin{lem}
\label{Lemma bTx- ai}
For $\Bla \in \vL_{n,r}^+(m)$, $\bT \in \SStd_m(\Bla)$ and $1 \leq i,j \leq m-1$, 
we have the following. 
\begin{enumerate} 
\item 
For $x \in [\bT]_{\a_j}$,  
\begin{align*}
[\bT_x^-]_{\a_i} 
	&= \begin{cases}
	\big( [\bT]_{\a_i} \setminus \{x\} \big) \cup \{ \fn_r(x)\} 
		& \text{ if } j=i \text{ and } \fn_r(x) \in [\bT_x^-]_{\a_i}, 
	\\ 
	[\bT]_{\a_i} \setminus \{x \} & \text{ if } j=i \text{ and } \fn_r (x) \not\in [\bT_x^-]_{\a_i}, 
	\\
	[\bT]_{\a_i} \cup \{\fn_b(x)\} & \text{ if } j=i-1 \text{ and } \fn_b (x) \in [\bT_x^-]_{\a_i}, 
	\\ 
	[\bT]_{\a_i} \cup \{x\} & \text{ if } j=i+1 \text{ and } x \in [\bT_x^-]_{\a_i},  
	\\
	[\bT]_{\a_i} & \text{ otherwise}, 
	\end{cases}
\\
[\bT_x^-]_{-\a_i} 
	&= \begin{cases}
		\big( [\bT]_{-\a_i} \setminus \fn_l(x) \big) \cup \{x\} 
			& \text{ if } j=i \text{ and } \fn_l(x) \in [\bT]_{-\a_i}, 
		\\
		[\bT]_{-\a_i} \cup \{x\} 
			& \text{ if } j=i \text{ and } \fn_l(x) \not\in [\bT]_{-\a_i}, 
		\\
		[\bT]_{-\a_i} \setminus \{x\} & \text{ if } j=i-1 \text{ and } x \in [\bT]_{-\a_i}, 
		\\
		[\bT]_{-\a_i} \setminus \{\fn_a(x)\} & \text{ if } j=i+1 \text{ and } \fn_a (x) \in [\bT]_{-\a_i}, 
		\\ 
		[\bT]_{-\a_i} & \text{ otherwise}. 
	\end{cases}
\end{align*}

\item 
For $x \in [\bT]_{-\a_j}$,  
\begin{align*}
[\bT_x^+]_{-\a_i} 
	&= \begin{cases}
		\big( [\bT]_{-\a_i} \setminus \{x\} \big) \cup \{ \fn_l (x) \} 
			& \text{ if } j=i \text{ and } \fn_l(x) \in [\bT_x^+]_{-\a_i}, 
		\\
		[\bT]_{-\a_i} \setminus \{x \} 
			& \text{ if } j=i \text{ and } \fn_l(x) \not\in [\bT_x^+]_{-\a_i}, 
		\\
		[\bT]_{-\a_i} \cup \{x\} 
			& \text{ if } j=i-1 \text{ and } x \in [\bT_x^+]_{-\a_i}, 
		\\
		[\bT]_{-\a_i} \cup \{ \fn_a (x)\} 
			& \text{ if } j=i+1 \text{ and } \fn_a(x) \in [\bT_x^+]_{-\a_i}, 
		\\
		[\bT]_{-\a_i} 
			& \text{ otherwise}, 
	\end{cases}
\\
[\bT_x^+]_{\a_i} 
	&= \begin{cases}
	\big( [\bT]_{\a_i} \setminus \{\fn_r(x) \} \big) \cup \{x\} 
		& \text{ if } j=i \text{ and } \fn_r (x) \in [\bT]_{\a_i}, 
	\\
	[\bT]_{\a_i} \cup \{ x\} 
		& \text{ if } j=i \text{ and } \fn_r (x) \not\in [\bT]_{\a_i}, 
	\\
	[\bT]_{\a_i} \setminus \{\fn_b (x) \} 
		& \text{ if } j=i-1 \text{ and } \fn_b(x) \in [\bT]_{\a_i}, 
	\\
	[\bT]_{\a_i} \setminus \{x\} 
		& \text{ if } j=i+1 \text{ and } x \in [\bT]_{\a_i}, 
	\\
	[\bT]_{\a_i} 
		& \text{ otherwise}.
	\end{cases}
\end{align*}
\end{enumerate}
\end{lem} 

%%%

\begin{lem}
\label{Lemma Tx- iL} 
For $\Bla \in \vL_{n,r}^+(m)$, $\bT \in \SStd_m(\Bla)$ and $1 \leq i,j \leq m-1$, 
we have the following. 
\begin{enumerate} 
\item 
For $x \in [\bT]_{\a_j}$,  
\begin{align*}
[\bT_x^-]_i^L 
	&= \begin{cases}
		[\bT]_i^L \cup\{x\} & \text{ if } j=i \text{ and } \bT(\fn_l(x))< i, 
		\\
		\big( [\bT]_i^L \setminus \{x\} \big) \cup \{ \fn_r (x) \} 
			& \text{ if } j=i-1 \text{ and } \bT(\fn_r(x))=i, 
		\\
		[\bT]_i^L \setminus \{x\} 
			& \text{ if } j=i-1 \text{ and } \bT(\fn_r(x))>i, 
		\\
		[\bT]_i^L 
			& \text{ otherwise}, 
	\end{cases}
\\
[\bT_x^-]_i^R 
	&= \begin{cases}
		\big( [\bT]_i^R \setminus \{\fn_l(x) \} \big) \cup \{x\} 
			& \text{ if } j=i \text{ and } \bT(\fn_l(x)) =i, 
		\\
		[\bT]_i^R \cup \{x\} 
			& \text{ if } j=i \text{ and } \bT(\fn_l(x)) < i, 
		\\
		[\bT]_i^R \setminus \{x\} 
			& \text{ if } j=i-1 \text{ and } \bT(\fn_r(x))>i, 
		\\
		[\bT]_i^R 
			& \text{ otherwise}, 
	\end{cases}
\end{align*}
where we note that 
$\bT(\fn_l(x)) \leq i$ if $j=i$, and $\bT (\fn_r(x)) \geq i$  if $j=i-1$ since $x \in [\bT]_{\a_j} $.

\item 
For $x \in [\bT]_{-\a_j}$,  
\begin{align*}
[\bT_x^+]_i^L 
	&= \begin{cases}
		[\bT]_i^L \setminus \{x\} 
			& \text{ if }j=i \text{ and } \bT(\fn_l(x)) <i, 
		\\
		\big( [\bT]_i^L \setminus \{\fn_r(x)\} \big) \cup \{x\} 
			& \text{ if } j=i-1 \text{ and } \bT(\fn_r(x)) =i, 
		\\
		[\bT]_i^L \cup \{x\} 
			& \text{ if } j=i-1 \text{ and } \bT(\fn_r(x)) >i, 
		\\
		[\bT]_i^L 
			& \text{ otherwise},  
	\end{cases}
\\
[\bT_x^+]_i^R 
	& = \begin{cases}
	\big( [\bT]_i^R \setminus \{x\} \big) \cup \{ \fn_l(x) \} 
		& \text{ if } j=i \text{ and } \bT (\fn_l(x)) =i, 
	\\
	[\bT]_i^R \setminus \{x\} 
		& \text{ if } j=i \text{ and } \bT (\fn_l(x)) < i, 
	\\
	[\bT]_i^R \cup \{x\} 
		& \text{ if } j=i-1 \text{ and } \bT(\fn_r(x))>i, 
	\\
	[\bT]_i^R 
		& \text{ otherwise}, 
	\end{cases}
\end{align*}
where we note that 
$\bT(\fn_l(x)) \leq i$ if $j=i$, and $\bT(\fn_r(x)) \geq i$ if $j=i-1$ since $x \in [\bT]_{-\a_j}$. 
\end{enumerate} 
\end{lem} 

%%%

By direct calculations using Lemma \ref{Lemma Tx- iL}, we have the following three lemmas.   

\begin{lem}
\label{Lemma BbTx- y} 
For $\Bla \in \vL_{n,r}^+(m)$, $\bT \in \SStd_m(\Bla)$ and $1 \leq i,j \leq m-1$, 
we have the following. 
\begin{enumerate} 
\item 
For $x \in [\bT]_{\a_j}$ and $y \in [\bT]_{\a_i} \cap [\bT_x^-]_{\a_i}$,  
we have $B_{\bT_x^-}^y = B_{\bT}^y D_{\bT_x^-}^y$, 
where 
\begin{align*}
D_{\bT_x^-}^y = \begin{cases}
	\dis \frac{ q c_{\Bz} (y) - q^{-1} c_{\Bz} (x) }{ c_{\Bz} (y) - c_{\Bz} (x)} 
		& \text{ if } j=i,  
	\\[1em] 
	\dis \frac{c_{\Bz} (y) - c_{\Bz} (x)}{ q c_{\Bz} (y) - q^{-1} c_{\Bz} (x)} 
		& \text{ if } j= i-1, 
	\\[1em] 
	1 & \text{ if } j \not= i, i-1.
	\end{cases}
\end{align*}

\item 
For $x \in [\bT]_{-\a_j}$ and $y \in [\bT]_{-\a_i} \cap [\bT_x^+]_{-\a_i}$, 
we have  $C_{\bT_x^+}^y = C_{\bT}^y D_{\bT_x^+}^y$, 
where 
\begin{align*}
D_{\bT_x^+}^y 
= \begin{cases}
	\dis \frac{q^{-1} c_{\Bz} (y) - q c_{\Bz} (x)}{ c_{\Bz}(y) - c_{\Bz} (x)} 
		& \text{ if } j=i, 
	\\[1em] 
	\dis \frac{ c_{\Bz} (y) - c_{\Bz} (x)}{ q^{-1} c_{\Bz} (y) - q c_{\Bz} (x) } 
		& \text{ if } j=i+1, 
	\\[1em] 
	1 & \text{ if } j \not=i, i+1.
\end{cases}
\end{align*}
\end{enumerate}
\end{lem} 

%%%
\begin{lem} 
\label{Lemma wtDbTxy}
For $\Bla \in \vL_{n,r}^+(m)$, $\bT \in \SStd_m(\Bla)$, $1 \leq i,j \leq m-1$, 
$x \in [\bT]_{\a_i}$ and $y \in [\bT]_{- \a_j}$, 
put 
\begin{align*}
\wt{D}_{\bT}^{x,y} 
	= \begin{cases}
		\dis \frac{ c_{\Bz} (y) - c_{\Bz} (x)}{ q c_{\Bz} (y) - q^{-1} c_{\Bz} (x)} 
			& \text{ if } j=i, 
		\\[1em] 
		\dis \frac{ q c_{\Bz} (y) - q^{-1} c_{\Bz} (x)}{ c_{\Bz}(y) - c_{\Bz} (x)} 
			& \text{ if } j=i-1, 
		\\[1em] 
		1 & \text{ if } j \not=i, i-1.
	\end{cases}
\end{align*} 
Then, we have the following.  
\begin{enumerate} 
\item 
If $x \in [\bT]_{\a_i} \cap [\bT_y^+]_{\a_i}$, 
	we have $B_{\bT_y^+}^x = B_{\bT}^x \wt{D}_{\bT}^{x,y}$. 

\item 
If $ y \in [\bT]_{-\a_j} \cap [\bT_x^-]_{-\a_j}$, 
	we have $C_{\bT_x^-}^y =C_{\bT}^y \wt{D}_{\bT}^{x,y}$. 
\end{enumerate} 
\end{lem}

%%%
\begin{lem} 
\label{Lemma BbT x+ x} 
For $\Bla \in \vL_{n,r}^+(m)$, $\bT \in \SStd_m(\Bla)$ and $1 \leq i \leq m-1$, 
we have the following. 
\begin{enumerate} 
\item 
For $x \in [\bT]_{-\a_i}$, we have $ x \in [\bT_x^+]_{\a_i}$ and 
\begin{align*}
B_{\bT_x^+}^x = q^{\mu_i(\bT) -1} 
	\prod_{y \in [\bT]_i \setminus \{x\}} \frac{ c_{\Bz} (x) - q^{-2} c_{\Bz} (y)}{c_{\Bz} (x) - c_{\Bz} (y)}. 
\end{align*}

\item 
For $x \in [\bT]_{\a_i}$, we have $x \in [\bT_x^-]_{-\a_i}$ and  
\begin{align*}
C_{\bT_x^-}^x = q^{-\mu_{i+1} (\bT) +1} 
	\prod_{ y \in [\bT]_{i+1} \setminus \{x\}} \frac{ c_{\Bz} (x) - q^2 c_{\Bz} (y)}{ c_{\Bz} (x) - c_{\Bz} (y)}.  
\end{align*}
\end{enumerate}
\end{lem}

%%%%%
\para 
For $\Bla \in \vL_{n,r}^+(\bm)$, $\bT \in \SStd_m(\Bla)$ and $1 \leq i \leq m-1$, 
we define a rational function 
\begin{align}
\label{Def fFibT z} 
\fF_{i,\bT}(z) = \prod_{ x \in [\bT]_i} \frac{ (1 - q^{-2} c_{\Bz} (x) q^i z)}{ (1 - c_{\Bz} (x) q^i z)} 
				\prod_{y \in [\bT]_{i+1}} \frac{(1 - q^2 c_{\Bz} (y) q^i z)}{( 1 - c_{\Bz} (y) q^i z )}. 
\end{align}
Put $\ti{z} = z^{-1}$ and $\ti{\fF}_{i, \bT}(\ti{z}) = \fF_{i,\bT} (\ti{z}^{-1})$. 
Then, we have 
\begin{align*}
\ti{\fF}_{i,\bT} (\ti{z}) 
&= \prod_{ x \in [\bT]_i} \frac{ (1 - q^{-2} c_{\Bz} (x) q^i \ti{z}^{-1})}{ (1 - c_{\Bz} (x) q^i \ti{z}^{-1})} 
				\prod_{y \in [\bT]_{i+1}} \frac{(1 - q^2 c_{\Bz} (y) q^i \ti{z}^{-1})}{( 1 - c_{\Bz} (y) q^i \ti{z}^{-1} )}
\\
&=q^{-2 \mu_i(\bT) + 2 \mu_{i+1} (\bT)} 
	\prod_{ x \in [\bT]_i} \frac{ (1 - q^{2} c_{\Bz} (x)^{-1} q^{-i} \ti{z})}{ (1 - c_{\Bz} (x)^{-1} q^{- i} \ti{z})} 
				\prod_{y \in [\bT]_{i+1}} \frac{(1 - q^{-2} c_{\Bz} (y)^{-1} q^{ - i} \ti{z})}{( 1 - c_{\Bz} (y)^{-1} q^{-i} \ti{z} )}. 
\end{align*}
This implies that 
\begin{align}
\label{tifF i bT z-1}
\ti{\fF}_{i,\bT} (z^{-1}) 
= q^{-2 \mu_i(\bT) + 2 \mu_{i+1} (\bT)} 
	\prod_{ x \in [\bT]_i} \frac{ (1 - q^{2} c_{\Bz} (x)^{-1} q^{-i} z^{-1} )}{ (1 - c_{\Bz} (x)^{-1} q^{- i} z^{-1} )} 
				\prod_{y \in [\bT]_{i+1}} \frac{(1 - q^{-2} c_{\Bz} (y)^{-1} q^{ - i} z^{-1} )}{( 1 - c_{\Bz} (y)^{-1} q^{-i} z^{-1} )}. 
\end{align}

We denote the residue of a rational function $F(z)$ at $z=c$ by $\Res_{z=c} (F(z))$. 
Then, we have 
\begin{align}
\label{fFibT z Res}
\fF_{i,\bT}(z) = \sum_{t \in \ZZ} \Res_{z=0} (z^{- t-1} \fF_{i,\bT}(z)) z^t.  
\end{align}
We also have 
$\ti{\fF}_{i, \bT} (\ti{z}) =\sum_{t \in \ZZ} \Res_{\ti{z}=0} (\ti{z}^{- t-1} \ti{\fF}_{i,\bT}(\ti{z})) \ti{z}^t$, 
where we remark that 
\begin{align*}
\Res_{z=\infty} (z^{-t-1} \fF_{i,\bT}(z)  dz )
= \Res_{\ti{z}=0} (\ti{z}^{t+1} \ti{\fF}_{i,\bT}(\ti{z}) ( - \frac{d \ti{z}}{\ti{z}^2}) )
= - \Res_{\ti{z}=0} (\ti{z}^{t-1} \ti{\fF}_{i,\bT} (\ti{z})  d \ti{z}). 
\end{align*}
This implies that 
\begin{align}
\label{fFibTz Res infty}
\fF_{i,\bT}(z) = \ti{\fF}_{i,\bT} (\ti{z})= - \sum_{t \in \ZZ} \Res_{z=\infty} (z^{-t-1} \fF_{i,\bT}(z) )  z^t. 
\end{align}

%%%
We have the following lemma by Corollary \ref{Cor BbTx not=0} (\roii) together with 
\eqref{Def fFibT z},\eqref{tifF i bT z-1}, \eqref{fFibT z Res} and \eqref{fFibTz Res infty}. 

%%%%%

\begin{lem} 
\label{Lemma psiiz vbT}
For $\Bla \in \vL_{n,r}^+(m)$, $\bT \in \SStd_m(\Bla)$ and $1 \leq i \leq m-1$, we have 
\begin{align}
\label{psi i z+ vbT Res}
\begin{split}
\psi_i^+(z) \cdot v_{\bT} 
	&= q^{\mu_i(\bT) - \mu_{i+1}(\bT)} \fF_{i,\bT}(z) v_{\bT} 
	\\
	&= q^{\mu_i(\bT) - \mu_{i+1} (\bT)} 
		\big( \sum_{t \in \ZZ} \Res_{z=0} (z^{- t-1} \fF_{i,\bT}(z)) z^t \big) v_{\bT}
\end{split}
\end{align}
and 
\begin{align}
\label{psi i - z vbT Res}
\begin{split}
\psi_i^-(z) \cdot v_{\bT} 
	&= q^{\mu_i(\bT) - \mu_{i+1} (\bT)} \ti{\fF}_{i,\bT} (z^{-1}) v_{\bT} 
	\\
	&= q^{\mu_i(\bT) - \mu_{i+1} (\bT)} 
		\big( - \sum_{t \in \ZZ} \Res_{z=\infty} (z^{-t-1} \fF_{i,\bT}(z) ) z^t \big) v_{\bT}. 
\end{split}
\end{align}
\end{lem} 

We need the following lemma when we check the relation (U6). 
\begin{lem} 
\label{Lemma pole z-t-1 fFibTz} 
Assume that $\Bz \in (\CC^{\times})^r$ satisfies the separation condition \eqref{separation cond}. 
For $\Bla \in \vL_{n,r}^+(m)$, $\bT \in \SStd_m(\Bla)$, $1 \leq i \leq m-1$ and $t \in \ZZ$, 
\begin{align*}
\{\text{pole of }z^{-t-1} \fF_{i,\bT} (z)\} 
	\subset \{ (c_{\Bz}(x) q^i)^{-1} \mid x \in [\bT]_{-\a_i} \cup [\bT]_{\a_i}\} \cup \{0, \infty\}.
\end{align*}
\begin{proof} 
From the definition \eqref{Def fFibT z}, we see that 
\begin{align*}
\{ \text{pole of }  \fF_{i,\bT} (z) \} \subset \{ (c_{\Bz} (x) q^i)^{-1} \mid x \in [\bT]_i \cup [\bT]_{i+1}\}. 
\end{align*}

For $x,x' \in [\bT]_i$ and $y,y' \in [\bT]_{i+1}$, we see that 
\begin{align}
\label{cBz x not= cBz x'}
\begin{split}
&c_{\Bz} (x) \not= c_{\Bz} (x') \text{ if } x \not=x', 
\quad 
c_{\Bz} (y) \not= c_{\Bz} (y') \text{ if } y \not=y', 
\\
& c_{\Bz} (x) = c_{\Bz} (y) \text{ if and only if } y = \fn_b (\fn_r(x)) 
\end{split}
\end{align}
thanks to the separation condition \eqref{separation cond}.  
Here, we note that $x \not\in [\bT]_i^R$ and $y \not\in [\bT]_{i+1}^L$ if $y= \fn_b(\fn_r(x))$ 
since $\bT$ is semi-standard. 
Moreover, we have the following. 
\begin{itemize} 
\item 
For $x \in [\bT]_i \setminus [\bT]_i^R$, we have 
$\fn_r(x) \in [\bT]_i$ and $c_{\Bz} (x) = q^{-2} c_{\Bz} (\fn_r(x))$. 
 
\item 
For $y \in [\bT]_{i+1} \setminus [\bT]_{i+1}^L$, we have  
$\fn_l(x) \in [\bT]_{i+1}$ and $c_{\Bz} (y) = q^2 c_{\Bz} (\fn_l(y))$.  

\item 
For $x \in [\bT]_i^R \setminus [\bT]_{-\a_i}$, we have  
$\fn_b(x) \in [\bT]_{i+1}$ and $c_{\Bz} (x) = q^2 c_{\Bz} (\fn_b(x))$. 

\item 
For $y \in [\bT]_{i+1}^L \setminus [\bT]_{\a_i}$, 
we have $\fn_a(y) \in [\bT]_i$ and $c_{\Bz} (y) = q^{-2} c_{\Bz}(\fn_a(y))$. 
\end{itemize} 
These, combined with \eqref{cBz x not= cBz x'}, imply that 
the rational function $ \fF_{i, \bT}(z)$  is regular at $z= (c_{\Bz}(x) q^i)^{-1}$  
for $x \in ([\bT]_i \setminus [\bT]_{-\a_i}) \cup ([\bT]_{i+1} \setminus [\bT]_{\a_i})$. 
Then, we conclude this lemma. 
\end{proof}
\end{lem} 

%%%
\para  
It is clear that the action in Proposition \ref{Prop Def DBz Bla} satisfies the relation (U1).  

We check the relation (U2).  
From the definition given in Proposition \ref{Prop Def DBz Bla}, 
we can calculate that 
\begin{align*}
&(w- q^{a_{ij}} z) e_i(z) e_j(w) \cdot v_{\bT} 
\\
&= (w - q^{a_{ij}} z) \sum_{s,t \in \ZZ} \Big\{ 
	\sum_{x \in [\bT]_{\a_j}} \sum_{y \in [\bT_x^-]_{\a_i}} 
		q^{i \cdot t + j \cdot s} c_{\Bz} (x)^s c_{\Bz} (y)^t B_{\bT}^x B_{\bT_x^-}^y v_{(\bT_x^-)_y^-} \Big\} z^t w^s 
\\
&= \sum_{s,t \in \ZZ} 
	\Big\{ \sum_{x \in [\bT]_{\a_j}} \sum_{y \in [\bT_x^-]_{\a_i}} 
		q^{i \cdot (t-1) + j \cdot (s-1)} c_{\Bz} (x)^{s-1} c_{\Bz} (y)^{t-1} 
	\\ & \hspace{10em} \times 
	\big( q^i c_{\Bz} (y) - q^{a_{ij}} q^j c_{\Bz} (x) \big) B_{\bT}^x B_{\bT_x^-}^y v_{(\bT_x^-)_y^-} 
	\Big\} z^t w^s. 
\end{align*} 
Applying Lemma \ref{Lemma bTx- ai}, we have 
\begin{align*}
&(w- q^{a_{ij}} z) e_i(z) e_j(w) \cdot v_{\bT} 
\\
&= \sum_{s,t \in \ZZ}  
	\Big\{ \sum_{x \in [\bT]_{\a_j}} \sum_{y \in [\bT]_{\a_i} \atop y \not=x}
		q^{i \cdot (t-1) + j \cdot (s-1)} c_{\Bz} (x)^{s-1} c_{\Bz} (y)^{t-1} 
			\\ & \hspace{10em} \times 
			\big( q^i c_{\Bz} (y) - q^{a_{ij}} q^j c_{\Bz} (x) \big) B_{\bT}^x B_{\bT_x^-}^y v_{(\bT_x^-)_y^-} 
	\\ & \hspace{5em} 
	+ \d_{j,i} \sum_{ x \in [\bT]_{\a_j} \atop \fn_r(x) \in [\bT_x^-]_{\a_i}} 
		q^{i \cdot (t-1) + j \cdot (s-1)} c_{\Bz} (x)^{s-1} c_{\Bz} (\fn_r(x))^{t-1} 
			\\ & \hspace{10em} \times 
			\big( q^i c_{\Bz} (\fn_r(x)) - q^{a_{ij}} q^j c_{\Bz} (x) \big) B_{\bT}^x B_{\bT_x^-}^{\fn_r(x)} v_{(\bT_x^-)_{\fn_r(x)}^-} 
	\\ & \hspace{5em} 
	+ \d_{j, i-1} \sum_{x \in [\bT]_{\a_j} \atop \fn_b(x) \in [\bT_x^-]_{\a_i}}
		q^{i \cdot (t-1) + j \cdot (s-1)} c_{\Bz} (x)^{s-1} c_{\Bz} (\fn_b(x))^{t-1} 
			\\ & \hspace{10em} \times 
			\big( q^i c_{\Bz} (\fn_b(x)) - q^{a_{ij}} q^j c_{\Bz} (x) \big) B_{\bT}^x B_{\bT_x^-}^{\fn_b(x)} v_{(\bT_x^-)_{\fn_b(x)}^-} 
	\\ & \hspace{5em} 
	+ \d_{ j, i+1} \sum_{x \in [\bT]_{\a_j} \atop x \in [\bT_x^-]_{\a_i}} 
		q^{i \cdot (t-1) + j \cdot (s-1)} c_{\Bz} (x)^{s-1} c_{\Bz} (x)^{t-1} 
			\\ & \hspace{10em} \times 
			\big( q^i c_{\Bz} (x) - q^{a_{ij}} q^j c_{\Bz} (x) \big) B_{\bT}^x B_{\bT_x^-}^x v_{(\bT_x^-)_x^-} 
	\Big\} z^t w^s, 
\end{align*}
where we note that 
$
\begin{cases} 
q^i c_{\Bz} (\fn_r(x)) = q^{a_{ij}} q^j c_{\Bz} (x) &\text{ if } j=i, 
\\
q^i c_{\Bz} (\fn_b(x)) = q^{a_{ij}} q^j c_{\Bz} (x) & \text{ if } j=i-1,  
\\
q^i c_{\Bz} (x) - q^{a_{ij}} q^j c_{\Bz} (x) & \text{ if } j=i+1. 
\end{cases}
$
As a consequence, we have 
\begin{align*}
&(w- q^{a_{ij}} z) e_i(z) e_j(w) \cdot v_{\bT} 
\\
&= \sum_{s,t \in \ZZ}  
	\Big\{ \sum_{(x,y) \in [\bT]_{\a_j} \times [\bT]_{\a_i} \atop x \not=y} 
		q^{i \cdot (t-1) + j \cdot (s-1)} c_{\Bz} (x)^{s-1} c_{\Bz} (y)^{t-1} 
			\\ & \hspace{10em} \times 
			\big( q^i c_{\Bz} (y) - q^{a_{ij}} q^j c_{\Bz} (x) \big) B_{\bT}^x B_{\bT_x^-}^y v_{(\bT_x^-)_y^-} 
	\Big\} z^t w^s. 
\end{align*}
Applying Lemma \ref{Lemma BbTx- y} to the right-hand side of this equation, we have 
\begin{align}
\label{eiz ejw cbT} 
\begin{split}
&(w- q^{a_{ij}} z) e_i(z) e_j(w) \cdot v_{\bT} 
\\
&= \sum_{s,t \in \ZZ}  
	\Big\{ \sum_{(x,y) \in [\bT]_{\a_j} \times [\bT]_{\a_i} \atop x \not=y} 
		q^{i \cdot (t-1) + j \cdot (s-1)} c_{\Bz} (x)^{s-1} c_{\Bz} (y)^{t-1} 
			\\ & \hspace{10em} \times 
			\big( q^i c_{\Bz} (y) - q^{a_{ij}} q^j c_{\Bz} (x) \big) B_{\bT}^x B_{\bT}^y D_{\bT_x^-}^y v_{(\bT_x^-)_y^-} 
	\Big\} z^t w^s, 
\end{split}
\end{align}
where we note that, for $(x,y)\in [\bT]_{\a_j} \times [\bT]_{\a_i}$ such that $x \not=y$, 
we have $y \in [\bT_x^-]_{\a_i}$  by Lemma \ref{Lemma bTx- ai}. 

On the other hand, we can compute that 
\begin{align*}
&(q^{a_{ij}} w -z) e_j(w) e_i(z) \cdot v_{\bT} 
\\
&= \sum_{s,t \in \ZZ} \Big\{ 
	\sum_{y \in [\bT]_{\a_i}} \sum_{x \in [\bT_y^-]_{\a_j}} 
		q^{i \cdot (t-1) + j \cdot (s-1)} c_{\Bz} (x)^{s-1} c_{\Bz} (y)^{t-1} 
			\\ & \hspace{10em} \times 
			\big( q^{a_{ij}} q^i c_{\Bz} (y) - q^j c_{\Bz} (x) \big) B_{\bT}^y B_{\bT_y^-}^x v_{(\bT_y^-)_x^-} 
	\Big\} z^t w^s. 
\end{align*}
Applying Lemma \ref{Lemma bTx- ai}, we have 
\begin{align*}
&(q^{a_{ij}} w -z) e_j(w) e_i(z) \cdot v_{\bT} 
\\
&= \sum_{s,t \in \ZZ} \Big\{ 
	\sum_{(x,y) \in [\bT]_{\a_j} \times [\bT]_{\a_i} \atop x \not=y} 
		q^{i \cdot (t-1) + j \cdot (s-1)} c_{\Bz} (x)^{s-1} c_{\Bz} (y)^{t-1} 
			\\ & \hspace{10em} \times 
			\big( q^{a_{ij}} q^i c_{\Bz} (y) - q^j c_{\Bz} (x) \big) B_{\bT}^y B_{\bT_y^-}^x v_{(\bT_y^-)_x^-} 
	\Big\} z^t w^s,  
\end{align*}
where we note that 
$\begin{cases}
q^{a_{ij}} q^i c_{\Bz} (y) = q^j c_{\Bz} (\fn_r(y)) & \text{ if } i=j, 
\\
q^{a_{ij}} q^i c_{\Bz} (y) = q^j c_{\Bz} (\fn_b (y)) & \text{ if } i = j-1, 
\\
q^{a_{ij}} q^i c_{\Bz} (y) = q^j c_{\Bz} (y) & \text{ if } i =j+1.
\end{cases}$
Applying Lemma \ref{Lemma BbTx- y}, we obtain 
\begin{align}
\label{ejw eiz vbT} 
\begin{split}
&(q^{a_{ij}} w -z) e_j(w) e_i(z) \cdot v_{\bT} 
\\
&= \sum_{s,t \in \ZZ} \Big\{ 
	\sum_{(x,y) \in [\bT]_{\a_j} \times [\bT]_{\a_i} \atop x \not=y} 
		q^{i \cdot (t-1) + j \cdot (s-1)} c_{\Bz} (x)^{s-1} c_{\Bz} (y)^{t-1} 
			\\ & \hspace{10em} \times 
			\big( q^{a_{ij}} q^i c_{\Bz} (y) - q^j c_{\Bz} (x) \big) B_{\bT}^y B_{\bT}^x D_{\bT_y^-}^x v_{(\bT_y^-)_x^-} 
	\Big\} z^t w^s. 
\end{split}
\end{align}
Moreover, for $(x,y) \in [\bT]_{\a_j} \times [\bT]_{\a_i}$ such that $x \not=y$, 
we can check that  \break 
$v_{(\bT_x^-)_y^-} = v_{(\bT_y^-)_x^-}$ 
and 
$\big( q^i c_{\Bz} (y) - q^{a_{ij}} q^j c_{\Bz} (x) \big) D_{\bT_x^-}^y 
	= \big( q^{a_{ij}} q^i c_{\Bz} (y) - q^j c_{\Bz} (x) \big)D_{\bT_y^-}^x$ 
by direct calculation. 
Then, the equations \eqref{eiz ejw cbT} and \eqref{ejw eiz vbT} imply 
\begin{align*}
(w- q^{a_{ij}} z) e_i(z) e_j(w) \cdot v_{\bT}  = (q^{a_{ij}} w -z) e_j(w) e_i(z) \cdot v_{\bT}.
\end{align*}

We can also check the relation (U3) in a similar way.  

%%%%%%%%%
\para 
We check the relation (U4). 
By the definition given in Proposition \ref{Prop Def DBz Bla} and \eqref{psi i z+ vbT Res}, 
we have 
\begin{align*}
&(w- q^{a_{ij}} z) \psi_i^+ (z) e_j(w) \cdot v_{\bT} 
\\
&= (w- q^{a_{ij}} z) \sum_{t \in \ZZ} \Big\{ \sum_{x \in [\bT]_{\a_j}} B_{\bT}^x (q^j c_{\Bz} (x))^t  
	\big( q^{\mu_i( \bT_x^-) - \mu_{i+1} (\bT_x^-)} \fF_{i, \bT_x^-} (z) \big) v_{\bT_x^-} \Big\} w^t 
\\
&= \sum_{t \in \ZZ} \Big\{ \sum_{x \in [\bT]_{\a_j}} 
		B_{\bT}^x (q^j c_{\Bz}(x))^{t-1} 
		q^{\mu_i(\bT) - \mu_{i+1} (\bT) + a_{ij}} (1 - q^{a_{ij}} q^j c_{\Bz} (x) z) \fF_{i,\bT_x^-}(z) v_{\bT_x^-} \Big\} w^t.  
\end{align*}
By direct calculation using Lemma \ref{Lemma Tx- iL}, we obtain 
\begin{align*}
\fF_{i,\bT_x^-}(z)= \frac{( 1 - q^{-a_{ij}} c_{\Bz} (x) q^j z)}{(1 - q^{a_{ij}} c_{\Bz} (x) q^j z)} \fF_{i, \bT}(z) 
\end{align*}
for $x \in [\bT]_{\a_j}$. Thus, we conclude that 
\begin{align*}
&(w- q^{a_{ij}} z) \psi_i^+ (z) e_j(w) \cdot v_{\bT} 
\\
&= q^{\mu_i(\bT) - \mu_{i+1} (\bT)} \sum_{t \in \ZZ} \Big\{ \sum_{x \in [\bT]_{\a_j}} B_{\bT}^x (q^j c_{\Bz} (x))^{t-1} 
	(q^{a_{ij}} - q^j c_{\Bz}(x) z) \fF_{i,\bT} (z) v_{\bT_x^-}\Big\} w^t.
\end{align*}

On the other hand, we also have 
\begin{align*}
&(q^{a_{ij}} w - z) e_j(w) \psi_i^+(z) \cdot v_{\bT} 
\\
&= (q^{a_{ij}} w - z) q^{\mu_i(\bT) - \mu_{i+1}(\bT)} \fF_{i,\bT} (z) 
	\sum_{t \in \ZZ} \Big\{ \sum_{x \in [\bT]_{\a_j}} B_{\bT}^x (q^j c_{\Bz} (x))^t v_{\bT_x^-} \Big\} w^t 
\\
&= q^{\mu_i(\bT) - \mu_{i+1} (\bT)} \sum_{t \in \ZZ} \Big\{ \sum_{x \in [\bT]_{\a_j}} B_{\bT}^x (q^j c_{\Bz} (x))^{t-1} 
	(q^{a_{ij}} - q^j c_{\Bz}(x) z) \fF_{i,\bT} (z) v_{\bT_x^-}\Big\} w^t.
\end{align*}
As a consequence, we have  
\begin{align*} 
(w- q^{a_{ij}} z) \psi_i^+ (z) e_j(w) \cdot v_{\bT}  = (q^{a_{ij}} w - z) e_j(w) \psi_i^+(z) \cdot v_{\bT}. 
\end{align*} 

In a similar way, 
we can check the relation replaced $\psi_i^+(z)$ with $\psi_i^-(z)$ in the above equation. We can also check the relation (U5). 

%%%%%%
\para 
We check the relation (U6).  
By the definition given in Proposition \ref{Prop Def DBz Bla}, we have  
\begin{align*}
&e_i(z) f_j(w) \cdot v_{\bT} 
\\
&= \sum_{s \in \ZZ} \Big\{ \sum_{y \in [\bT]_{-\a_j}} C_{\bT}^y (q^j c_{\Bz} (y))^s 
	\Big( \sum_{t \in \ZZ} \big\{ \sum_{x \in [\bT_y^+]_{\a_i}} B_{\bT_y^+}^x (q^i c_{\Bz} (x))^t v_{(\bT_y^+)_x^-} \big\} z^t \Big) 
	\Big\} w^s 
\\
&= \sum_{s,t \in \ZZ} \Big\{ 
	\sum_{y \in [\bT]_{-\a_j}} \sum_{x \in [\bT_y^+]_{\a_i}} 
		q^{i \cdot t + j \cdot s} c_{\Bz} (y)^s c_{\Bz} (x)^t C_{\bT}^y B_{\bT_y^+}^x  v_{(\bT_y^+)_x^-} \Big\} w^s z^t. 
\end{align*}
Thanks to Lemma \ref{Lemma bTx- ai}, 
we see that $[\bT_y^+]_{\a_i} = 
	\begin{cases} 
		\big( [\bT]_{\a_i} \cap [\bT_y^+]_{\a_i} \big) \cup \{y\} & \text{ if } j=i, 
		\\
		[\bT]_{\a_i} \cap [\bT_y^+]_{\a_i} & \text{ if } j \not=i
	\end{cases} $
for $x \in [\bT]_{-\a_j}$. 
This implies that 
\begin{align*}
&e_i(z) f_j(w) \cdot v_{\bT}
\\
&= \sum_{s, t \in \ZZ} \Big\{ \sum_{y \in [\bT]_{-\a_j}} \sum_{x \in [\bT]_{\a_i} \cap [\bT_y^+]_{\a_i}} 
		q^{i \cdot t + j \cdot s} c_{\Bz} (y)^s c_{\Bz} (x)^t C_{\bT}^y B_{\bT_y^+}^x  v_{(\bT_y^+)_x^-} 
	\\ & \hspace{4em} 
	+ \d_{i,j} \sum_{ y \in [\bT]_{-\a_i}} q^{i \cdot (s+t) } c_{\Bz} (y)^{s+t} C_{\bT}^y B_{\bT_y^+}^y  v_{(\bT_y^+)_y^-} 
	\Big\} w^s z^t. 
\end{align*}
Furthermore, applying Lemma \ref{Lemma wtDbTxy}, 
we have 
\begin{align}
\label{eiz fjw vbT}
\begin{split}
&e_i(z) f_j(w) \cdot v_{\bT}
\\
&= \sum_{s, t \in \ZZ} \Big\{ \sum_{y \in [\bT]_{-\a_j}} \sum_{x \in [\bT]_{\a_i} \cap [\bT_y^+]_{\a_i}} 
		q^{i \cdot t + j \cdot s} c_{\Bz} (y)^s c_{\Bz} (x)^t C_{\bT}^y B_{\bT}^x \wt{D}_{\bT}^{x,y}  v_{(\bT_y^+)_x^-} 
	\\ & \hspace{4em} 
	+ \d_{i,j} \sum_{ y \in [\bT]_{-\a_i}} q^{i \cdot (s+t) } c_{\Bz} (y)^{s+t} C_{\bT}^y B_{\bT_y^+}^y  v_{\bT} 
	\Big\} w^s z^t. 
\end{split} 
\end{align}
Similarly, the definition in Proposition \ref{Prop Def DBz Bla}, Lemma \ref{Lemma bTx- ai} and Lemma \ref{Lemma wtDbTxy} 
imply that 
\begin{align}
\label{fjw eiz vbT} 
\begin{split}
&f_j(w) e_i(z) \cdot v_{\bT} 
\\
&= \sum_{s,t \in \ZZ} \Big\{ 
	\sum_{x \in [\bT]_{\a_i}} \sum_{y \in [\bT]_{-\a_j} \cap [\bT_x^-]_{-\a_j} } 
		q^{i \cdot t + j \cdot s} c_{\Bz} (x)^t c_{\Bz} (y)^s B_{\bT}^x C_{\bT}^y \wt{D}_{\bT}^{x,y} v_{(\bT_x^-)_y^+} 
	\\ & \hspace{4em} 
	+ \d_{i,j} \sum_{x \in [\bT]_{\a_i}} q^{i \cdot (s+t)} c_{\Bz} (x)^{s+t} B_{\bT}^x C_{\bT_x^-}^x v_{\bT} 
	\Big\} w^s z^t. 
\end{split}
\end{align} 

For  $x \in [\bT]_{\a_i}$ and $y \in [\bT]_{-\a_j}$, 
we see that 
$x \in [\bT_y^+]_{\a_i}$ if and only if $ y \in [\bT_x^-]_{-\a_j}$ 
by  Lemma \ref{Lemma bTx- ai}.  
Thus, there exists a bijection 
\begin{align*}
\{(y,x) \in [\bT]_{-\a_j} \times [\bT]_{\a_i} \mid x \in [\bT_y^+]_{\a_i}\} 
\ra 
\{(x,y) \in [\bT]_{\a_i} \times [\bT]_{-\a_j} \mid y \in  [\bT_x^-]_{-\a_j} \} 
\end{align*}
given by $(y,x) \mapsto (x,y)$. 
Combined with \eqref{eiz fjw vbT} and \eqref{fjw eiz vbT}, we have 
\begin{align*}
&(e_i(z) f_j(w) - f_j(w) e_i(z) ) \cdot v_{\bT} 
\\
&= \d_{i,j} \sum_{s,t \in \ZZ} \Big\{
	\Big( \sum_{y \in [\bT]_{-\a_i}} q^{i \cdot (s+t)} c_{\Bz} (y)^{s+t} C_{\bT}^y B_{\bT_y^+}^y  
			- \sum_{x \in [\bT]_{\a_i}} q^{i \cdot (s+t)} c_{\Bz} (x)^{s+t} B_{\bT}^x C_{\bT_x^-}^x \Big) v_{\bT} 
	\Big\} w^s z^t 
\\
&= \d_{i,j} \sum_{s,t' \in \ZZ} \Big\{
	\Big( \sum_{y \in [\bT]_{-\a_i}} q^{i \cdot t' } c_{\Bz} (y)^{t'} C_{\bT}^y B_{\bT_y^+}^y  
			- \sum_{x \in [\bT]_{\a_i}} q^{i \cdot t'} c_{\Bz} (x)^{t'} B_{\bT}^x C_{\bT_x^-}^x \Big) v_{\bT} 
	\Big\} w^s z^{-s} z^{t'}, 
\end{align*}
where we put $t'=s+t$. 
As a consequence, we have 
\begin{align}
\label{eiz fjw - fjw eiz vbT}
\begin{split}
&(e_i(z) f_j(w) - f_j(w) e_i(z) ) \cdot v_{\bT} 
\\
&= \d_{i,j} \Big( \d \left( \frac{w}{z} \right) 
	\sum_{t \in \ZZ} \Big\{ \sum_{y \in [\bT]_{-\a_i}} (q^i c_{\Bz}(y))^t C_{\bT}^y B_{\bT_y^+}^y 
						- \sum_{x \in [\bT]_{\a_i}} (q^i c_{\Bz} (x))^t B_{\bT}^x C_{\bT_x^-}^x \Big\} z^t 
	\Big) v_{\bT}. 
\end{split}
\end{align}

On the other hand,  
we have 
\begin{align}
\label{psii+z - psii-z vbT}
\begin{split}
& (\psi_i^+(z) - \psi_i^- (z) ) \cdot v_{\bT}  
\\
&= q^{\mu_i(\bT) - \mu_{i+1} (\bT)} \Big(  \sum_{t \in \ZZ} 
	\Big\{ \Res_{z=0} (z^{-t-1} \fF_{i,\bT}(z) ) + \Res_{z= \infty} (z^{-t-1} \fF_{i,\bT} (z)) \Big\} z^t  \Big) v_{\bT}  
\end{split} 
\end{align}
by Lemma \ref{Lemma psiiz vbT}.  
Applying the residue theorem together with Lemma \ref{Lemma pole z-t-1 fFibTz}, 
we have 
\begin{align}
\label{Resz=0} 
\begin{split}
& \Res_{z=0} (z^{-t-1} \fF_{i,\bT}(z) ) + \Res_{z= \infty} (z^{-t-1} \fF_{i,\bT} (z)) 
\\
&= - \sum_{x \in [\bT]_{-\a_i} \cup [\bT]_{\a_i}} \Res_{z= (c_{\Bz} (x) q^i)^{-1}} (z^{-t-1} \fF_{i,\bT}(z)).
\end{split}
\end{align}
Put 
\begin{align*}
&\fP(z) = \prod_{y \in [\bT]_i} (1 - q^{-2} c_{\Bz}(y) q^i z) \prod_{y' \in [\bT]_{i+1}} (1- q^2 c_{\Bz}(y') q^i z), 
\\
& \fQ(z) = \prod_{y \in [\bT]_i} (1 - c_{\Bz}(y) q^i z) \prod_{y' \in [\bT]_{i+1}} (1 - c_{\Bz} (y') q^i z). 
\end{align*}
Then we have $\fF_{i,\bT}(z) = \fP(z) \cdot \fQ(z)^{-1}$. 
Then, for $x  \in [\bT]_{-\a_i} \cup [\bT]_{\a_i}$, 
we have 
\begin{align*}
\Res_{z = (c_{\Bz}(x)q^i)^{-1}} (z^{-t-1} \fF_{i,\bT}(z)) 
&= \lim_{z \to (c_{\Bz}(x)q^i)^{-1} } (z- (c_{\Bz}(x)q^i)^{-1}) \frac{ z^{-t-1} \fP(z)}{\fQ(z)} 
\\
&= \frac{ (c_{\Bz}(x) q^i)^{t+1} \fP ((c_{\Bz}(x)q^i)^{-1})}{ \fQ'( (c_{\Bz}(x)q^i)^{-1})}
\end{align*}
thanks to \eqref{cBz x not= cBz x'}. 
Therefore, we can compute that 
\begin{align*}
&\Res_{z = (c_{\Bz}(x)q^i)^{-1}} (z^{-t-1} \fF_{i,\bT}(z)) 
\\
&= \begin{cases}
	\dis 
	- q^{-1} (q-q^{-1}) (c_{\Bz}(x)q^i)^{t} 
		\prod_{y \in [\bT]_i \setminus \{x\}} \frac{(c_{\Bz} (x) - q^{-2} c_{\Bz}(y) ) }{ (c_{\Bz} (x) - c_{\Bz}(y) ) }
	 	\prod_{y' \in [\bT]_{i+1}} \frac{ (c_{\Bz} (x) - q^2 c_{\Bz}(y') )}{ (c_{\Bz} (x) - c_{\Bz}(y')  ) }
	\hspace{-6em}
	\\
	& \text{ if } x \in [\bT]_{-\a_i}, 
	\\[1em]
	\dis 
	q (q-q^{-1}) (c_{\Bz}(x)q^i)^{t} 
		\prod_{y \in [\bT]_i} \frac{  (c_{\Bz} (x) - q^{-2} c_{\Bz}(y) ) }{ (c_{\Bz} (x) - c_{\Bz}(y) ) }
		\prod_{y' \in [\bT]_{i+1} \setminus \{x\}} \frac{ (c_{\Bz} (x) - q^2 c_{\Bz}(y'))}{(c_{\Bz} (x) - c_{\Bz}(y'))} 
	\hspace{-6em}
	\\
	& \text{ if } x \in [\bT]_{\a_i}.
	\end{cases}
\end{align*}
Combimed with \eqref{BbTx CbTx} and Lemma \ref{Lemma BbT x+ x}, we have 
\begin{align}
\label{Res z=cbZxqi -1}
\begin{split}
&\Res_{z = (c_{\Bz}(x)q^i)^{-1}} (z^{-t-1} \fF_{i,\bT}(z)) 
\\
&= \begin{cases}
	- q^{-\mu_i(\bT) + \mu_{i+1} (\bT) } (q-q^{-1}) (c_{\Bz}(x) q^i)^t  B_{\bT_{x}^+}^{x} C_{\bT}^x 
		& \text{ if } x \in [\bT]_{-\a_i}, 
	\\[0.5em]
	q^{- \mu_i(\bT) + \mu_{i+1} (\bT)} (q-q^{-1}) (c_{\Bz} (x) q^i)^t  B_{\bT}^x  C_{\bT_x^-}^x 
		& \text{ if } x \in [\bT]_{\a_i}. 
\end{cases}
\end{split}
\end{align}
The equations \eqref{psii+z - psii-z vbT}, \eqref{Resz=0} and \eqref{Res z=cbZxqi -1} imply that 
\begin{align}
\label{psii+ psii- vbT}
\begin{split}
&(\psi_i^+(z) - \psi_i^-(z)) \cdot v_{\bT} 
\\
&= (q-q^{-1}) \Big( \sum_{t \in \ZZ} \Big\{ 
	\sum_{y \in [\bT]_{-\a_i}} (c_{\Bz}(y) q^i)^t B_{\bT_y^+}^y C_{\bT}^y 
	- \sum_{x \in [\bT]_{\a_i}} (c_{\Bz} (x) q^i)^t B_{\bT}^x C_{\bT_x^-}^x 
	\Big\} z^t \Big) v_{\bT}. 
\end{split} 
\end{align}
Then, by comparing \eqref{eiz fjw - fjw eiz vbT} and \eqref{psii+ psii- vbT}, we have 
\begin{align*}
(e_i(z) f_j(w) - f_j(w) e_i(z)) \cdot v_{\bT} 
= \frac{\d_{i,j}}{q-q^{-1}} \d \left( \frac{w}{z} \right) (\psi_i^+(z) - \psi_i^-(z)) \cdot v_{\bT}. 
\end{align*}

%%%%%
\para 
We check the relation (U7). 

If $j \not=i, i \pm 1$,  
we can easily check $\big( e_i(z) e_j(w) \big) \cdot v_{\bT} = \big( e_j(w) e_i(z) \big) \cdot v_{\bT}$ directly. 

By definition, we can calculate that 
\begin{align*}
& e_{i+1} (w) e_i(z_1) e_i(z_2) \cdot v_{\bT} 
\\
&= \sum_{t_1,t_2,s\in \ZZ} q^{i \cdot (t_1+t_2) + (i+1) \cdot s} 
	\Big\{ \sum_{x_2 \in [\bT]_{\a_i}} \sum_{x_1 \in [\bT_{x_2}^-]_{\a_i}} \sum_{y \in [(\bT_{x_2}^-)_{x_1}^-]_{\a_{i+1}}} 
		c_{\Bz}(x_2)^{t_2} c_{\Bz} (x_1)^{t_1}   c_{\Bz} (y)^s 
		\\ & \qquad \times 
		B_{\bT}^{x_2} B_{\bT_{x_2}^-}^{x_1} B_{(\bT_{x_2}^-)_{x_1}^-}^y v_{((\bT_{x_2}^-)_{x_1}^-)_y^-} \Big\} 
		z_1^{t_1} z_2^{t_2} w^s. 
\end{align*}

Put $f(x_1,x_2,y)= c_{\Bz} (x_1)^{t_1}   c_{\Bz}(x_2)^{t_2}  c_{\Bz} (y)^s 
	B_{\bT_{x_2}^-}^{x_1} B_{\bT}^{x_2}  B_{(\bT_{x_2}^-)_{x_1}^-}^y v_{((\bT_{x_2}^-)_{x_1}^-)_y^-} $, 
and we have the following equations by  Lemma \ref{Lemma bTx- ai}. 
\begin{align*}
&\sum_{x_2 \in [\bT]_{\a_i}} \sum_{x_1 \in [\bT_{x_2}^-]_{\a_i}} \sum_{y \in [(\bT_{x_2}^-)_{x_1}^-]_{\a_{i+1}}} 
	f(x_1,x_2,y) 
\\
&= \sum_{x_2 \in [\bT]_{\a_i}} \sum_{x_1 \in [\bT]_{\a_i} \atop x_1 \not=x_2} \sum_{y \in [(\bT_{x_2}^-)_{x_1}^-]_{\a_{i+1}}} 
		f(x_1,x_2,y)
	+ \sum_{x_2 \in [\bT]_{\a_i} \atop \fn_r(x_2) \in [\bT_{x_2}^-]_{\a_i}}  \sum_{y \in [(\bT_{x_2}^-)_{\fn_r(x_2)}^-]_{\a_{i+1}}}
		f(\fn_r(x_2),x_2,y)
\\
&= \sum_{x_2 \in [\bT]_{\a_i}} \sum_{x_1 \in [\bT]_{\a_i} \atop x_1 \not=x_2} \sum_{y \in [\bT_{x_2}^-]_{\a_{i+1}}}
		f(x_1,x_2,y)
	+ \sum_{x_2 \in [\bT]_{\a_i}} \sum_{x_1 \in [\bT]_{\a_i} \atop x_1 \not=x_2, \, \fn_b(x_1) \in [(\bT_{x_2}^-)_{x_1}^-]_{\a_{i+1}}} 
		f(x_1,x_2, \fn_b(x_1))
	\\ & \quad 
	+ \sum_{x_2 \in [\bT]_{\a_i} \atop \fn_r(x_2) \in [\bT_{x_2}^-]_{\a_i}}  \sum_{y \in [\bT_{x_2}^-]_{\a_{i+1}}}
		f(\fn_r(x_2),x_2,y) 
	\\ & \hspace{3em} 
	+ \sum_{x_2 \in [\bT]_{\a_i} 
			\atop \fn_r(x_2) \in [\bT_{x_2}^-]_{\a_i}, \, \fn_b (\fn_r(x_2)) \in [(\bT_{x_2}^-)_{\fn_r(x_2)}^-]_{\a_{i+1}}}  
		f(\fn_r(x_2),x_2,\fn_b (\fn_r(x_2))), 
\end{align*}
where we note that 
\begin{itemize}
\item 
$\fn_b(x_1) \in [(\bT_{x_1}^-)_{x_2}^-]_{\a_{i+1}} $ if and only if $\fn_b(x_1) \in [\bT_{x_1}^-]_{\a_{i+1}}$ 
for $x_1,x_2 \in [\bT]_{\a_i}$ such that $x_1 \not= x_2$ by Lemma \ref{Lemma bTx- ai},  
where we note that $x_1=x_2$ if $\fn_b(x_1) = \fn_b(x_2) \in [(\bT_{x_1}^-)_{x_2}^-]_{\a_{i+1}}$. 

\item 
$\fn_b (\fn_r(x_2)) \not\in [(\bT_{x_2}^-)_{\fn_r(x_2)}^-]_{\a_{i+1}}$ for $x_2 \in [\bT]_{\a_i}$ 
since $(\bT_{x_2}^-)_{\fn_r(x_2)}^- (\fn_l( \fn_b(\fn_r(x_2)))) = (\bT_{x_2}^-)_{\fn_r(x_2)}^- ( \fn_b(x_2)) =\bT (\fn_b(x_2)) > \bT(x_2) =i+1$.
\end{itemize}
Then, we have 
\begin{align*}
&\sum_{x_2 \in [\bT]_{\a_i}} \sum_{x_1 \in [\bT_{x_2}^-]_{\a_i}} \sum_{y \in [(\bT_{x_2}^-)_{x_1}^-]_{\a_{i+1}}} 
	f(x_1,x_2,y) 
\\
&= \sum_{x_2 \in [\bT]_{\a_i}} \sum_{x_1 \in [\bT]_{\a_i} \atop x_1 \not=x_2} \sum_{y \in [\bT_{x_2}^-]_{\a_{i+1}}}
		f(x_1,x_2,y)
	+ \sum_{x_2 \in [\bT]_{\a_i}} \sum_{x_1 \in [\bT]_{\a_i} \atop x_1 \not=x_2, \, \fn_b(x_1) \in [\bT_{x_1}^-]_{\a_{i+1}}} 
		f(x_1,x_2, \fn_b(x_1))
	\\ & \quad 
	+ \sum_{x_2 \in [\bT]_{\a_i} \atop \fn_r(x_2) \in [\bT_{x_2}^-]_{\a_i}}  \sum_{y \in [\bT_{x_2}^-]_{\a_{i+1}}}
		f(\fn_r(x_2),x_2,y) 
\\
&= \sum_{x_2 \in [\bT]_{\a_i}} \sum_{x_1 \in [\bT]_{\a_i} \atop x_1 \not=x_2} \sum_{y \in [\bT]_{\a_{i+1}}}
		f(x_1,x_2,y)
	+\sum_{x_2 \in [\bT]_{\a_i} \atop \fn_b(x_2) \in [\bT_{x_2}^-]_{\a_{i+1}}} \sum_{x_1 \in [\bT]_{\a_i} \atop x_1 \not=x_2} 
		f(x_1,x_2,\fn_b(x_2))
	\\ & \quad 
	+ \sum_{x_2 \in [\bT]_{\a_i}} \sum_{x_1 \in [\bT]_{\a_i} \atop x_1 \not=x_2, \, \fn_b(x_1) \in [\bT_{x_1}^-]_{\a_{i+1}}} 
		f(x_1,x_2, \fn_b(x_1))
	\\ & \quad 
	+ \sum_{x_2 \in [\bT]_{\a_i} \atop \fn_r(x_2) \in [\bT_{x_2}^-]_{\a_i}}  \sum_{y \in [\bT]_{\a_{i+1}}}
		f(\fn_r(x_2),x_2,y) 
	+ \sum_{x_2 \in [\bT]_{\a_i} \atop \fn_r(x_2) \in [\bT_{x_2}^-]_{\a_i}, \, \fn_b(x_2) \in [\bT_{x_2}^-]_{\a_{i+1}}}  
		f(\fn_r(x_2),x_2, \fn_b(x_2)) 
\\
&= \sum_{(x_1,x_2,y) \in [\bT]_{\a_i} \times [\bT]_{\a_i} \times [\bT]_{\a_{i+1}} \atop x_1 \not=x_2} 
		f (x_1,x_2,y) 
	\\ & \quad 
	+ \sum_{(x_1,x_2) \in [\bT]_{\a_i} \times [\bT]_{\a_i} \atop x_1 \not=x_2, \, \fn_b(x_2) \in [\bT_{x_2}^-]_{\a_{i+1}}} 
		f (x_1,x_2, \fn_b(x_2)) 
	+ \sum_{(x_1,x_2) \in [\bT]_{\a_i} \times [\bT]_{\a_i} \atop x_1 \not=x_2, \, \fn_b(x_2) \in [\bT_{x_2}^-]_{\a_{i+1}}} 
		f (x_2,x_1, \fn_b(x_2))  
	\\ & \quad 
	+ \sum_{ (x,y) \in [\bT]_{\a_i} \times [\bT]_{\a_{i+1}} \atop \fn_r (x) \in [\bT_{x}^-]_{\a_i}} 
		f (\fn_r(x), x, y) 
	+ \sum_{x \in [\bT]_{\a_i} \atop \fn_r(x) \in [\bT_x^-]_{\a_i}, \, \fn_b(x) \in [\bT_x^-]_{\a_{i+1}}} 
		f (\fn_r(x), x, \fn_b(x)).  
\end{align*}
Furthermore, 
applying Lemma \ref{Lemma BbTx- y}, we have 
\begin{align}
\label{ei+1w eiz1 eiz2 vbT} 
& e_{i+1} (w) e_i(z_1) e_i(z_2) \cdot v_{\bT} 
\\
\notag
&= \sum_{t_1,t_2,s \in \ZZ} q^{ i \cdot (t_1+t_2) + (i+1) \cdot s} 
	\\ \notag
	& \quad \times \Big\{ 
	\sum_{(x_1,x_2,y) \in [\bT]_{\a_i} \times [\bT]_{\a_i} \times [\bT]_{\a_{i+1}} \atop x_1 \not=x_2} 
		c_{\Bz} (x_1)^{t_1} c_{\Bz}(x_2)^{t_2} c_{\Bz} (y)^s 
		D_{\bT_{x_2}^-}^{x_1} D_{(\bT_{x_2}^-)_{x_1}^-}^y D_{\bT_{x_2}^-}^y
		\\ \notag
		& \hspace{5em} \times 
		 B_{\bT}^{x_1}  B_{\bT}^{x_2} B_{\bT}^y \, v_{((\bT_{x_2}^-)_{x_1}^-)_y^-} 
	\\ \notag & \hspace{3em} 
	+ \sum_{(x_1,x_2) \in [\bT]_{\a_i} \times [\bT]_{\a_i} \atop x_1 \not=x_2, \, \fn_b(x_2) \in [\bT_{x_2}^-]_{\a_{i+1}}} 
		c_{\Bz} (x_1)^{t_1} c_{\Bz}(x_2)^{t_2} c_{\Bz} (\fn_b(x_2))^s 
		D_{\bT_{x_2}^-}^{x_1} D_{(\bT_{x_2}^-)_{x_1}^-}^{\fn_b(x_2)}
		\\ & \notag \hspace{5em} \times 
		B_{\bT}^{x_1} B_{\bT}^{x_2}  B_{\bT_{x_2}^-}^{\fn_b(x_2)} v_{((\bT_{x_2}^-)_{x_1}^-)_{\fn_b(x_2)}^-} 
	\\ & \notag \hspace{3em} 
	+ \sum_{(x_1,x_2) \in [\bT]_{\a_i} \times [\bT]_{\a_i} \atop x_1 \not= x_2, \, \fn_b(x_2) \in [\bT_{x_2}^-]_{\a_{i+1}}}
		c_{\Bz} (x_2)^{t_1} c_{\Bz}(x_1)^{t_2} c_{\Bz} (\fn_b(x_2))^s 
		D_{\bT_{x_1}^-}^{x_2} D_{(\bT_{x_2}^-)_{x_1}^-}^{\fn_b(x_2)} 
		\\ & \notag \hspace{5em} \times 
		B_{\bT}^{x_2} B_{\bT}^{x_1}  B_{\bT_{x_2}^-}^{\fn_b(x_2)} v_{((\bT_{x_1}^-)_{x_2}^-)_{\fn_b (x_2)}^-}  
	\\ & \notag \hspace{3em} 
	+ \sum_{(x,y) \in [\bT]_{\a_i} \times [\bT]_{\a_{i+1}} \atop \fn_r(x) \in [\bT_x^-]_{\a_i}} 
		c_{\Bz} (\fn_r(x))^{t_1} c_{\Bz}(x)^{t_2} c_{\Bz} (y)^s 
		D_{(\bT_{x}^-)_{\fn_r(x)}^-}^y D_{\bT_{x}^-}^y 
		\\ & \notag \hspace{5em} \times 
		B_{\bT_{x}^-}^{\fn_r(x)} B_{\bT}^{x} B_{\bT}^y v_{((\bT_{x}^-)_{\fn_r(x)}^-)_y^-} 
	\\ & \notag \hspace{3em} 
	+ \sum_{x \in [\bT]_{\a_i} \atop \fn_r(x) \in [\bT_x^-]_{\a_i}, \, \fn_b(x) \in [\bT_x^-]_{\a_{i+1}}} 
		c_{\Bz} (\fn_r(x))^{t_1} c_{\Bz}(x)^{t_2} c_{\Bz} (\fn_b(x))^s 
		D_{(\bT_{x}^-)_{\fn_r(x)}^-}^{\fn_b(x)}
		\\ & \notag \hspace{5em} \times 
		B_{\bT_{x}^-}^{\fn_r(x)} B_{\bT}^{x}  B_{\bT_{x}^-}^{\fn_b(x)}  v_{((\bT_{x}^-)_{\fn_r(x)}^-)_{\fn_b(x)}^-}   
	\Big\} z_1^{t_1} z_2^{t_2} w^s.
\end{align}
Similarly, we have 
\begin{align}
\label{eiz1 eiz2 ei+1w vbT} 
&e_i(z_1) e_i(z_2) e_{i+1}(w) \cdot v_{\bT} 
\\
& \notag
= \sum_{t_1,t_2,s \in \ZZ} q^{i \cdot (t_1+t_2) + (i+1) \cdot s} 
	\\ & \notag \quad \times \Big\{ 
	\sum_{(x_1,x_2,y) \in [\bT]_{\a_i} \times [\bT]_{\a_i} \times [\bT]_{\a_{i+1}} \atop x_1 \not= x_2 } 
		c_{\Bz} (x_1)^{t_1} c_{\Bz} (x_2)^{t_2} c_{\Bz} (y)^s 
		D_{\bT_y^-}^{x_1} D_{\bT_y^-}^{x_2} D_{(\bT_y^-)_{x_2}^-}^{x_1} 
		%\\ & \hspace{5em} \times 
		B_{\bT}^{x_1} B_{\bT}^{x_2} B_{\bT}^y \, v_{((\bT_{x_1}^-)_{x_2}^-)_y^-} 
	\\ \notag & \hspace{3em} 
	+ \sum_{(x,y) \in [\bT]_{\a_i} \times [\bT]_{\a_{i+1}} \atop \fn_r (x) \in [\bT_x^-]_{\a_i}} 
		c_{\Bz} (\fn_r (x))^{t_1} c_{\Bz} (x)^{t_2} c_{\Bz} (y)^s 
		D_{(\bT_x^-)_y^-}^{\fn_r(x)} D_{\bT_y^-}^x 
		%\\ & \hspace{5em} \times 
		B_{\bT_x^-}^{\fn_r(x)} B_{\bT}^x B_{\bT}^y \, v_{((\bT_x^-)_y^-)_{\fn_r(x)}^-} 
	\\  \notag & \hspace{3em} 
	+ \sum_{(x,y) \in [\bT]_{\a_i} \times [\bT]_{\a_{i+1}} \atop y \in [\bT_y^-]_{\a_i}} 
		c_{\Bz} (x)^{t_1} c_{\Bz} (y)^{t_2} c_{\Bz} (y)^s 
		D_{\bT_y^-}^x D_{(\bT_y^-)_y^-}^x 
		%\\ & \hspace{5em} \times 
		B_{\bT_y^-}^y B_{\bT}^x B_{\bT}^y \, v_{((\bT_x^-)_y^-)_y^-} 
	\\  \notag & \hspace{3em} 
	+ \sum_{(x,y) \in [\bT]_{\a_i} \times [\bT]_{\a_{i+1}} \atop y \in [\bT_y^-]_{\a_i}} 
		c_{\Bz} (y)^{t_1} c_{\Bz} (x)^{t_2} c_{\Bz} (y)^s 
		D_{\bT_y^-}^x D_{(\bT_y^-)_x^-}^y 
		%\\ & \hspace{5em} \times 
		B_{\bT_y^-}^y B_{\bT}^x B_{\bT}^y \, v_{((\bT_x^-)_y^-)_y^-} 
	\\ \notag & \hspace{3em} 
	+ \sum_{x \in [\bT]_{\a_i} \atop \fn_r(x) \in [\bT]_{\a_{i+1}} \cap [(\bT_x^-)_{\fn_r(x)}^-]_{\a_i}} 
		c_{\Bz}(\fn_r(x))^{t_1} c_{\Bz} (x)^{t_2} c_{\Bz} (\fn_r(x))^s 
		D_{\bT_{\fn_r(x)}^-}^x 
		\\ \notag & \hspace{5em} \times 
		B_{(\bT_{\fn_r(x)}^-)_x^-}^{\fn_r(x)} B_{\bT}^x B_{\bT}^{\fn_r(x)} \, v_{((\bT_x^-)_{\fn_r(x)}^-)_{\fn_r(x)}^-}  
	\Big\} z_1^{t_1} z_2^{t_2} w^s, 
\end{align}
\begin{align}
\label{e_iz1 ei+1w e_iz2 vbT} 
& e_i(z_1) e_{i+1} (w) e_i(z_2) \cdot v_{\bT} 
\\
&\notag 
= \sum_{t_1,t_2, s \in \ZZ} q^{ i \cdot (t_1 + t_2) + (i+1) \cdot s} 
	\\ \notag & \quad \times \Big\{  
		\sum_{( x_1, x_2, y) \in [\bT]_{\a_i} \times [\bT]_{\a_i} \times [\bT]_{\a_{i+1}} \atop x_1 \not= x_2 } 
		c_{\Bz} (x_1)^{t_1} c_{\Bz} (x_2)^{t_2} c_{\Bz} (y)^s 
		D_{\bT_{x_2}^-}^{x_1} D_{(\bT_{x_2}^-)_y^-}^{x_1} D_{\bT_{x_2}^-}^y 
		\\\notag  & \hspace{5em} \times 
		B_{\bT}^{x_1} B_{\bT}^{x_2} B_{\bT}^y \, v_{((\bT_{x_1}^-)_{x_2}^-)_y^-} 
	\\ \notag & \hspace{3em} 
	+ \sum_{(x_1,x_2) \in [\bT]_{\a_i} \times [\bT]_{\a_i} \atop x_1 \not= x_2, \, \fn_b (x_2) \in [\bT_{x_2}^-]_{\a_{i+1}}} 
		c_{\Bz} (x_1)^{t_1} c_{\Bz} (x_2)^{t_2} c_{\Bz} (\fn_b (x_2))^s 
		D_{\bT_{x_2}^-}^{x_1} D_{(\bT_{x_2}^-)_{\fn_b (x_2)}^-}^{x_1} 
		\\ \notag & \hspace{5em} \times 
		B_{\bT_{x_2}^-}^{\fn_b (x_2)} B_{\bT}^{x_1} B_{\bT}^{x_2} \, v_{((\bT_{x_1}^-)_{x_2}^-)_{\fn_b(x_2)}^-} 
	\\\notag  & \hspace{3em} 
	+ \sum_{(x,y) \in [\bT]_{\a_i} \times [\bT]_{\a_{i+1}} \atop \fn_r(x) \in [\bT_x^-]_{\a_i}} 
		c_{\Bz} (\fn_r (x))^{t_1} c_{\Bz} (x)^{t_2} c_{\Bz} (y)^s 
		D_{(\bT_{x}^-)_y^-}^{\fn_r (x)} D_{\bT_x^-}^y 
		% \\ & \hspace{5em} \times 
		B_{\bT_x^-}^{\fn_r(x)} B_{\bT}^x B_{\bT}^y \, v_{((\bT_x^-)_y^-)_{\fn_r(x)}^-} 
	\\ \notag & \hspace{3em} 
	+ \sum_{(x,y) \in [\bT]_{\a_i} \times [\bT]_{\a_{i+1}} \atop y \in [\bT_y^-]_{\a_i}} 
		c_{\Bz} (y)^{t_1} c_{\Bz}(x)^{t_2} c_{\Bz}(y)^s 
		D_{(\bT_y^-)_x^-}^y D_{\bT_x^-}^y 
		% \\ & \hspace{5em} \times 
		B_{\bT_y^-}^y B_{\bT}^x B_{\bT}^y \, v_{((\bT_x^-)_y^-)_y^-} 
	\\ \notag & \hspace{3em} 
	+ \sum_{ x \in [\bT]_{\a_i} \atop \fn_b(x) \in [\bT_x^-]_{\a_{i+1}}, \, \fn_r (x) \in [\bT_x^-]_{\a_i}} 
		c_{\Bz} (\fn_r (x))^{t_1} c_{\Bz} (x)^{t_2} c_{\Bz}( \fn_b (x))^s 
		D_{(\bT_x^-)_{\fn_b(x)}^-}^{\fn_r (x)} 
		\\ \notag & \hspace{5em} \times 
		B_{\bT_x^-}^{\fn_r (x)} B_{\bT_x^-}^{\fn_b(x)} B_{\bT}^x \, v_{((\bT_x^-)_{\fn_b(x)}^-)_{\fn_r(x)}^-} 
	\\ \notag & \hspace{3em} 
	+ \sum_{x \in [\bT]_{\a_i} \atop \fn_r(x) \in [\bT]_{\a_{i+1}} \cap [(\bT_x^-)_{\fn_r(x)}^-]_{\a_i}} 
		c_{\Bz} (\fn_r(x))^{t_1} c_{\Bz} (x)^{t_2} c_{\Bz} (\fn_r (x))^s 
		D_{\bT_x^-}^{\fn_r (x)} 
		\\ \notag & \hspace{5em} \times 
		B_{(\bT_x^-)_{\fn_r(x)}^-}^{\fn_r(x)} B_{\bT}^{\fn_r(x)} B_{\bT}^x \, v_{((\bT_x^-)_{\fn_r(x)}^-)_{\fn_r(x)}^-} 
	\Big\} z_1^{t_1} z_2^{t_2} w^s. 
\end{align}

By \eqref{ei+1w eiz1 eiz2 vbT} and \eqref{eiz1 eiz2 ei+1w vbT}, 
we have 
\begin{align}
\label{w z1 z2 z2 z1 z1 z2 z2 z1 w}
&\Big( e_{i+1} (w) \big( e_i (z_1) e_i(z_2) + e_i(z_2) e_i(z_1) \big) 
 	+ \big( e_i(z_1) e_i(z_2) + e_i(z_2) e_i (z_1) \big) e_{i+1} (w) \Big) \cdot v_{\bT} 
\\
&\notag
= \sum_{t_1,t_2, s \in \ZZ} q^{i \cdot (t_1 + t_2) + (i+1) \cdot s} 
	\\ \notag & \quad \times \Big\{  
	\sum_{(x_1,x_2,y) \in [\bT]_{\a_i} \times [\bT]_{\a_i} \times [\bT]_{\a_{i+1}} \atop x_1 \not= x_2} 
		\big( c_{\Bz}(x_1)^{t_1} c_{\Bz} (x_2)^{t_2} + c_{\Bz} (x_1)^{t_2} c_{\Bz} (x_2)^{t_1} \big) c_{\Bz} (y)^s 
		\\ \notag & \hspace{5em} \times 
		\big( D_{\bT_{x_2}^-}^{x_1} D_{(\bT_{x_2}^-)_{x_1}^-}^y D_{\bT_{x_2}^-}^y 
			+ D_{\bT_y^-}^{x_1} D_{\bT_y^-}^{x_2} D_{(\bT_y^-)_{x_2}^-}^{x_1} \big)  
		B_{\bT}^{x_1} B_{\bT}^{x_2} B_{\bT}^y 
		\, v_{((\bT_{x_1}^-)_{x_2}^-)_y^-} 
	\\ \notag & \hspace{3em} 
	+ \sum_{(x_1,x_2) \in [\bT]_{\a_i} \times [\bT]_{\a_i} \atop x_1 \not= x_2, \, \fn_b(x_2) \in [\bT_{x_2}^-]_{\a_{i+1}}} 
		\big( c_{\Bz}(x_1)^{t_1} c_{\Bz} (x_2)^{t_2} + c_{\Bz} (x_1)^{t_2} c_{\Bz} (x_2)^{t_1} \big) c_{\Bz} (\fn_b (x_2))^s 
		\\ \notag & \hspace{5em} \times 
		\big( D_{\bT_{x_2}^-}^{x_1} D_{(\bT_{x_2}^-)_{x_1}^-}^{\fn_b (x_2)} 
			+ D_{\bT_{x_1}^-}^{x_2} D_{(\bT_{x_2}^-)_{x_1}^-}^{\fn_b (x_2)} \big) 
		B_{\bT}^{x_1} B_{\bT}^{x_2} B_{\bT_{x_2}^-}^{\fn_b(x_2)} 
		\, v_{((\bT_{x_1}^-)_{x_2}^-)_{\fn_b(x_2)}^-} 
	\\ \notag & \hspace{3em} 
	+ \sum_{(x,y) \in [\bT]_{\a_i} \times [\bT]_{\a_{i+1}} \atop \fn_r (x) \in [\bT_x^-]_{\a_i}} 
		\big( c_{\Bz} (\fn_r (x))^{t_1} c_{\Bz} (x)^{t_2} + c_{\Bz} (\fn_r(x))^{t_2} c_{\Bz} (x)^{t_1} \big) c_{\Bz} (y)^s 
		\\ \notag & \hspace{5em} \times 
		\big( D_{(\bT_x^-)_{\fn_r(x)}^-}^y D_{\bT_x^-}^y + D_{(\bT_x^-)_y^-}^{\fn_r(x)} D_{\bT_y^-}^x \big) 
		B_{\bT}^x B_{\bT}^y B_{\bT_x^-}^{\fn_r(x)} 
		\, v_{((\bT_x^-)_y^-)_{\fn_r(x)}^-} 
	\\ \notag & \hspace{3em} 
	+ \sum_{(x,y) \in [\bT]_{\a_i} \times [\bT]_{\a_{i+1}} \atop y \in [\bT_y^-]_{\a_i}} 
		\big( c_{\Bz} (x)^{t_1} c_{\Bz} (y)^{t_2} + c_{\Bz}(x)^{t_2} c_{\Bz} (y)^{t_1} \big) c_{\Bz} (y)^s 
		\\ \notag & \hspace{5em} \times 
		\big( D_{\bT_y^-}^x D_{(\bT_y^-)_y^-}^x + D_{\bT_y^-}^x D_{(\bT_y^-)_x^-}^y \big)
		B_{\bT}^x B_{\bT}^y B_{\bT_y^-}^y 
		\, v_{((\bT_x^-)_y^-)_y^-} 
	\\ \notag & \hspace{3em} 
	+ \sum_{x \in [\bT]_{\a_i} \atop \fn_r(x) \in [\bT_x^-]_{\a_i}, \, \fn_b(x) \in [\bT_x^-]_{\a_{i+1}}} 
		\big( c_{\Bz} (\fn_r (x))^{t_1} c_{\Bz} (x)^{t_2} + c_{\Bz} (\fn_r (x))^{t_2} c_{\Bz} (x)^{t_1} \big) c_{\Bz} (\fn_b(x))^s  
		\\ \notag & \hspace{5em} \times 
		D_{(\bT_x^-)_{\fn_r (x)}^-}^{\fn_b(x)} 
		B_{\bT}^x B_{\bT_x^-}^{\fn_r(x)} B_{\bT_x^-}^{\fn_b(x)} 
		\, v_{((\bT_x^-)_{\fn_r(x)}^-)_{\fn_b(x)}^-} 
	\\ \notag & \hspace{3em} 
	+ \sum_{x \in [\bT]_{\a_i} \atop \fn_r(x) \in [\bT]_{\a_{i+1}}\cap [(\bT_x^-)_{\fn_r(x)}^-]_{\a_i}} 
		\big( c_{\Bz} (\fn_r (x))^{t_1} c_{\Bz} (x)^{t_2} + c_{\Bz}(\fn_r(x))^{t_2} c_{\Bz} (x)^{t_1} \big) c_{\Bz} (\fn_r(x))^s 
		\\ \notag & \hspace{5em} \times 
		D_{\bT_{\fn_r(x)}^-}^x B_{\bT}^x B_{\bT}^{\fn_r(x)} B_{(\bT_{\fn_r(x)}^-)_x^-}^{\fn_r(x)} 
		\, v_{((\bT_x^-)_{\fn_r(x)}^-)_{\fn_r(x)}^-} 
	\Big\} z_1^{t_1} z_2^{t_2} w^s. 
\end{align}
By \eqref{e_iz1 ei+1w e_iz2 vbT}, we also have 
\begin{align}
\label{z1 w z2 z2 w z1}
& \Big( e_i(z_1) e_{i+1} (w) e_i(z_2) + e_i(z_2) e_{i+1} (w) e_i(z_1) \Big) \cdot v_{\bT} 
\\
&\notag 
= \sum_{t_1,t_2,s \in \ZZ} q^{i \cdot (t_1 + t_2) + (i+1) \cdot s} 
	\\ \notag & \quad \times \Big\{ 
	\sum_{( x_1, x_2, y) \in [\bT]_{\a_i} \times [\bT]_{\a_i} \times [\bT]_{\a_{i+1}} \atop x_1 \not= x_2 } 
		\big( c_{\Bz} (x_1)^{t_1} c_{\Bz} (x_2)^{t_2} + c_{\Bz} (x_1)^{t_2} c_{\Bz} (x_2)^{t_1} \big) c_{\Bz} (y)^s 
		\\ \notag & \hspace{5em} \times 
		D_{\bT_{x_2}^-}^{x_1} D_{(\bT_{x_2}^-)_y^-}^{x_1} D_{\bT_{x_2}^-}^y 
		B_{\bT}^{x_1} B_{\bT}^{x_2} B_{\bT}^y 
		\, v_{((\bT_{x_1}^-)_{x_2}^-)_y^-} 
	\\ \notag & \hspace{3em} 
	+ \sum_{(x_1,x_2) \in [\bT]_{\a_i} \times [\bT]_{\a_i} \atop x_1 \not= x_2, \, \fn_b (x_2) \in [\bT_{x_2}^-]_{\a_{i+1}}} 
		\big( c_{\Bz} (x_1)^{t_1} c_{\Bz} (x_2)^{t_2} + c_{\Bz} (x_1)^{t_2} c_{\Bz} (x_2)^{t_1} \big) c_{\Bz} (\fn_b (x_2))^s 
		\\ \notag & \hspace{5em} \times 
		D_{\bT_{x_2}^-}^{x_1} D_{(\bT_{x_2}^-)_{\fn_b (x_2)}^-}^{x_1} 
		B_{\bT}^{x_1} B_{\bT}^{x_2}  B_{\bT_{x_2}^-}^{\fn_b (x_2)} 
		\, v_{((\bT_{x_1}^-)_{x_2}^-)_{\fn_b(x_2)}^-} 
	\\ \notag & \hspace{3em} 
	+ \sum_{(x,y) \in [\bT]_{\a_i} \times [\bT]_{\a_{i+1}} \atop \fn_r(x) \in [\bT_x^-]_{\a_i}} 
		\big( c_{\Bz} (\fn_r (x))^{t_1} c_{\Bz} (x)^{t_2} +  c_{\Bz} (\fn_r (x))^{t_2} c_{\Bz} (x)^{t_1} \big) c_{\Bz} (y)^s 
		\\ \notag & \hspace{5em} \times 
		D_{(\bT_{x}^-)_y^-}^{\fn_r (x)} D_{\bT_x^-}^y 
		B_{\bT}^x B_{\bT}^y B_{\bT_x^-}^{\fn_r(x)}  
		\, v_{((\bT_x^-)_y^-)_{\fn_r(x)}^-} 
	\\ \notag & \hspace{3em} 
	+ \sum_{(x,y) \in [\bT]_{\a_i} \times [\bT]_{\a_{i+1}} \atop y \in [\bT_y^-]_{\a_i}} 
		\big( c_{\Bz}(x)^{t_1} c_{\Bz} (y)^{t_2}  + c_{\Bz}(x)^{t_2} c_{\Bz} (y)^{t_1}   \big) c_{\Bz}(y)^s 
		\\ \notag & \hspace{5em} \times 
		D_{(\bT_y^-)_x^-}^y D_{\bT_x^-}^y 
		B_{\bT}^x B_{\bT}^y B_{\bT_y^-}^y 
		\, v_{((\bT_x^-)_y^-)_y^-} 
	\\ \notag & \hspace{3em} 
	+ \sum_{ x \in [\bT]_{\a_i} \atop\fn_r (x) \in [\bT_x^-]_{\a_i}, \, \fn_b(x) \in [\bT_x^-]_{\a_{i+1}} } 
		\big( c_{\Bz} (\fn_r (x))^{t_1} c_{\Bz} (x)^{t_2} + c_{\Bz} (\fn_r (x))^{t_2} c_{\Bz} (x)^{t_1} \big) c_{\Bz}( \fn_b (x))^s 
		\\ \notag & \hspace{5em} \times 
		D_{(\bT_x^-)_{\fn_b(x)}^-}^{\fn_r (x)} 
		B_{\bT}^x  B_{\bT_x^-}^{\fn_r (x)} B_{\bT_x^-}^{\fn_b(x)} 
		\, v_{((\bT_x^-)_{\fn_r(x)}^- )_{\fn_b(x)}^-}
	\\ \notag & \hspace{3em} 
	+ \sum_{x \in [\bT]_{\a_i} \atop \fn_r(x) \in [\bT]_{\a_{i+1}} \cap [(\bT_x^-)_{\fn_r(x)}^-]_{\a_i}} 
		\big( c_{\Bz} (\fn_r(x))^{t_1} c_{\Bz} (x)^{t_2} + c_{\Bz} (\fn_r(x))^{t_2} c_{\Bz} (x)^{t_1} \big) c_{\Bz} (\fn_r (x))^s 
		\\ \notag & \hspace{5em} \times 
		D_{\bT_x^-}^{\fn_r (x)} 
		B_{\bT}^x  B_{\bT}^{\fn_r(x)}  B_{(\bT_x^-)_{\fn_r(x)}^-}^{\fn_r(x)} 
		\, v_{((\bT_x^-)_{\fn_r(x)}^-)_{\fn_r(x)}^-} 
	\Big\} z_1^{t_1} z_2^{t_2} w^s. 
\end{align}
Furthermore, 
we can obtain the following equations by direct calculations. 
\begin{itemize}
\item 
For $(x_1,x_2,y), (x_2,x_1,y) \in [\bT]_{\a_i} \times [\bT]_{\a_i} \times [\bT]_{\a_{i+1}}$ such that $x_1 \not=x_2$, 
we have $v_{((\bT_{x_1}^-)_{x_2}^-)_y^-} = v_{((\bT_{x_2}^-)_{x_1}^-)_y^-}$, and 
\begin{align*}
&\big( D_{\bT_{x_2}^-}^{x_1} D_{(\bT_{x_2}^-)_{x_1}^-}^y D_{\bT_{x_2}^-}^y 
		+ D_{\bT_y^-}^{x_1} D_{\bT_y^-}^{x_2} D_{(\bT_y^-)_{x_2}^-}^{x_1} \big) 
	+ \big( D_{\bT_{x_1}^-}^{x_2} D_{(\bT_{x_1}^-)_{x_2}^-}^y D_{\bT_{x_1}^-}^y 
		+ D_{\bT_y^-}^{x_2} D_{\bT_y^-}^{x_1} D_{(\bT_y^-)_{x_1}^-}^{x_2} \big) 
\\
&= [2] \Big( \frac{ c_{\Bz}(y) - c_{\Bz}(x_1) }{ q c_{\Bz}(y) - q^{-1} c_{\Bz} (x_1) } 
			\cdot \frac{ c_{\Bz}(y) - c_{\Bz} (x_2) }{ q c_{\Bz} (y) - q^{-1} c_{\Bz} (x_2) } +1 \Big)
\\
&= [2] \big( D_{\bT_{x_2}^-}^{x_1} D_{(\bT_{x_2}^-)_y^-}^{x_1} D_{\bT_{x_2}^-}^y 
			+  D_{\bT_{x_1}^-}^{x_2} D_{(\bT_{x_1}^-)_y^-}^{x_2} D_{\bT_{x_1}^-}^y \big). 
\end{align*}

%%%%
\item 
For $(x_1,x_2) \in [\bT]_{\a_i} \times [\bT]_{\a_i}$ such that $x_1 \not=x_2$ and $\fn_b (x_2) \in [\bT_{x_2}^-]_{\a_{i+1}}$, 
\begin{align*}
D_{\bT_{x_2}^-}^{x_1} D_{(\bT_{x_2}^-)_{x_1}^-}^{\fn_b(x_2)} + D_{\bT_{x_1}^-}^{x_2} D_{(\bT_{x_2}^-)_{x_1}^-}^{\fn_b (x_2)} 
&= [2] \frac{  q c_{\Bz} (x_1) - q^{-1} c_{\Bz} (x_2)  }{  c_{\Bz} (x_1) - c_{\Bz} (x_2) } 
= [2] D_{\bT_{x_2}^-}^{x_1} D_{(\bT_{x_2}^-)_{\fn_b(x_2)}^-}^{x_1}.  
\end{align*}

%%%
\item  
For $(x,y) \in [\bT]_{\a_i} \times [\bT]_{\a_{i+1}} $ such that  $\fn_r (x) \in [\bT_x^-]_{\a_i}$, 
\begin{align*}
D_{(\bT_x^-)_{\fn_r(x)}^-}^y D_{\bT_x^-}^y + D_{(\bT_x^-)_y^-}^{\fn_r(x)} D_{\bT_y^-}^x 
= [2] \frac{ c_{\Bz} (y) - c_{\Bz}(x)}{ q c_{\Bz}(y) - q^{-1} c_{\Bz} (x)} 
=[2] D_{(\bT_x^-)_y^-}^{\fn_r(x)} D_{\bT_x^-}^y. 
\end{align*}

%%%
\item 
For $(x,y) \in [\bT]_{\a_i} \times [\bT]_{\a_{i+1}}$ such that $y \in [\bT_y^-]_{\a_i}$, 
\begin{align*}
D_{\bT_y^-}^x D_{(\bT_y^-)_y^-}^x + D_{\bT_y^-}^x D_{(\bT_y^-)_x^-}^y 
= [2] 
= [2] D_{(\bT_y^-)_x^-}^y D_{\bT_x^-}^y. 
\end{align*}

%%%
\item 
For $x \in [\bT]_{\a_i}$ such that $\fn_r(x) \in [\bT_x^-]_{\a_i}$ and $\fn_b(x) \in [\bT_x^-]_{\a_{i+1}}$, 
\begin{align*}
D_{(\bT_x^-)_{\fn_r(x)}^-}^{\fn_b(x)} 
= [2] 
=[2] D_{(\bT_x^-)_{\fn_b(x)}^-}^{\fn_r(x)}. 
\end{align*}

%%%
\item  
For $x \in [\bT]_{\a_i}$ such that $\fn_r(x) \in [\bT]_{\a_{i+1}} \cap [(\bT_x^-)_{\fn_r(x)}^-]_{\a_i}$, 
\begin{align*}
D_{\bT_{\fn_r(x)}^-}^x =1 = [2] D_{\bT_x^-}^{\fn_r(x)}. 
\end{align*}
\end{itemize} 

Combining these equations with \eqref{w z1 z2 z2 z1 z1 z2 z2 z1 w} and \eqref{z1 w z2 z2 w z1}, 
we conclude that 
\begin{align*}
&\Big( e_{i+1} (w) \big( e_i (z_1) e_i(z_2) + e_i(z_2) e_i(z_1) \big) 
 	+ \big( e_i(z_1) e_i(z_2) + e_i(z_2) e_i (z_1) \big) e_{i+1} (w) \Big) \cdot v_{\bT} 
\\
&= [2] \Big( e_i(z_1) e_{i+1} (w) e_i(z_2) + e_i(z_2) e_{i+1} (w) e_i(z_1) \Big) \cdot v_{\bT}. 
\end{align*}

We can check another relation in (U7) and also relations in (U8) in a similar way.

%%%%%%%%%%%%%%%%%%%%%%%%%%%%%%%%%%%%%%%%%%%%%%%%%%%%%%%%%%%%%%%

%%%%%%%%%%%%%%%%%%%%%%%%%%%%%%%%%%%%%%%%%%%%%%%%%%%%%%%%%%%%%%%

\section{Cyclotomic $q$-Schur algebras} 

In this section, we recall the definition of cyclotomic $q$-Schur algebras and their cellular basis given in \cite{DJM98}, 
see also \cite{Mat04} and \cite{JM00}. 
Here, we modify some of combinatorial notation in \cite{DJM98} 
in order to discuss the connection with shifted quantum affine algebras through the Schur-Weyl duality obtained in \cite{Wad24}. 

\para 
We take parameters $q, Q_0,Q_1,\dots, Q_{r-1} \in \CC^{\times}$, and put $\bQ=(Q_0,\dots, Q_{r-1})$. 
We assume that $q$ is not a root of unity. 
The Ariki-Koike algebra $\sH_{n,r} = \sH_{n,r}(q,\bQ)$ associated with the complex reflection group 
$\fS_n \ltimes (\ZZ / r \ZZ)^n$ of type $G(r,1,n)$  is an associative algebra over $\CC$ generated by 
$T_0, T_1,\dots, T_{n-1}$ subject to the following defining relations: 
\begin{align*}
&(T_0- Q_0) (T_0-Q_1) \dots (T_0-Q_{r-1})=0, 
	\quad 
	(T_i-q) (T_i + q^{-1})=0 
	\quad (1 \leq i \leq n-1), 
\\
& T_0 T_1 T_0 T_1 = T_1 T_0 T_1 T_0, 
	\quad 
	T_i T_{i+1} T_i = T_{i+1} T_i T_{i+1} \quad (1 \leq i \leq n-2), 
\\
& T_i T_j = T_j T_i \quad (|i-j| >1). 
\end{align*} 

Let $\ast : \sH_{n,r} \ra \sH_{n,r}$ ($h \mapsto  h^{\ast}$) be the algebra anti-automorphism such that 
$T_i^{\ast} = T_i$ ($0 \leq i \leq n-1$). 

For $1 \leq i \leq n-1$, let $s_i =(i,i+1)$ be the $i$-th adjacent transposition. 
For $w \in \fS_n$, put $T_w = T_{i_1} T_{i_2} \dots T_{i_l}$ when $w=s_{i_1} s_{i_2} \dots s_{i_l}$ is reduced expression.  

We define the elements $\TT_{i,j}, \wt{\TT}_{j,i} \in \sH_{n,r}$ ($1 \leq i \leq j \leq n-1$) by 
$\TT_{i,j} = T_i T_{i+1} \dots T_j$ and $\wt{\TT}_{j,i} = T_j T_{j-1} \dots T_i$ respectively. 
For convenience, we put $\TT_{i,i-1} = \wt{\TT}_{i-1,i} =1$. 
We also put 
$L_i = \wt{\TT}_{i-1,1} T_0 \TT_{i,i-1}$ ($1 \leq i \leq n$). 

%%%%
\para 
Put $\w=(\underbrace{1,1,\dots,1}_{n})$. 
For $\Bla \in \vL_{n,r}^+$, let $\Tab_{\w}(\Bla)$ be the set of $\Bla$-tableaux with weight $\w$. 
Then, an element $\bt \in \Tab_{\w}(\Bla)$ is a bijection $\bt: [\Bla] \ra \{1,2,\dots,n\}$. 
We say that a $\Bla$-tableau $\bt $  is standard if $\bt$ is semi-standard and belongs to $\Tab_w(\Bla)$. 
Let $\Std(\Bla)$ be the set of standard tableaux of shape $\Bla$. 
For $\Bla \in \vL_{n,r}^+$, we denote by $\bt^{\Bla}$ the standard $\Bla$-tableau given by 
\begin{align*}
\bt^{\Bla} ((a,b,c)) = \sum_{k=1}^{c-1} |\la^{(k)}| + \sum_{i=1}^{a-1} \la_i^{(c)} + b 
\quad ((a,b,c) \in [\Bla]).
\end{align*}
For example, 
if $\Bla =((3,2), (2,2,1), (4,1) ) \in \vL_{15,3}^+$, we have 
\begin{align*}
\bt^{\Bla}
= \left( \, \begin{array}{|c|c|c|} \hline 1 & 2 & 3 \\ \hline 4 & 5 \\ \cline{1-2} \multicolumn{1}{c}{} \end{array} \, , 
	\, 
	\begin{array}{|c|c|c|} \hline 6 & 7 \\ \hline 8 & 9 \\ \hline 10 \\ \cline{1-1} \end{array} \, , 
	\, 
	\begin{array}{|c|c|c|c|} \hline 11 & 12 & 13 & 14 \\ \hline 15 \\ \cline{1-1} \multicolumn{1}{c}{} \end{array} 
	\, 
	\right). 
\end{align*}

Let $\fS_n$ be the symmetric group on $\{1,2,\dots, n\}$, 
where $\fS_n$ acts on $\{1,2,\dots, n\}$ from right. 
The symmetric group $\fS_n$ acts on $\Tab_\w(\Bla)$ from right 
by permuting the numbers 
$\bt ((a,b,c))$ for  $(a,b,c) \in [\Bla]$. 
Namely, we have $(\bt \cdot \s) ((a,b,c)) = \bt((a,b,c)) \cdot \s $ 
for $\bt \in \Tab_\w(\Bla)$, $\s \in \fS_n$ and $(a,b,c) \in [\Bla]$. 
Then, for $\bt \in \Tab_\w (\Bla) $, we can define the permutation $d(\bt) \in \fS_n$ by $\bt= \bt^{\Bla} \cdot d (\bt)$. 

For $\Bla \in \vL_{n,r}^+$ and $\bfs, \bt \in \Std (\Bla)$, 
we define the element $m_{\bfs \bt} \in \sH_{n,r}$ by 
\begin{align*}
m_{\bfs \bt} = T_{d(\bfs)}^{\ast} m_{\Bla} T_{d (\bt)}.
\end{align*}
Then, 
the set $\{ m_{\bfs \bt} \mid \bfs, \bt \in \Std (\Bla), \,\Bla \in \vL_{n,r}^+\}$ 
gives a cellular basis of $\sH_{n,r}$ by \cite[Theorem 3.26]{DJM98}. 
%%%%%

\para  
For $\bm=(m_1,m_2,\dots, m_r) \in \ZZ_{>0}^r$ such that $m_1+m_2+\dots + m_r =m$, 
put $m_{\lan 0 \ran} =0$ and $m_{\lan k \ran} = \sum_{j=1}^k m_j$ ($1 \leq k \leq r$).  
For $\mu=(\mu_1, \dots, \mu_m) \in \vL_n(m)$ and $ 0 \leq k \leq r$, 
set $a_k^{\mu} (\bm) = \sum_{i=1}^{m_{\lan k \ran}} \mu_i$. 
We define an element $m_{\mu} \in \sH_{n,r}$ ($\mu \in \vL_n(m)$)  by 
\begin{align*}
m_{\mu} = \Big( \sum_{ w \in \fS_{\mu}} q^{\ell (w)} T_w \Big) \Big( \prod_{k=1}^{r-1} \prod_{i=1}^{a_k^{\mu} (\bm)} (L_i - Q_k) \Big), 
\end{align*} 
where $\fS_{\mu}$ is the parabolic subgroup of $\fS_n$ with respect to the composition $\mu$, 
and $\ell(w)$ is the length of $w$. 
Then, we define the cyclotomic $q$-Schur algebra 
$\sS_{n,r}(\bm)$ with respect to $\bm$ by 
\begin{align*}
\sS_{n,r} (\bm) = \End_{\sH_{n,r}^{\opp}} \Big( \bigoplus_{\mu \in \vL_n(m)} M_{[\bm]}^{\mu} \Big), 
\text{ where } M_{[\bm]}^{\mu} = m_{\mu} \sH_{n,r}.
\end{align*}

%%%
\para 
We take $\bm=(m_1,m_2,\dots, m_r) \in (\ZZ_{>0})^r$ such that $m_1+m_2+\dots + m_r =m$. 
For $\Bla \in \vL_{n,r}^+$ and $\Bla$-tableau $\bT$, 
we say that $\bT$ is semi-standard with respect to $\bm=(m_1, \dots, m_r)$ if 
$\bT$ is a semi-standard tableau satisfying the following additional condition 
\begin{align*}
\bT((a,b,c)) > m_{\lan c-1 \ran} =m_1+\dots+m_{c-1} 
\text{ for } (a,b,c) \in [\Bla].  
\end{align*}

The semi-standard tableau in \eqref{ex semi-std} is not semi-standard with respect to $(1,1,2)$, 
but the tableau 
\begin{align*}
\bS'= \left( \, \begin{array}{|c|c|c|} \hline 1 & 2 & 2 \\ \hline 2 & 4 \\ \cline{1-2} \multicolumn{1}{c}{} \end{array} \, , 
	\, 
	\begin{array}{|c|c|c|} \hline 2 & 2 \\ \hline 3 & 4 \\ \hline 4 \\ \cline{1-1} \end{array} \, , 
	\, 
	\begin{array}{|c|c|c|c|} \hline 3 & 3 & 4 & 4 \\ \hline 4 \\ \cline{1-1} \multicolumn{1}{c}{} \end{array} 
	\, 
	\right) 
\end{align*}
is semi-standard with respect to $(1,1,2)$.

For $\Bla \in \vL_{n,r}^+$ and $\mu \in \vL_n(m)$,  
let $\SStd^{[\bm]}(\Bla,\mu)$ be the set of semi-standard tableaux with respect to $\bm$ 
of shape $\Bla$ with weight $\mu$, 
and put 
\begin{align*}
\SStd_m^{[\bm]} (\Bla) = \bigcup_{\mu \in \vL_n(m)} \SStd^{[\bm]} (\Bla, \mu).
\end{align*}  
We also put 
\begin{align*}
&\wt{\vL}_{n,r}^+[\bm] = \{ \Bla \in \vL_{n,r}^+ \mid \SStd_m^{[\bm]} (\Bla) \not=\emptyset \}, 
\\
&\vL_{n,r}^+[\bm] = \{\Bla \in \vL_{n,r}^+ \mid \ell (\la^{(k)}) \leq m_k \text{ for all } 1  \leq k \leq r\}.
\end{align*}

For $\Bla \in \vL_{n,r}^+[\bm]$, 
we define the $\Bla$-tableau $\bT^{\Bla}_{[\bm]}$ by 
$\bT_{[\bm]}^{\Bla}((a,b,c)) = m_{\lan c-1 \ran} +a$ 
for $(a,b,c) \in [\Bla]$. 
For  example, 
for $\Bla =((3,2), (2,2,1), (4,1)) \in \vL_{15,3}^+$ and $\bm=(2,4,2)$, we have 
\begin{align*}
\bT_{[\bm]}^{\Bla}
= \left( \, \begin{array}{|c|c|c|} \hline 1 & 1 & 1 \\ \hline 2 & 2 \\ \cline{1-2} \multicolumn{1}{c}{} \end{array} \, , 
	\, 
	\begin{array}{|c|c|c|} \hline 3 & 3 \\ \hline 4 & 4 \\ \hline 5 \\ \cline{1-1} \end{array} \, , 
	\, 
	\begin{array}{|c|c|c|c|} \hline 7 & 7 & 7 & 7 \\ \hline 8 \\ \cline{1-1} \multicolumn{1}{c}{} \end{array} 
	\, 
	\right). 
\end{align*}
Then, we see that $\bT_{[\bm]}^{\Bla} \in \SStd_m^{[\bm]} (\Bla)$, 
and this implies that $\vL_{n,r}^+[\bm] \subset \wt{\vL}_{n,r}^+[\bm]$.  
It is clear that 
$\vL_{n,r}^+ = \vL_{n,r}^+[\bm]$ if $m_k \geq n$ for all $1 \leq k \leq r$. 
%%%

\remarks\ 
\begin{enumerate} 
\item 
Put $\vG[\bm] =\{(i,k) \in \ZZ^2 \mid 1 \leq i \leq m_k, \, 1 \leq k \leq r\}$, 
and we identify the set $\vG[\bm]$ with the set $\{1,2, \dots, m\}$ by the bijection 
\begin{align*}
\xi : \vG[\bm] \ra \{1,2,\dots, m\}, \quad (i,k) \mapsto \sum_{p=1}^{k-1} m_p +i.
\end{align*}
Then, for $\bT \in \SStd_m^{[\bm]}(\Bla)$, 
the map $\xi^{-1} \circ \bT : [\Bla] \ra \vG[\bm]$ is a semi-standard tableau in the sense of \cite{DJM98}. 

\item 
We have the injective map 
\begin{align*}
\hat{\xi} : \, & \vL_{n,r}^+[\bm] \ra \vL_n(m), 
\\
&\Bla =(\la^{(1)}, \dots, \la^{(r)}) 
\mapsto ( \la^{(1)}_1,\dots, \la^{(1)}_{m_1}, \la_1^{(2)},\dots, \la_{m_2}^{(2)}, \dots, \la_1^{(r)}, \dots, \la_{m_r}^{(r)}), 
\end{align*}
and we regard $\vL_{n,r}^+[\bm]$ as a subset of $\vL_n(m)$ by this injection. 
Then we see that $\owt \bT_{[\bm]}^{\Bla} = \hat{\xi} (\Bla)$. 
\end{enumerate}
%%%%

\para 
For $\mu =(\mu_1,\mu_2,\dots, \mu_m)\in \vL_n(m)$, 
put $N_0^{\mu} =0$ and $N_k^{\mu} = \sum_{i=1}^k \mu_i$ ($1 \leq k \leq m$). 
We note that $N_m^{\mu} =n$. 
For $\Bla \in \vL_{n,r}^+$, $\mu \in \vL_n(m)$ and $\bt \in \Std(\Bla)$, 
we define the $\Bla$-tableau $\mu (\bt)$ by 
\begin{align*}
\mu(\bt) ((a,b,c)) = k 
\text{ if } N_{k-1}^{\mu} < \bt((a,b,c)) \leq N_k^{\mu}. 
\end{align*}
By definition, we have $\owt (\mu(\bt)) = \mu$. 

For $\Bla \in \vL_{n,r}^+$, $\mu,\nu \in \vL_n(m)$, $\bS \in \SStd^{[\bm]}(\Bla,\mu)$ and $\bT \in \SStd^{[\bm]} (\Bla, \nu)$, 
we define the element $m_{\bS \bT} \in \sH_{n,r}$ by  
\begin{align*}
m_{\bS \bT} = \sum_{\bfs, \bt \in \Std(\Bla) \atop \bS = \mu (\bfs), \, \bT = \nu(\bt)} 
q^{\ell (d (\bfs)) + \ell (d (\bt))} m_{\bfs \bt}, 
\end{align*}
and also define the element $\vf_{\bS \bT} \in \sS_{n,r}(\bm)$ by 
\begin{align*}
\vf_{\bS \bT} (m_{\tau} h)= \d_{\tau, \nu} m_{\bS \bT} h 
\quad (\tau \in \vL_n(m), \, h \in \sH_{n,r}).
\end{align*}

We define the partial order $\trreq$ on $\vL_{n,r}^+$ by 
\begin{align*}
\Bla \trreq \Bm  \text{ if and only if } 
\sum_{k=1}^{c-1} |\la^{(k)}| + \sum_{i=1}^j \la^{(c)}_i 
	\geq \sum_{k=1}^{c-1} |\mu^{(k)}| + \sum_{i=1}^j \mu_i^{(c)} 
\end{align*}
for all $1 \leq c \leq r$ and $j \geq 0$. 
Then, 
the set 
\begin{align*}
\{\vf_{\bS \bT} \mid \bS, \bT \in \SStd_m^{[\bm]}(\Bla), \, \Bla \in \wt{\vL}_{n,r}^+[\bm] \}
\end{align*} 
gives a cellular basis of $\sS_{n,r}(\bm)$ with respect to the order $\trreq$ on $\wt{\vL}_{n,r}^+[\bm]  \subset \vL_{n,r}^+$ 
by \cite[Theorem 6.6 and Theorem 6.12]{DJM98} (see also \cite[Theorem 4.11]{Mat04}). 

%%%
\para 
By general theory of cellular algebras in \cite{GL96}, we can consider the cell modules of $\sS_{n,r}(\bm)$ as follows. 

For $\Bla \in \wt{\vL}_{n,r}^+[\bm]$, 
let $\sS_{n,r}^{\trr \Bla}$ be the subspace of $\sS_{n,r}(\bm)$ spanned by 
\begin{align*}
\{\vf_{\bS \bT} \mid \bS, \bT \in \SStd_m^{[\bm]}(\Bm) \text{ for } \Bm \in \wt{\vL}_{n,r}^+[\bm] \text{ such that } 
	\Bm \trr \Bla \}.
\end{align*}
Then, the subspace $\sS_{n,r}^{\trr \Bla}$ becomes a two-sided ideal of $\sS_{n,r}(\bm)$. 

For $\Bla \in \vL_{n,r}^+[\bm]$ and $\bT \in \SStd_m^{[\bm]}(\Bla)$, 
we denote by $\vf_{\bT}$ 
the element $\vf_{\bT \bT^{\Bla}_{[\bm]}} + \sS_{n,r}^{\trr \Bla}$ of $\sS_{n,r}(\bm) / \sS_{n,r}^{\trr \Bla}$. 
Let $W(\Bla)$ be the subspace of $\sS_{n,r}(\bm)/ \sS_{n,r}^{\trr \Bla}$ spanned by 
$\{\vf_{\bT} \mid \bT \in \SStd_m^{[\bm]}(\Bla)\}$, 
and $W(\Bla)$ becomes a left $\sS_{n,r}(\bm)$-submodule of $\sS_{n,r}(\bm)/ \sS_{n,r}^{\trr \Bla}$. 
We call it the cell module of $\sS_{n,r}(\bm)$. 

%%%%%%%%%%%%%%%%
\remark 
In the case where $m_k \geq n$ for all $1 \leq k \leq r$, 
we see that $\vL_{n,r}^+ = \vL_{n,r}^+[\bm]$, 
and the cyclotomic $q$-Schur algebra $\sS_{n,r}(\bm)$ is a quasi-hereditary algebra 
(see \cite[Theorem 4.14]{Mat04}). 
This case is the most interesting in the representation theory of cyclotomic $q$-Schur algebras. 
However, in this paper, we are interested in the representations of shifted quantum affine algebras. 
Thus, we also consider cases without this condition.  

%%%%%%%%%%%%%%%%%%%%%%%%%%%%%%%%%%%%%%%%%%%%%%%%%%%%%%%%%%%%%%%

%%%%%%%%%%%%%%%%%%%%%%%%%%%%%%%%%%%%%%%%%%%%%%%%%%%%%%%%%%%%%%%

\section{Cell modules of  cyclotomic $q$-Schur algebras through the Schur-Weyl dualities}

In this section, 
we regard the cell module $W(\Bla)$ ($\Bla \in \vL_{n,r}^+[\bm]$) of $\sS_{n,r}(\bm)$ 
as a module of shifted quantum affine algebra $U_{q,[\bb_{\bm}]}(L\Fsl_m)$ 
through the Schur-Weyl duality in \cite{Wad24}. 
Then, we give the $q$-character of $W(\Bla)$ explicitly using combinatorics. 
We also give the explicit action of $U_{q,[\bb_{\bm}]}$ on $W(\Bla)$ 
in the case where the parameters $\bQ=(Q_0,\dots, Q_{r-1})$ satisfies the separation condition. 

%%%%%
\para \label{uqnnm Scnr} 
For $\bm =(m_1,m_2,\dots, m_r) \in (\ZZ_{>0})^r$ such that $m=m_1+m_2+\dots + m_r$, 
we take $\bb_{\bm} =(b_1,b_2,\dots, b_{m-1}) \in \{0,1\}^{m-1}$ as 
\begin{align*}
b_i = \begin{cases} 
	1 & \text{ if } \xi^{-1} (i) = (m_k,k) \text{ for some } k, 
	\\
	0 & \text{ otherwise}, 
\end{cases} 
\end{align*}
and we consider the shifted quantum affine algebra $U_{q,[\bb_{\bm}]}= U_{q,[\bb_{\bm}]} (L \Fsl_m)$ of the shift $\bb_{\bm}$. 
By \cite[(8.6.2) and Corollary 8.8]{Wad24}, we have a homomorphism 
\begin{align*}
\rho_{n,r} : U_{q,[\bb_{\bm}]} \ra \sS_{n,r}(\bm)
\end{align*} 
which follows from the action given in \cite[Proposition 8.7]{Wad24}.  
We  regard a $\sS_{n,r}(\bm)$-module as a $U_{q,[\bb_{\bm}]}$-module through the homomorphism $\rho_{n,r}$. 
We remark that the left action of $U_{q,[\bb_{\bm}]}$ on 
$\bigoplus_{\mu \in \vL_n(m)} M_{[\bm]}^{\mu} = \bigoplus_{\mu \in \vL_n(m)} m_{\mu} \sH_{n,r}$ 
commutes with the right action of $\sH_{n,r}$.  
The following proposition is obtained from \cite[Proposition 8.7]{Wad24} together with \eqref{ei spm1 h e} and \eqref{psiis+} 
using the same computation as in \cite[\S7]{Wad16}. 

%%%
\begin{prop} 
\label{Prop act uqbm Mmu} 
For $\mu \in \vL_n(m)$, $1 \leq i \leq m-1$, $t \in \ZZ$ and $s \in \ZZ_{> 0}$, we have the following. 
\begin{enumerate}
\item 
If $b_i=0$, we have 
\begin{align*}
&e_{i,t} \cdot m_{\mu} 
	= q^{t \cdot i} q^{- \mu_{i+1} +1}  m_{\mu + \a_i} L_{N_i^{\mu}+1}^t 
		\Big( \sum_{k=1}^{\mu_{i+1}} q^{k-1} \TT_{N_i^{\mu} +1, \, N_i^{\mu} + k-1}\Big), 
\\
&f_{i,t} \cdot m_{\mu} 
	= q^{t \cdot i} q^{- \mu_i +1} m_{\mu - \a_i} L_{N_i^{\mu}}^t 
		\Big( \sum_{k=1}^{\mu_i} q^{k-1} \wt{\TT}_{N_i^{\mu}-1, \, N_i^{\mu} -k +1} \Big), 
\\
&\psi_{i, s}^+ \cdot m_{\mu} 
	\\ &= \begin{cases}
		q^{\mu_i - \mu_{i+1}} m_{\mu} &  \text{ if } s=0, 
		\\
		q^{s \cdot i} q^{\mu_i-\mu_{i+1}} (q-q^{-1}) m_{\mu} 
			\\ \quad \times 
			\Phi_s (L_{N_{i-1}^{\mu}+1}, L_{N_{i-1}^{\mu} +2}, \dots, L_{N_{i-1}^{\mu} + \mu_i} 
					\mid L_{N_i^{\mu}+1}, L_{N_i^{\mu} +2}, \dots, L_{N_i^{\mu} + \mu_{i+1}})
			 & \text{ if } s >0,
		\end{cases} 
\\
& \psi_{i, -s}^- \cdot m_{\mu} 
	\\ &= \begin{cases}
		q^{-(\mu_i - \mu_{i+1})} m_{\mu} & \text{ if } s=0, 
		\\ 
		q^{- s \cdot i} q^{-(\mu_i - \mu_{i+1})} (q-q^{-1}) m_{\mu} 
			\\ \quad \times 
			\Phi_s (L_{N_i^{\mu} +1}^{-1}, L_{N_i^{\mu}+2}^{-1}, \dots, L_{N_i^{\mu}+\mu_{i+1}}^{-1} 
				\mid L_{N_{i-1}^{\mu} +1}^{-1}, L_{N_{i-1}^{\mu} +2}^{-1}, \dots, L_{N_{i-1}^{\mu} +\mu_i}^{-1}) 
			& \text{ if } s >0, 
		\end{cases}
\end{align*}
where we put $m_{\mu+\a_i} =0$ (resp. $m_{\mu -\a_i} =0$) if $\mu+\a_i \not\in \vL_n(m)$  (resp. $\mu - \a_i \not\in \vL_n(m)$). 

\item 
If $b_i=1$ and $\xi^{-1} (i) =(m_k,k)$, we have 
\begin{align*}
&e_{i,t} \cdot m_{\mu} 
	= q^{t \cdot i} q^{- \mu_{i+1} +1}  m_{\mu + \a_i} L_{N_i^{\mu}+1}^t 
		\Big( \sum_{k=1}^{\mu_{i+1}} q^{k-1} \TT_{N_i^{\mu} +1, \, N_i^{\mu} + k-1}\Big), 
\\
&f_{i,t} \cdot m_{\mu} 
	= q^{t \cdot i} q^{- \mu_i +1} m_{\mu - \a_i} L_{N_i^{\mu}}^t 
		(-Q_k^{-1})  (L_{N_i^{\mu}} - Q_k) \Big( \sum_{k=1}^{\mu_i} q^{k-1} \wt{\TT}_{N_i^{\mu}-1, \, N_i^{\mu} -k +1} \Big), 
\\
&\psi_{i, - b_i +s}^+ 
	\\&= \begin{cases}
		- q^{-i} Q_k^{-1} q^{\mu_i - \mu_{i+1}} m_{\mu} 
			& \text{ if } s=0, 
		\\ 
		q^{\mu_i-\mu_{i+1}} m_{\mu} 
		\\ \quad  \times 
		\Big( 1 - Q_k^{-1} (q-q^{-1}) 
			\Phi_1 (L_{N_{i-1}^{\mu}+1}, \dots, L_{N_{i-1}^{\mu} +\mu_i} \mid L_{N_i^{\mu} +1}, \dots, L_{N_i^{\mu} + \mu_{i+1}}) 
		\Big) \hspace{-5em} 
			\\ & \text{ if } s =1, 
		\\
		q^{ (s-1)\cdot i} q^{\mu_i - \mu_{i+1}} (q-q^{-1}) m_{\mu} 
		\\ \quad \times 
		\Big( \Phi_{s-1} (L_{N_{i-1}^{\mu}+1}, \dots, L_{N_{i-1}^{\mu} +\mu_i} 
						\mid L_{N_i^{\mu} +1}, \dots, L_{N_i^{\mu} + \mu_{i+1}}) 
			\\ \hspace{3em} 
			- Q_k^{-1} \Phi_s (L_{N_{i-1}^{\mu}+1}, \dots, L_{N_{i-1}^{\mu} +\mu_i} 
							\mid L_{N_i^{\mu} +1}, \dots, L_{N_i^{\mu} + \mu_{i+1}}) \Big)
			 & \text{ if } s >1,
	\end{cases}
\\
& \psi_{i,-s}^- \cdot m_{\mu} 
	\\ &= \begin{cases}
		q^{- (\mu_i - \mu_{i+1})} m_{\mu} 
			& \text{ if } s=0, 
		\\ 
		q^{- i} q^{ - (\mu_i - \mu_{i+1})} m_{\mu} 
		\\ \quad \times 
		\Big( (q-q^{-1}) \Phi_1 (L_{N_i^{\mu}+1}^{-1}, \dots, L_{N_i^{\mu} + \mu_{i+1}}^{-1} 
				\mid L_{N_{i-1}^{\mu} +1}^{-1}, \dots, L_{N_{i-1}^{\mu} + \mu_i}^{-1} )  
			- Q_k^{-1} \Big) \hspace{-3em} 
			\\ & \text{ if } s=1,  
		\\ 
		q^{-s \cdot i} q^{- (\mu_i - \mu_{i+1})} (q-q^{-1}) m_{\mu} 
		\\ \quad \times 
		\Big( \Phi_s (L_{N_i^{\mu}+1}^{-1}, \dots, L_{N_i^{\mu} + \mu_{i+1}}^{-1} 
					\mid L_{N_{i-1}^{\mu} +1}^{-1}, \dots, L_{N_{i-1}^{\mu} + \mu_i}^{-1} )  
			\\ \hspace{3em} 
			- Q_k^{-1} \Phi_{s-1}  (L_{N_i^{\mu}+1}^{-1}, \dots, L_{N_i^{\mu} + \mu_{i+1}}^{-1} 
								\mid L_{N_{i-1}^{\mu} +1}^{-1}, \dots, L_{N_{i-1}^{\mu} + \mu_i}^{-1} ) \Big)
			& \text{ if } s>1.
	\end{cases}
\end{align*}
\end{enumerate}
\end{prop}

%%%%%
\para  
For $\Bla \in \vL_{n,r}^+$, $\bT \in \SStd_m^{[\bm]}$ and $1 \leq i \leq m$, 
let $\bT_{\da i}$  denote the subtableau of $\bT$ consisting of all entries $j \leq i$. 
By the definition of semi-standard tableaux, we see that $\shape (\bT_{\da i}) \in \vL_{\mu_1+\dots+\mu_i,\, r}^+$ 
if $\owt \bT=\mu=(\mu_1,\dots, \mu_m)$. 
For  example, 
if 
\begin{align*}
\bT= \left( \, \begin{array}{|c|c|c|} \hline 1 & 2 & 2 \\ \hline 2 & 4 \\ \cline{1-2} \multicolumn{1}{c}{} \end{array} \, , 
	\, 
	\begin{array}{|c|c|c|} \hline 2 & 2 \\ \hline 3 & 4 \\ \hline 4 \\ \cline{1-1} \end{array} \, , 
	\, 
	\begin{array}{|c|c|c|c|} \hline 3 & 3 & 4 & 4 \\ \hline 4 \\ \cline{1-1} \multicolumn{1}{c}{} \end{array} 
	\, 
	\right), 
\end{align*}
then we have 
\begin{align*}
&\bT_{\da 1}= \left( \, \begin{array}{|c|c|c|} \hline 1 \\ \hline  \end{array} \, , 
	\, 
	\emptyset, 
	\, 
	\emptyset 
	\, 
	\right), 
\quad 
\bT_{\da 2}= \left( \, \begin{array}{|c|c|c|} \hline 1 & 2 & 2 \\ \hline 2  \\ \cline{1-1}  \end{array} \, , 
	\, 
	\begin{array}{|c|c|c|} \hline 2 & 2 \\ \hline \multicolumn{1}{c}{}   \end{array} \, , 
	\, 
	\emptyset 
	\, 
	\right), 
\\
&\bT_{\da 3}= \left( \, \begin{array}{|c|c|c|} \hline 1 & 2 & 2 \\ \hline 2  \\ \cline{1-1}  \end{array} \, , 
	\, 
	\begin{array}{|c|c|c|} \hline 2 & 2 \\ \hline 3  \\ \cline{1-1}  \end{array} \, , 
	\, 
	\begin{array}{|c|c|c|c|} \hline 3 & 3 \\ \cline{1-2} \multicolumn{1}{c}{} \end{array} 
	\, 
	\right), 
\quad 
\bT_{\da i} = \bT \quad (i \geq 4).
\end{align*}

For $\Bla \in \vL^+_{n,r}$ and $\mu \in \vL_n(m)$, 
we define the partial order $\trreq$ on $\SStd^{[\bm]} (\Bla, \mu)$ by 
\begin{align*} 
\bT \trreq \bS  \text{ if and only if } \shape ( \bT_{\da i}) \trreq \shape (\bS_{\da i}) \text{ for all } 1 \leq i \leq m.
\end{align*}

We can prove the following proposition in a similar way as in \cite[proof of Theorem 3.10]{JM00} 
by using Proposition \ref{Prop act uqbm Mmu}. 

%%%%%

\begin{prop}
\label{psi vf bT}
For $\Bla \in \vL_{n,r}^+[\bm]$, $\mu \in \vL_n(m)$ and $\bT \in \SStd^{[\bm]} (\Bla,\mu)$, 
we have the following equations in the cell module $W(\Bla)$. 

\begin{enumerate}
\item 
For $1 \leq i \leq m-1$, we have 
$\psi_{i,0}^- \cdot \vf_{\bT} = q^{- (\mu_i - \mu_{i+1})} \vf_{\bT}$. 

\item 
For $1 \leq i \leq m-1$ and $s \in \ZZ_{\geq 0}$, we have the following. 
\begin{enumerate} 
\item 
If $b_i=0$, we have 
\begin{align*}
&\psi_{i,s}^+ \cdot \vf_{\bT} 
\\
&= \begin{cases}
		q^{\mu_i - \mu_{i+1}} \vf_{\bT} & \text{ if } s=0, 
		\\ \dis 
		q^{s \cdot i} q^{\mu_i - \mu_{i+1}}  (q-q^{-1}) \Phi_{i,s} (\bT) \vf_{\bT} 
			+ \sum_{\bS \in \SStd^{[\bm]}(\Bla, \mu) \atop \bS \trr \bT} r_{\bS} \vf_{\bS}
			\quad (r_{\bS} \in \CC) 
		& \text{ if } s >0. 
	\end{cases}
\end{align*}

\item 
If $b_i=1$ and $\xi^{-1} (i) = (m_k,k)$, we have 
\begin{align*}
& \psi_{i, -b_i +s}^+ \cdot \vf_{\bT} 
\\
&= \begin{cases}
	- q^{-i} Q_k^{-1} q^{\mu_i - \mu_{i+1}} \vf_{\bT} & \text{ if } s=0, 
	\\ \dis 
	q^{\mu_i - \mu_{i+1}} \Big( 1 - Q_k^{-1} (q-q^{-1}) \Phi_{i,1} (\bT) \Big) \vf_{\bT} 
		+ \sum_{\bS \in \SStd^{[\bm]}(\Bla, \mu) \atop \bS \trr \bT} r_{\bS} \vf_{\bS}
			\quad (r_{\bS} \in \CC)  
		& \text{ if } s=1, 
	\\ \dis 
	q^{(s-1) \cdot i} q^{\mu_i - \mu_{i+1}} (q-q^{-1}) \Big( \Phi_{i,s-1} (\bT) - Q_k^{-1} \Phi_{i, s} (\bT) \Big) \vf_{\bT} 
		\\\dis
		\quad 
		+ \sum_{\bS \in \SStd^{[\bm]}(\Bla, \mu) \atop \bS \trr \bT} r_{\bS} \vf_{\bS}
		\quad (r_{\bS} \in \CC)  
		& \text{ if } s >1.
\end{cases}
\end{align*}
\end{enumerate}
For $\Phi_{i,s} (\bT)$ in the above equations, 
we use the content $c_{\bQ}(x) = q^{2 (b-a)} Q_{c-1}$ of $x=(a,b,c) \in [\Bla]$ with respect to $\bQ=(Q_0,Q_1,\dots, Q_{r-1})$. 
\end{enumerate}
\end{prop}

%%%
\para  
For $\Bla \in \vL_{n,r}^+[\bm]$, put $\hat{\xi}(\Bla) =(\la_1,\la_2, \dots, \la_m)$.  
For $1 \leq i,j \leq m-1$ and $t \in \ZZ$, 
we have 
\begin{align*}
\psi_{j,0}^- \cdot (e_{i,t} \cdot \vf_{\bT_{[\bm]}^{\Bla}}) 
= q^{-a_{ji}} e_{i,t} \cdot (\psi_{j,0}^- \cdot \vf_{\bT_{[\bm]}^{\Bla}}) 
= q^{- (\la_i - \la_{i+1} + a_{ji})}  ( e_{i,t} \cdot \vf_{\bT^{\Bla}_{[\bm]}}) 
\end{align*}
by Proposition \ref{psi vf bT} (\roi) and the relation (U4'). 
Thus, we see that $e_{i,t} \cdot \vf_{\bT_{[\bm]}^{\Bla}}$ is a linear combination of 
$\{ \vf_{\bT} \mid \bT \in \SStd^{[\bm]}(\Bla, \hat{\xi} (\Bla) + \a_i)\}$ by  Proposition \ref{psi vf bT} (\roi) again. 
However, we also see that  $\SStd^{[\bm]} (\Bla, \hat{\xi} (\Bla) +\a_i)$ is an empty set 
by the definition of semi-standard tableaux with respect to $\bm$. 
As a consequence, we have $e_{i,t} \cdot \vf_{\bT^{\Bla}_{[\bm]}}=0$ for all $1 \leq i \leq m-1$ and $t \in \ZZ$. 
Thus, $\vf_{\bT^{\Bla}_{[\bm]}}$ is a singular vector. 
We also see that $\SStd^{[\bm]}(\Bla, \hat{\xi} (\Bla)) = \{ \bT^{\Bla}_{[\bm]} \}$. 
Thus, $\vf_{\bT^{\Bla}_{[\bm]}}$ is an $\ell_{[\bb_{\bm}]}$-weight vector by Proposition \ref{psi vf bT} (\roii). 
(We can obtain  similar equations for the action of $\psi_{i,-s}^-$ from Proposition \ref{Prop act uqbm Mmu}.) 
Let $L(\Bla)$ denote the simple quotient of the $U_{q,[\bb_{\bm}]}$-submodule 
$U_{q,[\bb_{\bm}]} \cdot \vf_{\bT^{\Bla}_{[\bm]}} \subset W(\Bla)$ generated by $\vf_{\bT^{\Bla}_{[\bm]}}$. 
Then, we have the following theorem.  

%%%%%
\begin{thm}
\label{Thm qcha cell}
For $\Bla \in \vL_{n,r}^+[\bm]$, we have the following.  

\begin{enumerate}
\item 
The $q$-character of $W(\Bla)$ is given by 
\begin{align*}
&\chi_{[\bb_{\bm}]} (W(\Bla)) = \sum_{\bT \in \SStd_m^{[\bm]} (\Bla)} X^{\bT}, 
\\
&\text{ where }
X^{\bT} = \prod_{i=1}^{m-1} \Big( \prod_{x \in [\bT]_i} Y_{i, \, q^{i-1} c_{\bQ}(x)} 
				\prod_{y \in [\bT]_{i+1}} Y^{-1}_{i, \, q^{i+1} c_{\bQ} (y)} \Big) 
		\\ & \hspace{7em} \times 
		\prod_{k=1}^{r-1} \tau_{m_{\lan k \ran}, \, - q^{- m_{\lan k \ran}} Q_k^{-1}} 
				Z_{m_{\lan k \ran}, \, q^{m_{\lan k \ran}} Q_k}. 
\end{align*}

\item 
The vector $\vf_{\bT^{\Bla}_{[\bm]}} \in L(\Bla)$ is an $\ell_{[\bb_{\bm}]}$-highest weight vector.  
The Drinfeld polynomials $(\bS=(S_i(z))_{1 \leq i \leq m-1}, \, \bP =(P_i(z))_{1 \leq i \leq m-1} )$ of $L(\Bla)$ 
are given by 
\begin{align*}
&S_i(z)= \begin{cases}
		1 & \text{ if } b_i=0, 
		\\
		- q^{- \la_1^{(k+1)} -i} Q_k^{-1}  + q^{ \la_1^{(k+1)} }  z 
		& \text{ if } b_i=1, 
	\end{cases}
\\
& 
P_i(z) = \prod_{p=1}^{\la_j^{(k)} - \la_{j+1}^{(k)}} (1 - q^{ 2 \la_j^{(k)}  - i+1 - 2p} q^{ 2 m_{\lan k-1 \ran}} Q_{k-1} z), 
\end{align*}
where  $\xi^{-1} (i) =(j,k)$. 
(Note that $\la_{m_k +1}^{(k)} =0$ since $\Bla \in \vL_{n,r}^+[\bm]$.) 
\end{enumerate}

\begin{proof}
(\roi). 
By Proposition \ref{Prop whPhiklz}, 
for $1 \leq i \leq m-1$, 
we have 
\begin{align*}
&q^{\mu_i - \mu_{i+1}} \big( z^0 + \sum_{s>0} q^{s \cdot i} (q-q^{-1}) \Phi_{i,s} (\bT) \big) 
\\
&= q^{\mu_i - \mu_{i+1}} \prod_{x \in [\bT]_i} \frac{(1 - q^{i-1} c_{\bQ}(x) q^{-1} z)}{ (1 - q^{i-1} c_{\bQ} (x) q z)} 
	\prod_{y \in [\bT]_{i+1}} \frac{ (1 - q^{i+1} c_{\bQ} (y) q z ) }{ (1 - q^{i+1} c_{\bQ} (y) q^{-1} z)}  
\end{align*}
if $b_i=0$, and 
\begin{align*}
& q^{\mu_i - \mu_{i+1}} \Big( - q^{-i} Q_k^{-1} z^{-1} + \big( 1 - Q_k^{-1} (q-q^{-1}) \Phi_{i,1} (\bT) \big) z^0 
	\\ & \hspace{5em} 
	+ \sum_{s >1} \Big( q^{(s-1) \cdot i} (q-q^{-1}) \big( \Phi_{i,s-1} (\bT) - Q_k^{-1} \Phi_{i,s} (\bT) \big) z^{s-1} \Big) 
\\
&= q^{\mu_i - \mu_{i+1}} \big( z^0 + \sum_{s >0} q^{ s \cdot i} (q-q^{-1}) \Phi_{i,s} (\bT) z^s \big) 
	\\ & \quad 
	- q^{\mu_i -\mu_{i+1}} q^{-i} Q_k^{-1} z^{-1} 
		\big( z^0 + q^i(q-q^{-1}) \Phi_{i,1} (\bT) z + \sum_{s >1} q^{s \cdot i } (q-q^{-1}) \Phi_{i,s} (\bT) z^s \big) 
\\
&= q^{\mu_i - \mu_{i+1}} \prod_{x \in [\bT]_i} \frac{(1 - q^{i-1} c_{\bQ}(x) q^{-1} z)}{ (1 - q^{i-1} c_{\bQ} (x) q z)} 
	\prod_{y \in [\bT]_{i+1}} \frac{ (1 - q^{i+1} c_{\bQ} (y) q z ) }{ (1 - q^{i+1} c_{\bQ} (y) q^{-1} z)}  
	\\ & \qquad \times 
	 (- q^{-i} Q_k^{-1}) \frac{(1 - q^i Q_k z)}{z}
\end{align*}
if $b_i=1$ and $\xi^{-1} (i)=(m_k,k)$. 
We note that $\xi^{-1}(i) = (m_k,k)$ if and only if $i=m_{\lan k \ran}$. 
Then, Proposition \ref{psi vf bT} implies (\roi). 

%%%

(\roii). 
We have already seen that $\vf_{\bT^{\Bla}_{[\bm]}}$ is an $\ell_{[\bb_{\bm}]}$-highest weight vector of $L(\Bla)$. 
By the definition of $\bT_{[\bm]}^{\Bla}$, we see that 
\begin{align*}
&[\bT^{\Bla}_{[\bm]}]_i= \{ (j, p, k) \mid 1 \leq p \leq \la_j^{(k)} \},  
\\
& [\bT^{\Bla}_{[\bm]}]_{i+1} = 
	\begin{cases}
	\{ (j+1, p , k) \mid 1 \leq p \leq \la_{j+1}^{(k)}\} & \text{ if } b_i=0, 
	\\ 
	\{(1, p ,k+1) \mid 1 \leq p \leq \la_1^{(k+1)}\} & \text{ if } b_i=1 
	\end{cases}
\end{align*}
when $\xi^{-1} (i) = (j,k)$. 
Thus, we have 
\begin{align*}
&\prod_{x \in [\bT^{\Bla}_{[\bm]}]_i} \frac{(1 - q^{i-1} c_{\bQ}(x) q^{-1} z)}{ (1 - q^{i-1} c_{\bQ} (x) q z)} 
	\prod_{y \in [\bT^{\Bla}_{[\bm]}]_{i+1}} \frac{ (1 - q^{i+1} c_{\bQ} (y) q z ) }{ (1 - q^{i+1} c_{\bQ} (y) q^{-1} z)}  
\\
&= \begin{cases}
	\dis 
	\prod_{p=1}^{\la_j^{(k)} } \frac{ (1 - q^{i + 2 (p-j)-1} Q_{k-1} q^{-1} z)}{ (1 - q^{i +2 (p-j) -1} Q_{k-1}  q z )} 
		\prod_{p'=1}^{\la_{j+1}^{(k)}} \frac{ (1 - q^{i+ 2 (p'- j-1) +1} Q_{k-1} q z)}{ (1 - q^{ i+ 2 (p'-j-1) +1} Q_{k-1} q^{-1} z)} 
		& \text{ if } b_i=0, 
	 \\ \dis 
	 \prod_{p=1}^{\la_{m_k}^{(k)} } \frac{ (1 - q^{i + 2 (p-j)-1} Q_{k-1} q^{-1} z)}{ (1 - q^{i +2 (p-j) -1} Q_{k-1}  q z )} 
	 \prod_{p'=1}^{\la_1^{(k+1)}} \frac{(1 - q^{i + 2 (p' - 1) +1} Q_k q  z)}{ ( 1 - q^{i+ 2 (p' - 1) +1} Q_k q^{-1} z)} 
	 & \text{ if } b_i=1 
	 \end{cases}
\\
&= \begin{cases}
	\dis 
	\prod_{p=1}^{\la_j^{(k)} - \la_{j+1}^{(k)}} 
		\frac{ ( 1 - q^{i+ 2 (\la_{j+1}^{(k)} +p -j)-1} Q_{k-1} q^{-1} z)}{(1 - q^{i + 2 (\la_{j+1}^{(k)} +p -j) -1} Q_{k-1} q z)} 
		& \text{ if } b_i=0,  
	\\ \dis 
	\prod_{p=1}^{\la_{m_k}^{(k)} } \Big( \frac{ (1 - q^{i + 2 (p-j)-1} Q_{k-1} q^{-1} z)}{ (1 - q^{i +2 (p-j) -1} Q_{k-1}  q z )} \Big)  
		\cdot \frac{(1 - q^{i + 2 \la_1^{(k+1)} } Q_k z)}{ (1 - q^i Q_k z)} 
		& \text{ if } b_i=1, 
	\end{cases}
\end{align*}
where we note that 
$i = m_{\lan k-1 \ran} +j$, 
\begin{align*}
\begin{cases}
i+ 2 (\la_{j+1}^{(k)} +p-j)-1 = 2 \la_j^{(k)} +  2 m_{\lan k-1 \ran} - i+1 - 2 (\la_j^{(k)} - \la_{j+1}^{(k)} - p +1) 
	& \text{ if } b_i=0, 
\\
i + 2 (p-j) -1 = 2 \la_{m_k}^{(k)} + 2 m_{\lan k-1 \ran} - i +1 - 2 (\la_{m_k}^{(k)} - p+1) 
	& \text{ if } b_i=1  
\end{cases} 
\end{align*} 
and $\la_{m_k +1}^{(k)} = 0$ since $\Bla \in \vL_{n,r}^+[\bm]$. 
We also have 
\begin{align*}
\mu_i(\bT_{[\bm]}^{\Bla}) = \la_j^{(k)} \text{ and } 
	\mu_{i+1} (\bT_{[\bm]}^{\Bla}) = \begin{cases} \la_{j+1}^{(k)} & \text{ if } b_i=0, \\ \la_1^{(k+1)} & \text{ if } b_i=1. \end{cases}  
\end{align*}
Then, the calculation in (\roi) for $\bT=\bT^{\Bla}_{[\bm]}$ imply (\roii). 
\end{proof}
\end{thm}

%%%
\remark 
In \cite[Proposition D.12]{KW21}, 
we have already calculated the  Drinfeld polynomials of $L(\Bla)$ 
in terms of $(q,\bQ)$-current algebras. 

%%%%%%%
\para 
We consider the separation condition 
\begin{align}
\label{sep cond bQ}
\prod_{0 \leq i < j \leq r-1} \prod_{ -n < k <n} (q^{ 2k } Q_i - Q_j) \not=0 
\end{align}
for parameters $\bQ=(Q_0,Q_1,\dots, Q_{r-1})$. 
Under this condition, we can describe the action of $U_{q,[\bb_{\bm}]}$ on the cell module $W(\Bla)$ ($\Bla \in \vL_{n,r}^+[\bm]$) 
explicitly like as Proposition \ref{Prop Def DBz Bla}. 
For this description, we prepare some combinatorial notations.  

For $\Bla \in \vL_{n,r}^+[\bm]$, $\bT \in \SStd_m^{[\bm]}(\Bla)$ and $1 \leq i \leq m-1$, 
put 
\begin{align*}
[\bT]_{\a_i}^{[\bm]} = \{ x \in [\bT]_{i+1} \mid \bT_x^- \in \SStd_m^{[\bm]}(\Bla)\}, 
\quad 
[\bT]_{-\a_i}^{[\bm]} = \{ x \in [\bT]_i \mid \bT_x^+ \in \SStd_m^{[\bm]} (\Bla)\}. 
\end{align*}

For  $x \in [\bT]_{\a_i}^{[\bm]}$ and $y \in [\bT]_{-\a_i}^{[\bm]}$ ($1 \leq i \leq m-1$), 
we also set 
\begin{align*}
&B_{\bT;[\bm]}^x = B_{\bT}^x \cdot \begin{cases}
	1 & \text{ if } b_i=0, 
	\\
	( 1 - Q_k^{-1} c_{\bQ} (x) ) & \text{ if } b_i=1 \text{ and } \xi^{-1} (i) =(m_k,k), 
	\end{cases}
\\
&C_{\bT ;[\bm]}^{y} = C_{\bT}^y. 
\end{align*}

\begin{prop}
Assume the separation condition \eqref{sep cond bQ}. 
For $\Bla \in \vL_{n,r}^+[\bm]$, 
let $\D_{\bQ}^{[\bm]}(\Bla)$ be a $\CC$-vector space with a basis $\{ v_{\bT} \mid \bT \in \SStd_m^{[\bm]} (\Bla)\}$. 
Then, we can define the action of $U_{q,[\bb_{\bm}]}$ on $\D_{\bQ}^{[\bm]} (\Bla)$ by 
\begin{align*}
& e_{i,t} \cdot v_{\bT} = \sum_{ x \in [\bT]_{\a_i}^{[\bm]}} B_{\bT ; [\bm]}^x (q^i c_{\bQ}(x))^t v_{\bT_x^-} 
	\quad (1 \leq i \leq m-1, \, t \in \ZZ), 
\\
& f_{i,t} \cdot v_{\bT} = \sum_{x \in [\bT]_{-\a_i}^{[\bm]}} C_{\bT;[\bm]}^x (q^i c_{\bQ} (x))^t v_{\bT_x^+} 
	\quad ( 1 \leq i \leq m-1, \,  t \in \ZZ), 
\\
& \psi_{i,s}^+ \cdot v_{\bT} = q^{\mu_i(\bT) - \mu_{i+1} (\bT) + s \cdot i} 
	\cdot 
	\begin{cases}
	v_{\bT} 
		& \text{ if }  s=0, 
	\\
	(q-q^{-1}) \Phi_{i,s} (\bT) v_{\bT} 
		& \text{ if }  s >0 
	\end{cases}
	\\ & \text{ for } 1 \leq i \leq m-1 \text{ such that } b_i =0, 
\\ 
&\psi_{i, - b_i +s}^+ \cdot v_{\bT} = q^{\mu_i(\bT) - \mu_{i+1} (\bT) + (-b_i +s) \cdot i} 
	\begin{cases}
		(- Q_k^{-1}) v_{\bT} 
			& \text{ if }  s=0, 
		\\ 
		\big( 1 - Q_k^{-1} (q-q^{-1}) \Phi_{i,1} (\bT) \big) v_{\bT} 
			& \text{ if } s=1, 
		\\
		(q-q^{-1} ) \big( \Phi_{i,s-1} (\bT) - Q_k^{-1} \Phi_{i,s} (\bT) \big) v_{\bT} 
			& \text{ if } s >1 
	\end{cases}
	\\ & \text{ for } i=m_{\lan k \ran} \, 
	 ( \text{in this case, we have } b_i=1 \text{ and } \xi^{-1} (i) =(m_k,k)), 
\\
& \psi_{i, -s}^- \cdot v_{\bT} = q^{- \mu_i(\bT) + \mu_{i+1} (\bT) - s \cdot i} 
	\cdot 
	\begin{cases}
		v_{\bT} 
			& \text{ if } s=0, 
		\\
		(q-q^{-1}) \Phi_{i, -s} (\bT) v_{\bT} 
			& \text{ if } s >0
	\end{cases}
	\\ & 
	\text{ for } 1 \leq i \leq m-1 \text{ such that } b_i =0, 
\\
& \psi_{i, -s}^- \cdot v_{\bT} = q^{-\mu_i(\bT) + \mu_{i+1} (\bT) - s \cdot i} 
	\cdot 
	\begin{cases}
		v_{\bT} 
			& \text{ if } s=0,  
		\\ 
		\big( (q-q^{-1}) \Phi_{i, -1} (\bT) - Q_k^{-1} \big) v_{\bT} 
			& \text{ if } s=1, 
		\\ 
		(q-q^{-1}) \big( \Phi_{i,-s} (\bT) - Q_k^{-1} \Phi_{i, - s+1} (\bT) 
			& \text{ if } s>1
	\end{cases}
	\\ & \text{ for } i = m_{\lan k \ran}. 
\end{align*}
\begin{proof}
We can check the well-definedness of this action in a similar way as in the proof of Proposition \ref{Prop Def DBz Bla}.
Here, we give a mention for the relation (U6). 

For $\bT \in \SStd_m^{[\bm]}(\Bla)$ and $1 \leq i \leq m-1$, 
we consider a rational function  
\begin{align*}
\fF_{i,\bT}^{[\bm]} (z) 
= \fF_{i,\bT}(z) \cdot 
	\begin{cases}
	1 & \text{ if } b_i=0, 
	\\ \dis 
	(- q^{-i} Q_k^{-1}) \frac{(1 - q^i Q_k z)}{z}
		& \text{ if } b_i=1 \text{ and } \xi^{-1} (i)=(m_k,k),
	\end{cases}
\end{align*}
where we set $\Bz = \bQ$. 
We note that $\xi^{-1} (i) = (m_k,k)$ if and only if $i=m_{\lan k \ran}$. 
By the calculation in the proof of Theorem \ref{Thm qcha cell} (\roi), 
we have 
\begin{align*}
\psi_i^+(z) \cdot v_{\bT} 
&= q^{\mu_i(\bT) - \mu_{i+1} (\bT)} \fF_{i,\bT}^{[\bm]} (z) v_{\bT}
\\
&= q^{\mu_i(\bT) - \mu_{i+1} (\bT)} \big( \sum_{t \in \ZZ} \Res_{z=0} (z^{-t-1} \fF_{i, \bT}^{[\bm]}(z)) z^t \big) v_{\bT}.
\end{align*}
We also have 
\begin{align*}
\psi_i^-(z) \cdot v_{\bT} = q^{\mu_i(\bT) - \mu_{i+1} (\bT)} 
	\big( - \sum_{t \in \ZZ} \Res_{z=\infty} (z^{-t-1} \fF_{i,\bT}^{[\bm]}(z)) z^t\big) v_{\bT}. 
\end{align*}
It is clear that 
$\{ \text{pole of } \fF_{i, \bT}^{[\bm]} (z)\} \subset \{(c_{\bQ} (x) q^i)^{-1} \mid x \in [\bT]_i \cup [\bT]_{i+1}\} \cup \{0 , \infty\}$. 

By Lemma \ref{Lemma pole z-t-1 fFibTz} and the definition of $\fF_{i,\bT}^{[\bm]}(z)$, 
both the rational functions $\fF_{i, \bT}(z)$ and  $\fF_{i,\bT}^{[\bm]} (z)$ are regular at $z= (c_{\bQ}(x)q^i)^{-1}$ for 
$x \in ([\bT]_i \setminus [\bT]_{-\a_i}) \cup ([\bT]_{i+1} \setminus [\bT]_{\a_i})$. 

By definitions of $\SStd_m^{[\bm]}(\Bla)$ and of $[\bT]_{\pm \a_i}^{[\bm]}$, 
we can check that 
$[\bT]_{-\a_i}^{[\bm]} = [\bT]_{-\a_i}$, 
$[\bT]_{\a_i}^{[\bm]} \subseteq [\bT]_{\a_i}$ and 
\begin{align*}
[\bT]_{\a_i} \setminus [\bT]_{\a_i}^{[\bm]} = \begin{cases} 
\{(1,1,k) \} & \text{ if } i= m_{\lan k-1 \ran} \text{ and } 
 	\bT((1,1,k)) = m_{\lan k-1 \ran} +1, 
\\ \emptyset & \text{ otherwise}.  
\end{cases}
\end{align*}
Moreover, 
for $ (1,1,k) \in [\bT]_{\a_i} \setminus [\bT]_{\a_i}^{[\bm]}$ ($i= m_{\lan k-1 \ran}$), 
we see that 
\begin{align*}
&\fF_{i, \bT}^{[\bm]} (z) 
\\
&= \fF_{i, \bT}(z) \cdot (-q^{-i} Q_{k-1}^{-1}) \frac{ (1- q^i Q_{k-1} z)}{z} 
\\
&= \prod_{ x \in [\bT]_i} \frac{ (1 - q^{-2} c_{\bQ} (x) q^i z)}{ (1 - c_{\bQ} (x) q^i z)} 
				\prod_{y \in [\bT]_{i+1}} \frac{(1 - q^2 c_{\bQ} (y) q^i z)}{( 1 - c_{\bQ} (y) q^i z )}
	\cdot (-q^{-i} Q_{k-1}^{-1}) \frac{ (1- c_{\bQ} ((1,1,k)) q^i  z)}{z}. 
\end{align*}
Thus, we conclude that $\fF_{i,\bT}^{[\bm]}(z)$ is regular at $z= (c_{\bQ} (y) q^i)^{-1}$ 
although $\fF_{i,\bT} (z)$ is not regular at $z= (c_{\bQ} (y) q^i)^{-1}$.
As a consequence, 
we have 
\begin{align*}
\{ \text{pole of } z^{-t-1} \fF_{i, \bT}^{[\bm]}(z) \} 
	\subset \{ (c_{\bQ}(x) q^i )^{-1} \mid x \in [\bT]_{- \a_i}^{[\bm]} \cup [\bT]_{\a_i}^{[\bm]} \} 
		\cup \{0, \infty\}.
\end{align*}
Then, we can check the relation (U6) in a similar way as in the proof of Proposition \ref{Prop Def DBz Bla}. 
\end{proof}
\end{prop}

%%%%%
\begin{thm}
\label{Thm DbQ bm iso W Bla}
Assume the separation condition \eqref{sep cond bQ}.  
For $\Bla \in \vL_{n,r}^+[\bm]$, 
the $U_{q,[\bb_{\bm}]}$-module $\D_{\bQ}^{[\bm]} (\Bla)$ is  a simple $\ell_{[\bb_{\bm}]}$-highest weight module 
with a highest weight vector $v_{\bT^{\Bla}_{[\bm]}}$. 
Moreover, we have  $\D_{\bQ}^{[\bm]}(\Bla) \cong W(\Bla)$ as $U_{q,[\bb_{\bm}]}$-modules. 
Then, the Drinfeld polynomials and  the $q$-character of $\D_{\bQ}^{[\bm]}(\Bla) \cong W(\Bla)$ 
are ones in Theorem \ref{Thm qcha cell}. 

\begin{proof}
We can prove that $\D_{\bQ}^{[\bm]}(\Bla)$ is simple in a similar way as in the proof of Theorem \ref{Thm DBz Bla}. 
We can also calculate the Drinfeld polynomials of $\D_{\bQ}^{[\bm]}(\Bla)$ in the same way 
as in the proof of Theorem \ref{Thm qcha cell}. 
By comparing Drinfeld polynomials, 
we have $\D_{\bQ}^{[\bm]}(\Bla) \cong L(\Bla)$. 
On the other hand, 
we have $\dim \D_{\bQ}^{[\bm]}(\Bla) = \sharp \SStd_m^{[\bm]}(\Bla) = \dim W(\Bla)$, 
and we conclude that $\D_{\bQ}^{[\bm]} (\Bla) \cong W(\Bla)$. 
\end{proof}
\end{thm}

%%%
\para 
Under the separation condition \eqref{sep cond bQ}, 
we can  construct $U_{q,[\bb_{\bm}]}$-module $W(\Bla)$ ($\Bla \in \vL_{n,r}^+[\bm]$) 
as the simple submodule (resp. the simple quotient) of a tensor product of some evaluation modules 
and a one-dimensional $U_{q,[\bb_{\bm}]}$-module 
as the following proposition. 

\begin{prop}
\label{Prop WBla D Bla}
Assume that $\bQ=(Q_0,Q_1,\dots, Q_{r-1}) \in (\CC^{\times})^r$ satisfies the separation condition \eqref{sep cond bQ}. 
Define $\Bb_{[\bm]} =(\b_{i, - b_i +s})^{0 \leq s \leq b_i}_{1 \leq i \leq m-1} \in \BB_{ [ \bb_{\bm}]}$ by 
\begin{align*}
\b_{i,0} =1 \text{ if } b_i=0, 
\quad 
\begin{cases}
\b_{i,-1} = - q^{-i} Q_k^{-1}, 
\\
\b_{i,0} = 1 
\end{cases}
\text{ if } b_i=1 \text{ and } \xi^{-1} (i) =(m_k,k). 
\end{align*}
Then, for $\Bla \in \vL_{n,r}^+[\bm] \subset \vL_{n,r}^+(m)$, we have the following. 

\begin{enumerate}
\item 
\begin{enumerate}
\item 
$\bT^{\Bla}_{[\bm]} \in \Sing_m^{\Bb_{[\bm]}} (\Bla)$. 

\item 
$\SStd_m (\Bla ; \geq \bT_{[\bm]}^{\Bla}) = \SStd_m^{[\bm]}(\Bla)$. 

\item 
We have $W(\Bla) \cong \D_{\bQ}(\Bla ; \geq \bT_{[\bm]}^{\Bla}) \otimes L_{\Bb_{[\bm]}}$ as $U_{q,[\bb_{\bm}]}$-modules. 
Moreover, we have 
$\Sing_m^{\Bb_{[\bm]}}(\Bla) \cap \SStd_m (\Bla ; \geq \bT_{[\bm]}^{\Bla}) =\{\bT_{[\bm]}^{\Bla}\}$. 
In particular, we have  
$\D_{\bQ}^{\Bb_{[\bm]} *} (\Bla) = \D_{\bQ}(\Bla ; \geq \bT_{[\bm]}^{\Bla}) \otimes L_{\Bb_{[\bm]}}$. 
\end{enumerate}

\item 
\begin{enumerate}
\item 
For $2 \leq k \leq r$, 
let $\bT^\sharp_{\Bla;k} \in \SStd_m(\Bla)$ be the semi-standard tableau given by 
\begin{align*}
\bT^{\sharp}_{\Bla ; k} ((a,b,c)) 
	= \begin{cases} 
		m_{\lan k-1 \ran} & \text{ if } (a,b,c)= (1,1,k), 
		\\
		m - (\la^{(c)})'_b +a & \text{ otherwise}.
	\end{cases}
\end{align*}
Then, we have $\bT^{\sharp}_{\Bla;k} \in \fSing_m^{\Bb_{[\bm]}}(\Bla)$ for $2 \leq k \leq r$. 

\item 
$\SStd_m(\Bla) \setminus \big( \bigcup_{k=2}^r \SStd_m(\Bla; \leq \bT^{\sharp}_{\Bla;k} )\big) = \SStd_m^{[\bm]}(\Bla)$. 

\item 
Put $\D_{\bQ}^{\Bb_{[\bm]} \sharp} (\Bla) 
	= \big( L_{\Bb_{[\bm]}} \otimes \D_{\bQ}(\Bla) \big) 
			/ \big( \sum_{k=2}^r L_{\Bb_{[\bm]}} \otimes \D_{\bQ} (\Bla ; \leq \bT^{\sharp}_{\Bla;k}) \big)$. 
Then we have $W(\Bla) \cong \D_{\bQ}^{\Bb_{[\bm]} \sharp}(\Bla)$ as $U_{q,[\bb_{\bm}]}$-modules.  
\end{enumerate}
\end{enumerate}
\begin{proof}
We consider the simple $U_q (L\Fsl_m)$-module $\D_{\bQ} (\Bla)$ in Proposition \ref{Prop Def DBz Bla} 
and the one-dimensional $U_{q,[\bb_{\bm}]}$-module $L_{\Bb_{[\bm]}} = \CC w$ associated with $\Bb_{[\bm]}$.

(\roi)-(a). 
By the definition of $\bT_{[\bm]}^{\Bla}$, we have 
\begin{align}
\label{bT bm Bla ai} 
[\bT_{[\bm]}^{\Bla}]_{\a_i} = 
	\begin{cases}
	\{ (1,1,k) \} & \text{ if } i=m_{\lan k-1 \ran} \text{ for some } 2 \leq k \leq r, 
	\\
	\emptyset & \text{ otherwise}. 
	\end{cases}
\end{align}
We also have 
\begin{align}
\label{qicQ 11k}
\begin{split}
&q^i c_{\bQ} ((1,1,k)) = q^i Q_{k-1}, 
\\
& S_i(z) = \b_{i,-1} + \b_{i,0} z = - q^{-i} Q_{k-1}^{-1} + z = - q^{-i} Q_{k-1}^{-1} (1 - q^i Q_{k-1} z)
\end{split}
\end{align}
if $i= m_{\lan k-1 \ran}$. 
Thus, we see that $\prod_{x \in [\bT_{[\bm]}^{\Bla}]_{\a_i}} (1 - q^i c_{\bQ}(x) z)$ divides $S_i(z)$ for each $1 \leq i \leq m-1$, 
and we have $\bT_{[\bm]}^{\Bla} \in \Sing_m^{\Bb_{[\bm]}} (\Bla)$. 

(\roi)-(b). It is clear from definitions. 

(\roi)-(c). 
By (\roi)-(a) and Theorem \ref{Thm D Bla otimes LBb}, 
the vector $v_{\bT_{[\bm]}^{\Bla}} \otimes w$ is singular and an $\ell_{[\bb_{\bm}]}$-weight vector. 
Let  $M$ be the $U_{q,[\bb_{\bm}]} $-submodule of 
$\D_{\bQ}(\Bla ; \geq \bT_{[\bm]}^{\Bla}) \otimes L_{\Bb_{[\bm]}}$ 
generated by $v_{\bT_{[\bm]}^{\Bla}} \otimes w$, 
and let $L$ be the simple quotient of $M$.

By the definition of $\bT_{[\bm]}^{\Bla}$ together with \eqref{bT bm Bla ai}, for $1 \leq i \leq m-1$, we see that 
\begin{align*}
&[\bT_{[\bm]}^{\Bla}]_{i+1}^{\dag}  
	= \begin{cases} 
	\{  (1,b,c) \mid 1 \leq b \leq \la^{(c)}_1 \} & \text{ if }  i=m_{\lan c-1 \ran} \text{ for some } 2 \leq c \leq r, 
	\\
	\emptyset & \text{ otherwise}, 
	\end{cases}
\\
& [\bT_{[\bm]}^{\Bla}]_{i+1}^{\dag R} 
	= \begin{cases} 
	\{  (1,\la_1^{(c)}, c)  \} & \text{ if }  i=m_{\lan c-1 \ran} \text{ for some } 2 \leq c \leq r, 
	\\
	\emptyset & \text{ otherwise}, 
	\end{cases}
\\
& [\bT_{[\bm]}^{\Bla}]_i^{\ddag} 
	= \{ (a,b,c) \mid \la_{a+1}^{(c)} < b \leq \la_a^{(c)} \} 
		\text{ if } i = m_{\lan c-1 \ran} +a.
\end{align*}
Then, by the same calculation as in the proof of Theorem \ref{Thm D Bla otimes LBb} (\rovi), 
we see that the Drinfeld polynomials $(\wt{\bS} =(\wt{S}_i(z))_{1 \leq i \leq m-1}$, $\wt{P}=(\wt{P}_i(z))_{1 \leq i \leq m-1}) $ 
of $L$ are given by 
\begin{align*}
\wt{S}_i(z) 
	&= q^{\mu_i(\bT_{[\bm]}^{\Bla}) - \mu_{i+1} (\bT_{[\bm]}^{\Bla}) - \sharp [\bT_{[\bm]}^{\Bla}]_i^{\ddag}}
		\cdot S_i(z) \cdot \frac{ \prod_{y \in [\bT_{[\bm]}^{\Bla}]_{i+1}^{\dag R}} (1 - q^{i+2} c_{\bQ}(y) z)}
						{ \prod_{ y \in [\bT_{[\bm]}^{\Bla}]_{\a_i}} (1 - q^i c_{\bQ} (y) z)}
	\\
	&= \begin{cases}
		\dis q^{\la^{(c-1)}_{m_{c-1}} - \la_1^{(c)} - \la^{(c-1)}_{m_{c-1}}} 
			\cdot \big(  - q^{-i} Q_{c-1}^{-1} (1 - q^i Q_{c-1} z) \big) 
			\cdot \frac{ (1 - q^{i+2} q^{2 (\la_1^{(c)} -1)} Q_{c-1} z)}{ (1 - q^i Q_{c-1} z)}  
			\hspace{-15em} 
			\\
			& \text{ if } i = m_{\lan c-1 \ran} \text{ for some } 2 \leq c \leq r, 
		\\ 
		1 & \text{ otherwise}, 
	\end{cases}
	\\
	&= \begin{cases}
		- q^{- \la_1^{(c)} -i} Q_{c-1}^{-1} + q^{\la_1^{(c)} }  z 
			& \text{ if } i = m_{\lan c-1 \ran} \text{ for some } 2 \leq c \leq r, 
		\\ 
		1 & \text{ otherwise}, 
	\end{cases} 
\\
\wt{P}_i(z) 
	&= \prod_{x \in [\bT_{[\bm]}^{\Bla}]_i^{\ddag}} (1 - q^{i-1} c_{\bQ}(x) z) 
	= \prod_{p=1}^{\la_a^{(c)} - \la_{a+1}^{(c)}} (1 - q^{i-1} q^{2 (\la_{a+1}^{(c)} + p -a)} Q_{c-1} z) 
	\\
	&= \prod_{p=1}^{\la_a^{(c)} - \la_{a+1}^{(c)}} (1 -  q^{2 \la_a^{(c)}  - i + 1 - 2 p} q^{2 m_{\lan c-1 \ran}}Q_{c-1} z) 
	\text{ if } i=m_{\lan c-1 \ran} +a,
\end{align*}
where we note that 
\begin{align*}
& i= m_{\lan c-1 \ran} \text{ if and only if } \xi^{-1} (i) = (m_{c-1}, c-1 ), 
\\
&i + 2 (\la_{a+1}^{(c)} + 2 p - 2a) -1 
	= 2 \la_a^{(c)} + 2 m_{\lan c-1 \ran} - i + 1 - 2 (\la_a^{(c)} - \la_{a+1}^{(c)} - p +1)
	\\ & \text{ if } i= m_{\lan c-1 \ran} +a.
\end{align*}
This implies that $L \cong W(\Bla)$  
thanks to Theorem \ref{Thm qcha cell} and Theorem \ref{Thm DbQ bm iso W Bla}. 
Then, by comparing the dimensions using (\roi)-(b), 
we have $W(\Bla) \cong L = \D_{\bQ}(\Bla ; \geq \bT_{[\bm]}^{\Bla}) \otimes L_{\Bb_{[\bm]}}$. 
As a consequence, we obtain (\roi)-(c). 

%%%%%
(\roii)-(a). 
By the definition of $\bT^{\sharp}_{\Bla;k}$, we see that 
\begin{align*}
[ \bT^{\sharp}_{\Bla;k} ]_{-\a_i} 
	= \begin{cases}
		\{(1,1,k)\} & \text{ if } i=m_{\lan k-1 \ran}, 
		\\
		\emptyset & \text{ otherwise} 
	\end{cases}
\end{align*}
for $1 \leq i \leq m-1$. 
Then, we see that $\prod_{x \in [\bT^{\sharp}_{\Bla;k}]_{-\a_i}} (1 - q^i c_{\bQ}(x) z)$ divides $S_i(z)$ by \eqref{qicQ 11k}, 
and we have $\bT^{\sharp}_{\Bla;k} \in \fSing_m^{\Bb_{[\bm]}}(\Bla)$ for $2 \leq k \leq r$. 

(\roii)-(b). 
By definition, we see that 
\begin{align*}
\SStd_m (\Bla ; \leq \bT^{\sharp}_{\Bla;k}) = \{ \bT \in \SStd_m(\Bla) \mid \bT ((1,1,k)) \leq m_{\lan k-1 \ran} \} 
\end{align*}
for $2 \leq k \leq r$. These equations imply (\roii)-(b).  

(\roii)-(c). 
We see that $\ol{w \otimes v_{\bT^{\Bla}_{[\bm]}}} \in \D_{\bQ}^{\Bb_{[\bm]}\sharp}(\Bla)$ is a singular vector by (\roii)-(b). 
We also see that the $\ell_{[\bb_{\bm}]}$-weight of 
$\ol{w \otimes v_{\bT^{\Bla}_{[\bm]}}} \in \D_{\bQ}^{\Bb_{[\bm]}\sharp}(\Bla)$ 
coincides with one of $v_{\bT^{\Bla}_{[\bm]}} \otimes w$ by Proposition \ref{Prop Ddc}. 
Then, we can prove (\roii)-(c) in the same way as in the proof of (\roi)-(c). 
\end{proof}
\end{prop}

%%%%%%%%%%%%
\para 
We have another construction of $U_{q,[\bb_{\bm}]}$-module $W(\Bla)$ ($\Bla \in \vL_{n,r}^+[\bm]$) as follows. 

For $\bQ=(Q_0,Q_1,\dots, Q_{r-1}) \in (\CC^{\times})^r$, we set  $\wh{\bQ}_{[\bm]} =(\wh{Q}_0, \wh{Q}_1,\dots, \wh{Q}_{r-1}) $ by 
$\wh{Q}_{k-1} = q^{2 m_{\lan k-1 \ran}} Q_{k-1}$ for $1 \leq k \leq r$. 

For $\Bla =(\la^{(1)}, \dots, \la^{(r)}) \in \vL_{n,r}^+[\bm]$, we define the $r$-partition 
$\wh{\Bla}_{[\bm]} =(\wh{\la}^{(1)}, \dots, \wh{\la}^{(r)}) $ by 
\begin{align*}
\wh{\la}^{(k)} = ( \underbrace{\la_1^{(k)}, \dots, \la_1^{(k)}}_{m_{\lan k-1 \ran}}, \la_1^{(k)}, \la_2^{(k)}, \dots, \la_{m_k}^{(k)}) 
\quad (1 \leq k \leq r). 
\end{align*}
For example, if $\bm= (2,3,3)$ and $\Bla=((3,2), (4,3,1), (3,3,2))$, we have 
\begin{align*}
\wh{\Bla}_{[\bm]}= ((3,2), (4,4,4,3,1), (3,3,3,3,3,3,3,2)).
\end{align*}  
Clearly, we have $ \wh{\Bla}_{[\bm]} \in \vL_{n,r}^+(m)$. 
There exists an injective map 
\begin{align*}
\Om_{[\bm]} : [\Bla] \ra [\wh{\Bla}_{[\bm]}] 
\text{ such that }
\Om_{[\bm]} ((a,b,c)) = (m_{\lan c-1 \ran} + a, b,c). 
\end{align*} 
 
We also set $\Bb^{\Bla}_{[\bm]}=(\b_{i, - b_i +s})^{0 \leq s \leq b_i}_{1 \leq i \leq m-1} \in \BB_{ [ \bb_{\bm}]}$ by 
\begin{align}
\label{Bb bm Bla}
\b_{i,0}=1\text{ if } b_i =0, \quad 
\begin{cases}
\b_{i,-1} = - q^{-\la_1^{(k+1)} -i} Q_k^{-1}, 
\\
\b_{i,0} = q^{\la_1^{(k+1)}} 
\end{cases}
\text{ if } b_i=1 \text{ and } \xi^{-1} (i)=(m_k,k).
\end{align}

%%%
\begin{prop}
\label{Prop WBla D wh Bla}
Assume that  both $\bQ$ and $\wh{\bQ}_{[\bm]}$ satisfy the separation condition \eqref{sep cond bQ}. 
For $\Bla \in \vL_{n,r}^+[\bm]$, we have the following. 
\begin{enumerate} 
\item 
We have $W(\Bla)  \cong \D_{\wh{\bQ}_{[\bm]}}^{\Bb^{\Bla}_{[\bm]}} (\wh{\Bla}_{[\bm]}) $
as  $U_{q,[\bb_{\bm}]}$-modules, 
where $\D_{\wh{\bQ}_{[\bm]}}^{\Bb^{\Bla}_{[\bm]}} (\wh{\Bla}_{[\bm]})$ is the simple submodule of 
$L_{\Bb_{[\bm]}^{\Bla}} \otimes \D_{\wh{\bQ}_{[\bm]}}(\wh{\Bla}_{[\bm]})$ given in Theorem \ref{Thm simple LBb Dla}. 

\item 
Let $\wh{\bT}^{\sharp}  \in \SStd_m (\wh{\Bla}_{[\bm]})$ be the semi-standard tableau given by 
\begin{align*}
\wh{\bT}^{\sharp} ((a,b,c)) = \begin{cases}
	a & \text{ if } a \leq m_{\lan c-1 \ran}, 
	\\
	m - (\wh{\la}^{(c)})'_b +a & \text{ if } a > m_{\lan c-1 \ran}. 
	\end{cases}
\end{align*}
Then, there exists  a bijection 
\begin{align*}
\wt{\Om}_{[\bm]} : \SStd_m (\wh{\Bla}_{[\bm]} ; \leq \wh{\bT}^{\sharp}) \ra \SStd_m^{[\bm]}(\Bla)
\end{align*} 
such that 
$\wt{\Om}_{[\bm]} ( \bT) : [\Bla] \ra \ZZ_{ > 0}$ is  the restriction of $\bT : [\wh{\Bla}_{[\bm]}] \ra \ZZ_{>0}$ 
through the injective map $\Om_{[\bm]} : [\Bla] \ra [\wh{\Bla}_{[\bm]}]$. 
In particular, we have 
$\wt{\Om}_{[\bm]} (\bT^{\wh{\Bla}_{[\bm]}}) = \bT^{\Bla}_{[\bm]}$.  

\item 
We have $\D_{\wh{\bQ}_{[\bm]}}^{\Bb^{\Bla}_{[\bm]}} (\wh{\Bla}_{[\bm]}) 
=L_{\Bb_{[\bm]}^{\Bla}} \otimes \D_{\wh{\bQ}_{[\bm]}}(\wh{\Bla}_{[\bm]} ; \leq \wh{\bT}^{\sharp})$. 
In particular, we have 
$\fSing_m^{\Bb_{[\bm]}^{\Bla}} (\wh{\Bla}_{[\bm]}) \cap \SStd_m (\wh{\Bla}_{[\bm]} ; \leq \wh{\bT}^{\sharp}) 
	= \{ \wh{\bT}^{\sharp } \}$. 
\end{enumerate} 
\begin{proof} 
(\roi). 
By definition,  we see that 
\begin{align}
\label{wh la c i} 
\wh{\la}_i^{(k)} = \begin{cases}   
	\wh{\la}_{i+1}^{(k)} & \text{ if } i \leq m_{\lan k-1 \ran}, 
	\\
	\la_j^{(k)} & \text{ if } i =m_{\lan k-1 \ran} +j \leq m_{\lan k \ran}, 
	\\ 
	0 & \text{ if } i > m_{\lan k \ran} 
	\end{cases}
\end{align}
for $1 \leq k \leq r$. 

For $\Bla \in \vL_{n,r}^+[\bm]$, we consider the simple $U_q(L\Fsl_m)$-module $\D_{\wh{\bQ}_{[\bm]}} (\wh{\Bla}_{[\bm]})$ 
in Proposition \ref{Prop Def DBz Bla}. 
Then, by Theorem \ref{Thm DBz Bla} (\rov) together with \eqref{wh la c i}, 
the Drinfeld polynomials $\bP=( P_i(z))_{1 \leq i \leq m-1}$ of $\D_{\wh{\bQ}_{[\bm]}}(\wh{\Bla}_{[\bm]})$ are given by 
\begin{align}
\label{Piz DbQ whBlabm} 
\begin{split}
P_i(z) 
= \prod_{k=1}^r \prod_{p=1}^{\wh{\la}^{(k)}_i - \wh{\la}^{(k)}_{i+1}} (1 - q^{2 \wh{\la}_i^{(k)} - i +1 - 2 p} \wh{Q}_{k-1} z) 
= \prod_{p=1}^{\la_j^{(k)} - \la_{j+1}^{(k)}} (1 - q^{2 \la_j^{(k)}   - i+1 - 2p} q^{2 m_{\lan k-1 \ran}} Q_{k-1} z) 
\end{split}
\end{align}
where $\xi^{-1}(i) =(j,k)$.  
Then, Theorem \ref{Thm simple LBb Dla}, Theorem \ref{Thm qcha cell}, Themre \ref{Thm DbQ bm iso W Bla}, 
the equations \eqref{Bb bm Bla} and \eqref{Piz DbQ whBlabm}
 imply (\roi). 

(\roii). 
By definitions, we see that 
\begin{align*}
\SStd_m (\wh{\Bla}_{[\bm]}; \leq \wh{\bT}^{\sharp}) 
= \left\{ \bT \in \SStd_m(\wh{\Bla}_{[\bm]}) \mid 
\begin{array}{l}
\bT((a,b,c)) =  a  \text{ if } a \leq m_{\lan c-1 \ran}, 
\\
\bT ((a,b,c)) > m_{\lan c-1 \ran} \text{ if } a > m_{\lan c-1 \ran} 
\end{array}
\right\}, 
\end{align*} 
and this implies the bijection $\wt{\Om}_{[\bm]}$. 

(\roiii). 
For $1 \leq i \leq m-1$, we have 
\begin{align*}
[\wh{\bT}^{\sharp}]_{-\a_i} = 
\begin{cases}
\{ (m_{\lan k-1 \ran}, \la_1^{(k)}, k) \} & \text{ if } i = m_{\lan k -1 \ran} \text{ for some } 2 \leq k <r,   
\\
\{ (m_{\lan r-1 \ran}, \la_1^{(r)}, r) \} & \text{ if } i=m_{\lan r-1 \ran} \text{ and } \la_1^{(r)} > \la_{m_r}^{(r)}, 
\\
\emptyset & \text{ otherwise}.
\end{cases}
\end{align*}
We also have 
\begin{align*}
&q^i c_{\wh{\bQ}_{[\bm]}}((m_{\lan k-1 \ran}, \la_1^{(k)}, k)) 
	= q^i q^{ 2 (\la_1^{(k)} - m_{\lan k-1 \ran})} \wh{Q}_{k-1} = q^{2 \la_1^{(k)} + i}  Q_{k-1}, 
\\
& S_i(z) = \b_{i,-1} + \b_{i,0} z = - q^{-\la_1^{(k)} -i} Q_{k-1}^{-1} + q^{\la_1^{(k)}} z
	= - q^{-\la_1^{(k)} -i} Q_{k-1}^{-1} ( 1 - q^{2 \la_1^{(k)} +i} Q_{k-1} z) 
\end{align*}
if $i= m_{\lan k-1 \ran} $. 
(Note that $ \xi^{-1} (i) = (m_{k-1}, k-1)$ if $i=m_{\lan k-1 \ran}$.)  
Thus, we see that $\prod_{x \in [\bT^{\sharp}]_{-\a_i}} (1 - q^i c_{\wh{\bQ}_{[\bm]}} (x) z)$ 
divides $S_i(z)$ for each $1 \leq i \leq m-1$. 
This implies $\wh{\bT}^{\sharp} \in \fSing_m^{\Bb_{[\bm]}^{\Bla}} (\wh{\Bla}_{[\bm]})$. 
Then, we have that 
 $\D_{\wh{\bQ}_{[\bm]}}^{\Bb^{\Bla}_{[\bm]}} (\wh{\Bla}_{[\bm]}) 
 \subset L_{\Bb_{[\bm]}^{\Bla}} \otimes \D_{\wh{\bQ}_{[\bm]}}(\wh{\Bla}_{[\bm]} ; \leq \wh{\bT}^{\sharp})$ 
by Theorem \ref{Thm simple LBb Dla}. 
By comparing the dimensions using (\roi) and (\roii), 
we obtain (\roiii). 
\end{proof}
\end{prop} 

%%%%%%%

%%%%%%%%%%%%%%%%%%%%%%%%%%%%%%%%%%%%%%%%%%%%%%%%%%%%%%%%%%%%%%%

%%%%%%%%%%%%%%%%%%%%%%%%%%%%%%%%%%%%%%%%%%%%%%%%%%%%%%%%%%%%%%%

%%%%%%%%%%%%%%%%%%%%%%%%%%%%%%%%%%%%%%%%%%%%%%%%%%%%%%%%%%%%%%%

\end{document}